\documentclass[11pt,twoside]{article}
\usepackage{amsmath,amsthm,amssymb,amsfonts}
\usepackage{tikz-cd}
\usepackage{xcolor}
\usepackage{enumitem}
\usepackage{hyphenat}
\usepackage[hidelinks]{hyperref}
\usepackage[all,cmtip]{xy}
\usepackage{tocloft}
\usepackage{etoolbox}

\usepackage{mathtools}
\usepackage[margin=1.25in, top=1in]{geometry}
\usepackage{mathrsfs}
\usepackage{fancyhdr}

\newcommand{\shorttitle}{A Non-graded Koszul Duality and Its Applications}
\newtheorem{conjecture}{Conjecture}
\newtheorem{theorem}{Theorem}[section]
\newtheorem{lemma}[theorem]{Lemma}
\newtheorem{corollary}[theorem]{Corollary}
\newtheorem{proposition}[theorem]{Proposition}
\theoremstyle{definition}
\newtheorem{definition}[theorem]{Definition}

\theoremstyle{remark}
\newtheorem{remark}[theorem]{Remark}

\begin{document}

\title{A Non-graded Koszul Duality and Its Applications}
\author{\scshape A. M. Bouhada}
\date{Dedicated to the memory of I.\,M.~Gel'fand.}
\maketitle

\begin{abstract}
Let \(\Lambda\) be a finite-dimensional Koszul algebra with Koszul dual
\(\Lambda^!\). We establish derived Koszul dualities at the level of bounded
derived categories, both in the graded setting
\(\mathsf{D}^{b}(\Lambda\textup{-gmod})\) and in the ungraded setting
\(\mathsf{D}^{b}(\Lambda\textup{-mod})\).

We first prove a graded derived Koszul duality for every finite-dimensional
Koszul algebra, with no Noetherian or coherence assumptions on the Koszul
dual. We then show that the bounded derived category
\(\mathsf{D}^{b}(\Lambda\textup{-mod})\) can be reconstructed from the graded
theory as the triangulated hull of an orbit category. This yields a genuinely
non-graded derived Koszul duality.

We further establish singular and dg refinements of these dualities. For
Iwanaga--Gorenstein Koszul algebras, this gives a stable Koszul duality for
graded Gorenstein-projective modules and their ungraded counterparts, providing
a non-graded form of the Bernstein--Gel'fand--Gel'fand correspondence.

As applications, we obtain new descriptions of the bounded derived categories
\(\mathsf{D}^{b}(\mathcal{O}_{\lambda})\) for all integral blocks of category
\(\mathcal{O}\), including singular blocks, which are non-graded counterparts of
the Koszul dualities of Beilinson, Ginzburg, and Soergel. We also establish analogous
dualities for certain categories of perverse sheaves arising in geometric
representation theory. Finally, we formulate conjectural descriptions of
bounded derived and singularity categories of finite-dimensional graded
algebras in terms of dg orbit categories.
\end{abstract}
\medskip


\medskip


\vspace{1em}
\tableofcontents

\vspace{2em}

\section{Introduction}

Koszul algebras, introduced by Priddy~\cite{38} following the seminal work of
Koszul on the cohomology of Lie algebras~\cite{32}, occupy a central position
in modern algebra, representation theory, and algebraic geometry. A defining
feature of this class of algebras is the existence of a highly structured
duality theory, namely \emph{Koszul duality}, which relates a Koszul algebra to
its Koszul dual through deep homological correspondences.

The origins of this theory go back to the work of
Bernstein--Gel'fand--Gel'fand, culminating in the celebrated
\emph{BGG correspondence}~\cite{8}, which identifies the stable category of
finite-dimensional graded modules over the exterior algebra with the bounded
derived category of coherent sheaves on projective space. This correspondence
has had far-reaching consequences; see, for instance,~\cite{17,18,19}. From another
geometric point of view, Koszul duality is closely tied to the structure of
derived categories of coherent sheaves. Beilinson~\cite{5} showed that the
bounded derived category of coherent sheaves on projective space is equivalent
to the bounded derived category of modules over a finite-dimensional Koszul
algebra, while Kapranov~\cite{25} extended this perspective to Grassmannians
and other homogeneous spaces.

A major conceptual advance was achieved by Beilinson, Ginzburg, and
Soergel~\cite{6}, who formulated Koszul duality in the framework of derived
categories. In particular, for a finite-dimensional Koszul algebra
\(\Lambda\) whose Koszul dual \(\Lambda^{!}\) is left Noetherian, they
obtained a triangulated equivalence
\(\mathsf{D}^{b}(\Lambda^{!}\textup{-}\mathrm{Fg}^{\mathbb{Z}})\simeq\mathsf{D}^{b}(\Lambda\textup{-gmod})\).
Subsequently, Martínez-Villa and Saorín~\cite{33} weakened these assumptions
by replacing Noetherianity with coherence.

Despite these developments, two fundamental problems remain open in general.
First, there is, to the best of our knowledge, no non-graded derived Koszul
duality describing the bounded derived category
\(\mathsf{D}^{b}(\Lambda\textup{-mod})\) of a finite-dimensional Koszul algebra
in terms of its Koszul dual. Second, even in the graded setting, the Koszul
duality of~\cite[Theorem~2.12.1]{6}, together
with~\cite[Lemma~2.12.7]{6}, identifies
\(\mathsf{D}^{b}(\Lambda\textup{-gmod})\) with a subcategory of complexes of
graded \(\Lambda^{!}\)-modules, but this subcategory is
known to be a bounded derived category of \(\Lambda^{!}\)-modules only when
\(\Lambda^{!}\) is left Noetherian~\cite[Lemma~2.12.8 and Theorem~2.12.6]{6}.
When \(\Lambda^{!}\) is infinite-dimensional and fails to satisfy finiteness
conditions such as Noetherianity or coherence, it is not known in general
whether this subcategory admits such a description.

The purpose of the present paper is to address both of these problems. For
every finite-dimensional Koszul algebra \(\Lambda\), we give an intrinsic
description of this subcategory, without imposing any finiteness conditions on
the Koszul dual algebra \(\Lambda^{!}\), and we use the resulting graded theory
to construct a genuinely non-graded derived Koszul duality. The key idea is to
reconstruct the bounded derived category of finite-dimensional modules from
graded Koszul duality by means of differential graded orbit constructions and
their triangulated hulls.

\noindent\textbf{Main results.}
Let \(\Lambda\) be a finite-dimensional Koszul algebra.

Our first main result identifies the subcategory of~\cite[Theorem~2.12.1]{6}
corresponding to \(\mathsf{D}^{b}(\Lambda\textup{-gmod})\): the Koszul
duality restricts to a triangulated equivalence
\[
\mathsf{D}^b\bigl(\Lambda^{!}\textup{-}\mathrm{Cop}^{\mathbb{Z}}\bigr)
\;\xrightarrow{\ \sim\ }\;
\mathsf{D}^b\bigl(\Lambda\textup{-gmod}\bigr),
\]
where \(\Lambda^{!}\textup{-}\mathrm{Cop}^{\mathbb{Z}}\) denotes the category of
locally finite-dimensional graded coperfect \(\Lambda^{!}\)-modules. In
particular, this identifies \(\mathsf{D}^{b}(\Lambda\textup{-gmod})\) entirely
in terms of the Koszul dual algebra. This equivalence admits a singular
refinement, yielding a graded singular Koszul duality
\[
\mathsf{D}^{b}\bigl(\Lambda^{!}\textup{-}\mathrm{Cop}^{\mathbb{Z}}/\Lambda^{!}\textup{-gmod}\bigr)
\;\xrightarrow{\ \sim\ }\;
\mathsf{D}_{\mathrm{sg}}\bigl(\Lambda\textup{-gmod}\bigr),
\]
where \(\Lambda^{!}\textup{-}\mathrm{Cop}^{\mathbb{Z}}/\Lambda^{!}\textup{-gmod}\)
is an exact quotient category which, when \(\Lambda^{!}\) is graded
cocoherent, is dual to the tails category of \((\Lambda^{!})^{\mathrm{op}}\).
When \(\Lambda\) is Iwanaga--Gorenstein, this further induces a graded
Bernstein--Gel'fand--Gel'fand type correspondence
\(\mathsf{D}^{b}(\Lambda^{!}\textup{-}\mathrm{Cop}^{\mathbb{Z}}/\Lambda^{!}\textup{-gmod})\simeq\Lambda\textup{-}\underline{\mathrm{Gproj}}^{\mathbb{Z}}\).
Both equivalences lift to quasi-equivalences of the corresponding dg
categories.

Our second main result provides a non-graded derived Koszul duality: we
construct a triangulated equivalence
\[
\mathrm{H}^{0}\Bigl(
\operatorname{pretr}\bigl(
\mathsf{D}^{b}_{\mathrm{dg}}\bigl(\Lambda^{!}\textup{-}\mathrm{Cop}^{\mathbb{Z}}\bigr)
/\langle 1\rangle[-1]
\bigr)
\Bigr)
\;\xrightarrow{\ \sim\ }\;
\mathsf{D}^{b}\bigl(\Lambda\textup{-mod}\bigr).
\]
Thus the bounded derived category of finite-dimensional \(\Lambda\)-modules is
recovered functorially from the graded Koszul dual side. We also establish a
singular analogue,
\[
\mathrm{H}^{0}\Bigl(
\operatorname{pretr}\bigl(
\mathsf{D}^{b}_{\mathrm{dg}}\bigl(\Lambda^{!}\textup{-}\mathrm{Cop}^{\mathbb{Z}}/\Lambda^{!}\textup{-gmod}\bigr)
/\langle 1\rangle[-1]
\bigr)
\Bigr)
\;\xrightarrow{\ \sim\ }\;
\mathsf{D}_{\mathrm{sg}}\bigl(\Lambda\textup{-mod}\bigr),
\]
and, when \(\Lambda\) is Iwanaga--Gorenstein, this yields a non-graded analogue
of the Bernstein--Gel'fand--Gel'fand correspondence.

These dualities have several applications. First, they provide new
descriptions of the bounded derived categories
\(\mathsf{D}^{b}(\mathcal{O}_{\lambda})\) for all integral blocks of the
Bernstein--Gel'fand--Gel'fand category \(\mathcal{O}\), including singular
blocks, thereby extending the classical picture to the non-graded setting.
Second, we obtain analogous derived Koszul dualities for certain categories of
perverse sheaves arising in geometric representation theory. Finally,
motivated by the dg orbit descriptions established in this paper, we formulate conjectural descriptions
of bounded derived categories and singularity categories of
finite-dimensional graded algebras in terms of dg orbit categories and their
triangulated hulls, and we show that the results of this paper confirm them
for Koszul algebras.

\noindent\textbf{Organization of the paper.}
Section~2 fixes notation and reviews the necessary background. Section~3
develops graded derived Koszul duality and its singular and dg refinements.
Section~4 treats the non-graded setting.
Section~5 presents applications to category \(\mathcal{O}\) and perverse
sheaves. Section~6 studies explicit classes of algebras, and Section~7
formulates conjectural extensions.
\section{Preliminaries and background}\label{sec:prelim}

In this section we fix notation and recall the background used throughout
the paper. Standard material is only recalled briefly, with references.

Let \(k\) be a field. We consider \(\mathbb{Z}_{\ge 0}\)-graded
\(k\)-algebras of the form \(\Lambda = kQ/I\), where \(Q\) is a finite quiver and \(I\subseteq kQ\) is a homogeneous ideal
contained in \((kQ)^{2}\), the ideal generated by all paths of length at
least \(2\). The grading on \(\Lambda\) is induced by the path-length grading
on the path algebra \(kQ\).

The ideal \(I\) is not assumed to be admissible; in particular, the algebra
\(\Lambda\) may be either finite-dimensional or locally finite-dimensional,
in the sense that each homogeneous component \(\Lambda_n\) is
finite-dimensional. This level of generality is required for our purposes:
although in most applications \(\Lambda\) will be finite-dimensional and
Koszul, its Koszul dual \(\Lambda^{!}\) need not be finite-dimensional.

We view such algebras as small \(k\)-linear graded categories whose
objects are the vertices of \(Q\): for \(i,j\in Q_{0}\), the morphism space
\(\Lambda(i,j)=e_{j}\Lambda e_{i}\) is spanned by the paths from \(i\) to \(j\).

\subsection{Modules and their categories}

Recall that a \emph{quiver} \(Q=(Q_0,Q_1,s,t)\) consists of a finite set \(Q_0\)
of vertices, a finite set \(Q_1\) of arrows, and two maps
\(s,t\colon Q_1\to Q_0\) assigning to each arrow \(\alpha\) its source
\(s(\alpha)\) and its target \(t(\alpha)\). The \emph{path algebra} \(kQ\) is
the \(k\)-vector space with basis the set of paths in \(Q\), including a
trivial path \(e_x\) of length zero for each vertex \(x\in Q_0\); the product
of two paths is their concatenation when it is defined, and zero otherwise.
The algebra \(kQ\) is graded by path length.

We say that \(Q\) is \emph{well directed} if its set of vertices admits a
partition \(Q_0=Q_0^{(1)}\sqcup Q_0^{(2)}\sqcup\cdots\sqcup Q_0^{(r)}\) into
nonempty subsets, called \emph{levels}, such that every arrow of \(Q\) starts
in some level \(Q_0^{(i)}\) and ends in the next level \(Q_0^{(i+1)}\). Each
level may contain several vertices. In particular, if \(y\in Q_0^{(j)}\) and
\(x\in Q_0^{(i)}\), then every path from \(y\) to \(x\) has length \(i-j\). In
the terminology of~\cite{2,3}, these are precisely the finite gradable quivers.

For instance, every orientation of a tree is well directed: fixing the level
of one vertex and requiring each arrow to raise the level by one determines the
levels up to a common shift, since the underlying graph has no cycles. In
particular, every Dynkin quiver is well directed, whatever its orientation. For
the quiver of type \(\mathbb{E}_6\) on the left below, the levels are
\(\{6\}\), \(\{1,3,5\}\) and \(\{2,4\}\), and redrawing the quiver level by
level gives the picture on the right:
\[
\begin{tikzcd}[row sep=small, column sep=small]
 & & 6 \arrow[d] & & \\
1 \arrow[r] & 2 & 3 \arrow[l] \arrow[r] & 4 & 5 \arrow[l]
\end{tikzcd}
\qquad\qquad
\begin{tikzcd}[row sep=tiny, column sep=normal]
 & 1 \arrow[dr] & \\
 & & 2 \\
6 \arrow[r] & 3 \arrow[ur] \arrow[dr] & \\
 & & 4 \\
 & 5 \arrow[ur] &
\end{tikzcd}
\]
Similarly, the square on the left below is well directed, with levels
\(\{a\}\), \(\{b,c\}\), \(\{d\}\). In contrast, the quiver on the right, which
is the quiver of the radical square zero algebra of~\cite[p.~157]{24}, is not
well directed, although it has no oriented cycles: there are two paths from
\(1\) to \(4\), of lengths \(3\) and \(2\), whereas in a well directed quiver
all paths between two given vertices have the same length. Likewise, a quiver
with an oriented cycle is never well directed.
\[
\begin{tikzcd}[row sep=small, column sep=small]
 & b \arrow[rd] & \\
a \arrow[ru] \arrow[rd] & & d\\
 & c \arrow[ru] &
\end{tikzcd}
\qquad\qquad
\begin{tikzcd}[row sep=small, column sep=small]
 & 2 \arrow[r] & 3 \arrow[rd] & \\
1 \arrow[ru] \arrow[rd] & & & 4\\
 & 5 \arrow[rru] & &
\end{tikzcd}
\]
Typical examples are also the quivers of the Beilinson algebras
(Remark~\ref{rem:beilinson}). A well directed quiver has no oriented cycles;
consequently, for every ideal \(I\) of \(kQ\), the algebra \(kQ/I\) is
finite-dimensional and of finite global dimension.

Let now \(\Lambda=kQ/I\) be as at the beginning of this section, and denote
the images in \(\Lambda\) of the trivial paths again by \(e_x\),
\(x\in Q_0\). By our assumptions on \(I\), the set
\(\{e_x\mid x\in Q_0\}\) is a complete set of pairwise orthogonal primitive
idempotents of \(\Lambda\).

As usual, we identify (left) \(\Lambda\)-modules with \(k\)-linear functors
\(M\colon\Lambda\to\mathrm{Vect}_k\), and \(\mathbb{Z}\)-graded
\(\Lambda\)-modules with \(k\)-linear functors
\(M\colon\Lambda\to\mathrm{GrVect}_k\) sending homogeneous morphisms of
degree \(d\) to homogeneous maps of degree \(d\). For a vertex \(x\in Q_0\) we
write \(M(x)=e_xM\), and \(M_n(x)\) for its homogeneous component of degree
\(n\). Thus each arrow \(\alpha\colon x\to y\) induces linear maps
\(M(\alpha)\colon M_n(x)\to M_{n+1}(y)\), \(n\in\mathbb{Z}\). Morphisms of
graded modules are homogeneous of degree \(0\).

We denote by \(\Lambda\textup{-}\mathrm{GMOD}\) the category of all
\(\mathbb{Z}\)-graded \(\Lambda\)-modules, by \(\Lambda\textup{-}\mathrm{GMod}\)
the full subcategory of locally finite-dimensional graded modules, that is,
those \(M\) with \(M_n(x)\) finite-dimensional for all \(x\in Q_0\) and
\(n\in\mathbb{Z}\), and by \(\Lambda\textup{-}\mathrm{gmod}\) the full
subcategory of finite-dimensional graded modules. Finally,
\(\Lambda\textup{-}\mathrm{mod}\) denotes the category of finite-dimensional
\(\Lambda\)-modules.
A graded module \(M\) is \emph{right bounded} if \(M_n=0\) for \(n\gg0\), and
\emph{left bounded} if \(M_n=0\) for \(n\ll0\). We denote by
\(\Lambda\textup{-}\mathrm{GMod}^{-}\) (resp.\
\(\Lambda\textup{-}\mathrm{GMod}^{+}\)) the full subcategory of
\(\Lambda\textup{-}\mathrm{GMod}\) consisting of the right (resp.\ left)
bounded modules. An object of \(\Lambda\textup{-}\mathrm{GMod}\) is
finite-dimensional if and only if it is both left and right bounded.

For \(i\in\mathbb{Z}\), the \emph{grading shift} \(\langle i\rangle\) is
defined by \(M\langle i\rangle_n(x)=M_{n+i}(x)\). Note that the grading shift
used in~\cite{6} is opposite to the convention adopted here.

For a graded vector space \(V=\bigoplus_{i\in\mathbb{Z}}V_i\), we denote by
\(\mathbb{D}V:=\bigoplus_{i\in\mathbb{Z}}\operatorname{Hom}_k(V_{-i},k)\) its
\emph{graded dual}, and we apply \(\mathbb{D}\) to graded modules vertexwise.
Then \(M\cong\mathbb{D}\mathbb{D}M\) for every locally finite-dimensional graded
module \(M\), and \(\mathbb{D}\) induces a duality between
\(\Lambda\textup{-}\mathrm{GMod}\) and \(\Lambda^{\mathrm{op}}\textup{-}\mathrm{GMod}\)
exchanging right bounded and left bounded modules.

For \(x\in Q_0\) and \(r\in\mathbb{Z}\), we set
\(P_{x}\langle r\rangle:=\Lambda(x,-)\langle r\rangle\) and
\(I_{x}\langle r\rangle:=\mathbb{D}\bigl(\Lambda^{\mathrm{op}}(x,-)\bigr)\langle r\rangle\),
and denote by \(S_{x}\langle r\rangle\) the corresponding shift of the simple
top of \(P_x\). These modules are indecomposable.

We denote by \(\Lambda\textup{-}\mathrm{Proj}^{\mathbb{Z}}\) (resp.\
\(\Lambda\textup{-}\mathrm{Inj}^{\mathbb{Z}}\)) the full subcategory of
\(\Lambda\textup{-}\mathrm{GMod}\) consisting of the modules isomorphic to
\emph{finite} direct sums of modules \(P_x\langle r\rangle\) (resp.\
\(I_x\langle r\rangle\)), with \(x\in Q_0\) and \(r\in\mathbb{Z}\). If \(\Lambda\)
is finite-dimensional, then each \(P_x\langle r\rangle\) and each
\(I_x\langle r\rangle\) is finite-dimensional, hence so are all objects of
\(\Lambda\textup{-}\mathrm{Proj}^{\mathbb{Z}}\) and
\(\Lambda\textup{-}\mathrm{Inj}^{\mathbb{Z}}\). In this case we write
\(\Lambda\textup{-}\mathrm{proj}^{\mathbb{Z}}\) and
\(\Lambda\textup{-}\mathrm{inj}^{\mathbb{Z}}\) for these categories; they are
the categories of projective and of injective objects of
\(\Lambda\textup{-}\mathrm{gmod}\), respectively. Similarly, we write
\(\Lambda\textup{-}\mathrm{proj}\) and \(\Lambda\textup{-}\mathrm{inj}\) for
the categories of finite-dimensional projective and injective
\(\Lambda\)-modules.

Forgetting the grading defines the \emph{forgetful functor}
\(F\colon \Lambda\textup{-}\mathrm{gmod}\to \Lambda\textup{-}\mathrm{mod}\).
A module \(M\in\Lambda\textup{-}\mathrm{mod}\) is \emph{gradable} if
\(M\cong F(N)\) for some \(N\in\Lambda\textup{-}\mathrm{gmod}\).

\begin{proposition}\label{prop:forget}
Let \(M\) and \(N\) be finite-dimensional graded \(\Lambda\)-modules. Then
\(F\) induces a natural isomorphism of vector spaces
\[
\bigoplus_{r\in\mathbb{Z}}
\operatorname{Hom}_{\Lambda\textup{-}\mathrm{gmod}}\!\bigl(M,N\langle r\rangle\bigr)
\;\xrightarrow{\ \sim\ }\;
\operatorname{Hom}_{\Lambda\textup{-}\mathrm{mod}}\!\bigl(M,N\bigr).
\]
In other words, since \(F\circ\langle1\rangle=F\), the functor \(F\) is a
\(\langle\langle1\rangle\rangle\)-Galois precovering in the sense of
Definition~\ref{def:galois} below, which applies verbatim to additive functors.
\end{proposition}

\subsection{Koszul algebras and their duals}

For background on Koszul algebras we refer to~\cite{6,11,33,34}. Let
\(\Lambda=kQ/I\), and denote by \(kQ_2\) the subspace of \(kQ\)
spanned by the paths of length \(2\). The algebra \(\Lambda\) is called
\emph{quadratic} if \(I\) is generated by \(I_2:=I\cap kQ_2\). Consider the
nondegenerate bilinear pairing \(kQ_2\times kQ^{\mathrm{op}}_2\to k\) for which
the paths of length \(2\) in \(Q\) and the corresponding opposite paths in
\(Q^{\mathrm{op}}\) form dual bases, and let
\(I_2^{\perp}\subseteq kQ^{\mathrm{op}}_2\) be the orthogonal complement of
\(I_2\). The \emph{quadratic dual} of \(\Lambda\) is the graded algebra
\(\Lambda^{!}:=kQ^{\mathrm{op}}/\langle I_2^{\perp}\rangle\). Its ideal of
relations is generated in degree \(2\), so \(\Lambda^{!}\) again satisfies our
standing assumptions; note, however, that \(\Lambda^{!}\) is in general only
locally finite-dimensional, even if \(\Lambda\) is finite-dimensional.

\begin{definition}\label{def:koszul}
The algebra \(\Lambda=\bigoplus_{i\ge 0}\Lambda_i\) is \emph{Koszul} if
\(\Lambda_0\) admits a \emph{linear} graded projective resolution
\(\cdots \to P^{-1} \to P^0 \to \Lambda_0 \to 0\), that is, one in which
\(P^{-n}\) is generated in degree \(n\) for every \(n\ge0\). Equivalently,
every simple graded module \(S_x\), \(x\in Q_0\), admits a linear graded
projective resolution.
\end{definition}

A Koszul algebra is quadratic. If \(\Lambda\) is Koszul, then so is
\(\Lambda^{!}\), one has \((\Lambda^{!})^{!}\cong\Lambda\), and there are
isomorphisms of graded vector spaces
\[
\Lambda^{!}\cong\bigoplus_{n\ge 0}\operatorname{Ext}^n_{\Lambda}(\Lambda_0,\Lambda_0),
\qquad
\Lambda\cong\bigoplus_{n\ge 0}\operatorname{Ext}^n_{\Lambda^{!}}(\Lambda^{!}_0,\Lambda^{!}_0),
\tag{$\ast$}
\]
under which \(\Lambda^{!}_n\) and \(\Lambda_n\) correspond to the \(n\)-th
extension groups; see~\cite[Theorems~2.10.1 and~2.10.2]{6}.

We shall use the following standard characterization of Koszulity in terms of
extensions between simple modules.

\begin{proposition}\label{prop:koszul-ext}
The algebra \(\Lambda\) is Koszul if and only if
\[
\operatorname{Ext}^{n}_{\Lambda\textup{-gmod}}\bigl(S_{x},\,S_{y}\langle -i\rangle\bigr)=0
\quad\text{for all }x,y\in Q_{0},\ n\ge0\text{ and }i\neq n,
\]
where \(\operatorname{Ext}\) is computed in \(\Lambda\textup{-GMod}\); for
finite-dimensional \(\Lambda\) this agrees with \(\operatorname{Ext}\) in
\(\Lambda\textup{-gmod}\).
\end{proposition}

\begin{proof}
Let \(x,y\in Q_0\), and let \(P^{\bullet}\to S_x\) be a minimal graded
projective resolution. By minimality, the differentials of
\(\operatorname{Hom}_{\Lambda\textup{-gmod}}(P^{\bullet},S_y\langle -i\rangle)\)
vanish~\cite[Cor.~2.5.4]{7}, so
\(\operatorname{Ext}^n_{\Lambda\textup{-gmod}}(S_x,S_y\langle -i\rangle)\cong
\operatorname{Hom}_{\Lambda\textup{-gmod}}(P^{-n},S_y\langle -i\rangle)\).
Since \(S_y\langle -i\rangle\) is concentrated in degree \(i\), these spaces
vanish for all \(y\) and all \(i\neq n\) if and only if \(P^{-n}\) is generated
in degree \(n\). Finally, a linear resolution is minimal, because the image of
each differential lies in \(\Lambda_{\ge1}P^{-n+1}=\operatorname{rad}P^{-n+1}\);
hence \(S_x\) admits a linear resolution if and only if its minimal resolution
is linear.
\end{proof}

Recall that if \(A=\bigoplus_{i\ge0}A_i\) is a locally finite-dimensional
graded algebra with semisimple degree-zero part, then
\(\operatorname{gldim}A=\sup\{n\ge0\mid\operatorname{Ext}^n_A(A_0,A_0)\neq0\}\);
see~\cite{16}. Together with \((\ast)\), this yields the following standard
fact.

\begin{theorem}\label{thm:fd-gldim}
Assume that \(\Lambda\) is Koszul. Then \(\Lambda\) is finite-dimensional if
and only if \(\Lambda^{!}\) has finite global dimension.
\end{theorem}

\begin{proof}
By the second isomorphism in \((\ast)\), we have
\(\Lambda_n\cong\operatorname{Ext}^n_{\Lambda^{!}}(\Lambda^{!}_0,\Lambda^{!}_0)\)
for every \(n\ge0\). Since each \(\Lambda_n\) is finite-dimensional, \(\Lambda\)
is finite-dimensional if and only if \(\Lambda_n=0\) for \(n\gg0\), that is,
if and only if \(\operatorname{Ext}^n_{\Lambda^{!}}(\Lambda^{!}_0,\Lambda^{!}_0)=0\)
for \(n\gg0\). By the formula for the global dimension recalled above, this
holds if and only if \(\operatorname{gldim}\Lambda^{!}<\infty\).
\end{proof}

\subsection{The bounded derived category of an exact category}

We use exact categories in the sense of Quillen and their bounded derived
categories \(\mathsf{D}^{b}(\mathcal{B})=\mathsf{K}^{b}(\mathcal{B})/\mathsf{Ac}^{b}(\mathcal{B})\);
see~\cite{13,27}. All exact categories in this paper are full additive
subcategories \(\mathcal{B}\) of an abelian category \(\mathcal{C}\) satisfying
the following \emph{two-out-of-three property}: for every short exact sequence
\(0\to X\to Y\to Z\to0\) in \(\mathcal{C}\), if two of \(X,Y,Z\) lie in
\(\mathcal{B}\), so does the third. In particular, \(\mathcal{B}\) is closed
under extensions in \(\mathcal{C}\), and we endow it with the exact structure
whose conflations are the short exact sequences of \(\mathcal{C}\) with all
terms in \(\mathcal{B}\).

\begin{lemma}\label{lem:2of3}
Every exact category \(\mathcal{B}\subseteq\mathcal{C}\) is weakly
idempotent complete.
\end{lemma}

\begin{proof}
If \(r\colon Y\to Z\) is a retraction in \(\mathcal{B}\), then
\(0\to\ker r\to Y\to Z\to0\) is exact in \(\mathcal{C}\), so
\(\ker r\in\mathcal{B}\) by the two-out-of-three property.
\end{proof}

Consequently, \(\mathsf{Ac}^{b}(\mathcal{B})\) is a thick subcategory of
\(\mathsf{K}^{b}(\mathcal{B})\)~\cite[\S10]{13}. Moreover, by the
two-out-of-three property, a bounded complex in \(\mathcal{B}\) is acyclic if
and only if it is exact in \(\mathcal{C}\). Hence a morphism of
\(\mathsf{K}^{b}(\mathcal{B})\) becomes an isomorphism in
\(\mathsf{D}^{b}(\mathcal{B})\) if and only if it is a quasi-isomorphism of
complexes in \(\mathcal{C}\).

If \(\mathcal{B}\) has enough projectives, respectively injectives, there are
the usual triangulated equivalences
\(\mathsf{K}^{-,b}(\operatorname{Proj}\mathcal{B})\simeq\mathsf{D}^{b}(\mathcal{B})\),
respectively
\(\mathsf{K}^{+,b}(\operatorname{Inj}\mathcal{B})\simeq\mathsf{D}^{b}(\mathcal{B})\),
which restrict to \(\mathsf{K}^{b}(\operatorname{Proj}\mathcal{B})\),
respectively \(\mathsf{K}^{b}(\operatorname{Inj}\mathcal{B})\), if
\(\mathcal{B}\) has finite homological dimension; here the \emph{homological
dimension} of \(\mathcal{B}\) is the smallest \(d\ge0\) such that
\(\operatorname{Ext}^{i}_{\mathcal{B}}(X,Y):=\operatorname{Hom}_{\mathsf{D}^b(\mathcal{B})}(X,Y[i])=0\)
for all \(X,Y\in\mathcal{B}\) and \(i>d\), if it exists. An object
\(I\in\mathcal{B}\) is \emph{strongly injective} if it is injective both in
\(\mathcal{B}\) and in \(\mathcal{C}\).

\begin{proposition}\label{prop:strong-inj-ff}
Let \(\mathcal{B}\subseteq\mathcal{C}\) be an exact category with
enough injectives, and suppose that every injective object of \(\mathcal{B}\)
is strongly injective. Then the canonical functor
\(\mathsf{D}^{b}(\mathcal{B})\to\mathsf{D}^{b}(\mathcal{C})\), induced by the
inclusion \(\mathcal{B}\hookrightarrow\mathcal{C}\), is fully faithful. 
\end{proposition}

\begin{proof}
The full faithfulness follows from~\cite[Theorem~12.1]{27}.
\end{proof}

A \emph{triple} \((\mathcal{A},\mathcal{B},\mathcal{C})\) is a chain of full
inclusions \(\mathcal{A}\subseteq\mathcal{B}\subseteq\mathcal{C}\), where
\(\mathcal{C}\) is abelian, \(\mathcal{B}\) is an exact category in
\(\mathcal{C}\) as above which has enough injectives, all of them strongly
injective, and \(\mathcal{A}\) is a Serre subcategory of \(\mathcal{C}\), that
is, closed under subobjects, quotients and extensions.

Let \(\Sigma\) be the class of morphisms \(s\) of \(\mathcal{C}\) with
\(\ker s,\operatorname{coker}s\in\mathcal{A}\). The \emph{quotient category}
\(\mathcal{C}/\mathcal{A}\) is the localization of \(\mathcal{C}\) at
\(\Sigma\). Since \(\Sigma\) admits a calculus of fractions, its morphisms
\(X\to Y\) are represented by \emph{roofs}
\(X\xrightarrow{\;f\;}Z\xleftarrow{\;s\;}Y\) with \(s\in\Sigma\), written
\(s^{-1}f\); two roofs \(s_1^{-1}f_1\) and \(s_2^{-1}f_2\) are equal if there
are morphisms \(u_i\colon Z_i\to Z_3\) with \(u_1f_1=u_2f_2\) and
\(u_1s_1=u_2s_2\in\Sigma\). Roofs are composed using the Ore condition: given
\(s^{-1}f\colon X\to Y\) and \(t^{-1}g\colon Y\to W\), with \(f\colon X\to M\),
\(s\colon Y\to M\), \(g\colon Y\to N\) and \(t\colon W\to N\), one chooses
morphisms \(v\colon M\to L\) and \(u\colon N\to L\) with \(vs=ug\) and
\(u\in\Sigma\), and sets \((t^{-1}g)\circ(s^{-1}f)=(ut)^{-1}(vf)\);
see~\cite[Chap.~III]{20} and~\cite{22}. The quotient functor
\(Q\colon\mathcal{C}\to\mathcal{C}/\mathcal{A}\) is exact, \(Q(X)=0\) if and
only if \(X\in\mathcal{A}\), \(\mathcal{C}/\mathcal{A}\) is abelian, and every
short exact sequence in \(\mathcal{C}/\mathcal{A}\) is isomorphic to the image
under \(Q\) of a short exact sequence in \(\mathcal{C}\)~\cite[Chap.~III]{20}.

We denote by \(\mathcal{B}/\mathcal{A}\) the full subcategory of
\(\mathcal{C}/\mathcal{A}\) whose objects are those of \(\mathcal{B}\). If
\(s\colon Y\to Z\) lies in \(\Sigma\), then \(Y\in\mathcal{B}\) if and only if
\(Z\in\mathcal{B}\), by the two-out-of-three property applied to
\(0\to\ker s\to Y\to\operatorname{Im}s\to0\) and
\(0\to\operatorname{Im}s\to Z\to\operatorname{coker}s\to0\). Thus all roofs
between objects of \(\mathcal{B}\) lie in \(\mathcal{B}\), and
\(\mathcal{B}/\mathcal{A}\) is the localization of \(\mathcal{B}\) at
\(\Sigma\cap\mathcal{B}\). Similarly, every short exact sequence in
\(\mathcal{C}/\mathcal{A}\) whose end terms lie in \(\mathcal{B}/\mathcal{A}\)
is isomorphic to the image of a short exact sequence of \(\mathcal{C}\) with
all its terms in \(\mathcal{B}\): starting from a short exact sequence of
\(\mathcal{C}\) representing it, one replaces its end terms by objects of
\(\mathcal{B}\) by a pushout and a pullback along morphisms of \(\Sigma\), and
the middle term then lies in \(\mathcal{B}\) by the two-out-of-three property.
Hence \(\mathcal{B}/\mathcal{A}\) is closed under extensions in
\(\mathcal{C}/\mathcal{A}\), and we regard it as an exact category with the
induced exact structure. In particular, \(Q\) restricts to an exact functor
\(Q\colon\mathcal{B}\to\mathcal{B}/\mathcal{A}\).

\begin{proposition}\label{prop:quot-wic}
Let \((\mathcal{A},\mathcal{B},\mathcal{C})\) be a triple. Then
\(\mathcal{B}/\mathcal{A}\) is weakly idempotent complete.
\end{proposition}

\begin{proof}
Let \(f\colon Y\to X\) be a retraction in \(\mathcal{B}/\mathcal{A}\),
represented by a roof \(Y\xrightarrow{g}Z\xleftarrow{s}X\) in \(\mathcal{B}\)
with \(s\in\Sigma\). Since \(Q(s)\) is an isomorphism, \(f\) has a kernel in
\(\mathcal{B}/\mathcal{A}\) if and only if \(Q(g)\) does. Since \(f\) is a
retraction, it is an epimorphism in \(\mathcal{C}/\mathcal{A}\), hence so is
\(Q(g)\), and therefore \(\operatorname{coker}g\in\mathcal{A}\). The
two-out-of-three property applied to
\(0\to\operatorname{Im}g\to Z\to\operatorname{coker}g\to0\) gives
\(\operatorname{Im}g\in\mathcal{B}\), and then applied to
\(0\to\ker g\to Y\to\operatorname{Im}g\to0\) gives \(\ker g\in\mathcal{B}\).
Since \(Q\) is exact, \(Q(\ker g)\) is a kernel of \(Q(g)\) in
\(\mathcal{C}/\mathcal{A}\); as it lies in \(\mathcal{B}/\mathcal{A}\), it is
also a kernel of \(Q(g)\) in \(\mathcal{B}/\mathcal{A}\).
\end{proof}

The exact functor \(Q\colon\mathcal{B}\to\mathcal{B}/\mathcal{A}\) induces a
triangulated functor
\(Q_*\colon \mathsf{D}^b(\mathcal{B}) \to \mathsf{D}^b(\mathcal{B}/\mathcal{A})\)
which vanishes on the thick subcategory
\[
\mathsf{D}^b_{\mathcal{A}}(\mathcal{B}) := \bigl\{ X^\bullet \in \mathsf{D}^b(\mathcal{B}) \bigm| H^n(X^\bullet) \in \mathcal{A} \text{ for all } n \bigr\},
\]
where cohomology is computed in \(\mathcal{C}\). Indeed, as shown in the proof
of Theorem~\ref{thm:quot-derived} below, \(\mathsf{D}^b_{\mathcal{A}}(\mathcal{B})\)
is generated, as a triangulated category, by the objects of \(\mathcal{A}\),
and \(Q(X)=0\) for \(X\in\mathcal{A}\). The following result is used
for the graded singular Koszul duality; compare~\cite[Theorem~3.2]{35}.

\begin{theorem}\label{thm:quot-derived}
Let \((\mathcal{A},\mathcal{B},\mathcal{C})\) be a triple. Then \(Q_*\)
induces a triangulated equivalence
\[
\mathsf{D}^b(\mathcal{B}) \big/ \mathsf{D}^b_{\mathcal{A}}(\mathcal{B})
\;\xrightarrow{\;\sim\;}
\mathsf{D}^b(\mathcal{B}/\mathcal{A}).
\]
\end{theorem}

\begin{proof}
By Proposition~\ref{prop:strong-inj-ff}, the functor
\(\mathsf{D}^b(\mathcal{B})\to\mathsf{D}^b(\mathcal{C})\) is fully faithful.
We check that the condition of~\cite[Theorem~12.1]{27} also holds for
\(\mathcal{B}/\mathcal{A}\subseteq\mathcal{C}/\mathcal{A}\). Let
\(0\to Q(X)\to E\to E'\to0\) be a short exact sequence in
\(\mathcal{C}/\mathcal{A}\) with \(X\in\mathcal{B}\). As above, we may assume
that it is the image of a short exact sequence
\(0\to X_1\xrightarrow{i} E_1\to E_1'\to0\) of \(\mathcal{C}\) with
\(X_1\in\mathcal{B}\). Choose an inflation \(X_1\to I\) in \(\mathcal{B}\)
with \(I\) strongly injective. Since \(I\) is injective in \(\mathcal{C}\),
this morphism extends along \(i\) to a morphism \(E_1\to I\). Applying \(Q\),
we obtain a morphism \(E\to Q(I)\) whose composite with \(Q(X)\to E\) is an
inflation of \(\mathcal{B}/\mathcal{A}\). Hence the functor
\(\mathsf{D}^b(\mathcal{B}/\mathcal{A})\to\mathsf{D}^b(\mathcal{C}/\mathcal{A})\)
is fully faithful by~\cite[Theorem~12.1]{27}.

The essential image of \(\mathsf{D}^b_{\mathcal{A}}(\mathcal{B})\) in
\(\mathsf{D}^b(\mathcal{C})\) is a triangulated subcategory of
\(\mathsf{D}^b_{\mathcal{A}}(\mathcal{C})\), since cohomology is computed in
\(\mathcal{C}\) in both cases, and it contains the objects of
\(\mathcal{A}\subseteq\mathcal{B}\). Since
\(\mathsf{D}^b_{\mathcal{A}}(\mathcal{C})\) is generated, as a triangulated
category, by the objects of \(\mathcal{A}\) (use truncations), the functor
\(\mathsf{D}^b(\mathcal{B})\to\mathsf{D}^b(\mathcal{C})\) restricts to an
equivalence
\(\mathsf{D}^b_{\mathcal{A}}(\mathcal{B})\simeq\mathsf{D}^b_{\mathcal{A}}(\mathcal{C})\).
Consequently, if \(Y\in\mathsf{D}^b(\mathcal{B})\) and \(s\colon Y\to Y'\) is a
morphism in \(\mathsf{D}^b(\mathcal{C})\) with cone in
\(\mathsf{D}^b_{\mathcal{A}}(\mathcal{C})\), then the cone of \(s\) lies in the
essential image of \(\mathsf{D}^b_{\mathcal{A}}(\mathcal{B})\), and \(Y'\) lies
in the essential image of \(\mathsf{D}^b(\mathcal{B})\), which is a
triangulated subcategory. Hence the roofs computing morphisms in
\(\mathsf{D}^b(\mathcal{C})/\mathsf{D}^b_{\mathcal{A}}(\mathcal{C})\) between
objects of \(\mathsf{D}^b(\mathcal{B})\) can be taken in
\(\mathsf{D}^b(\mathcal{B})\), and the induced functor
\[
\mathsf{D}^b(\mathcal{B})/\mathsf{D}^b_{\mathcal{A}}(\mathcal{B})\longrightarrow
\mathsf{D}^b(\mathcal{C})/\mathsf{D}^b_{\mathcal{A}}(\mathcal{C})\simeq
\mathsf{D}^b(\mathcal{C}/\mathcal{A})
\]
is fully faithful, the last equivalence being~\cite[Theorem~3.2]{35}. This
functor is isomorphic to the composite of the functor \(\overline{Q_*}\)
induced by \(Q_*\) with the fully faithful functor
\(\mathsf{D}^b(\mathcal{B}/\mathcal{A})\to\mathsf{D}^b(\mathcal{C}/\mathcal{A})\).
Hence \(\overline{Q_*}\) is fully faithful. Finally, the essential images of
\(\mathsf{D}^b(\mathcal{B})/\mathsf{D}^b_{\mathcal{A}}(\mathcal{B})\) and of
\(\mathsf{D}^b(\mathcal{B}/\mathcal{A})\) in
\(\mathsf{D}^b(\mathcal{C}/\mathcal{A})\) are both the triangulated
subcategory generated by the images of the objects of \(\mathcal{B}\), so
\(\overline{Q_*}\) is also dense.
\end{proof}

If, moreover, the canonical functor
\(\mathsf{D}^b(\mathcal{A})\to\mathsf{D}^b_{\mathcal{A}}(\mathcal{B})\) is an
equivalence, the left-hand side identifies with
\(\mathsf{D}^b(\mathcal{B})/\mathsf{D}^b(\mathcal{A})\).

We now fix notation for singularity categories. For a graded algebra
\(\Lambda\) we write \(\mathsf{D}^{b}(\Lambda\textup{-mod})\),
\(\mathsf{D}^{b}(\Lambda\textup{-gmod})\) and
\(\mathsf{D}^{b}(\Lambda\textup{-GMod})\) for the bounded derived categories of
the corresponding module categories, and
\[
\mathsf{D}_{\textup{sg}}(\Lambda\textup{-gmod})
:=\mathsf{D}^{b}(\Lambda\textup{-gmod})/\mathsf{K}^{b}(\Lambda\textup{-proj}^{\mathbb{Z}}),
\qquad
\mathsf{D}_{\textup{sg}}(\Lambda\textup{-mod})
:=\mathsf{D}^{b}(\Lambda\textup{-mod})/\mathsf{K}^{b}(\Lambda\textup{-proj})
\]
for the singularity categories. A module \(M\in \Lambda\textup{-mod}\) is
\emph{Gorenstein-projective} if \(M \cong \ker(P^{0}\to P^{1})\) for an exact
complex \(P^\bullet\) of projective modules such that
\(\operatorname{Hom}_{\Lambda}(P^{\bullet},\Lambda)\) is exact. We denote by
\(\Lambda\textup{-}\underline{\textup{Gproj}}\) the stable category of
Gorenstein-projective \(\Lambda\)-modules, and by
\(\Lambda\textup{-}\underline{\textup{Gproj}}^{\mathbb{Z}}\) its graded
analogue. A finite-dimensional algebra \(\Lambda\) is
\emph{Iwanaga--Gorenstein} if
\(\operatorname{injdim}{}_{\Lambda}\Lambda<\infty\) and
\(\operatorname{injdim}\Lambda_{\Lambda}<\infty\). In this case the
Gorenstein-projective modules are the maximal Cohen--Macaulay modules, that is,
the modules \(M\) with \(\operatorname{Ext}^{i}_{\Lambda}(M,\Lambda)=0\) for
all \(i>0\), and by a theorem of Buchweitz~\cite{12} (see also
Orlov~\cite{36} for the graded setting) there are triangulated equivalences
\[
\mathsf{D}_{\textup{sg}}(\Lambda\textup{-mod})
\simeq
\Lambda\textup{-}\underline{\textup{Gproj}},
\qquad
\mathsf{D}_{\textup{sg}}(\Lambda\textup{-gmod})
\simeq
\Lambda\textup{-}\underline{\textup{Gproj}}^{\mathbb{Z}}.
\]

\subsection{Differential graded categories and orbit categories}

We use dg categories, dg functors and quasi-equivalences as in~\cite{29};
all categories and functors are \(k\)-linear. Recall that a \emph{dg category}
\(\mathcal{A}\) is a category whose morphism spaces are complexes
\(\operatorname{Hom}_{\mathcal{A}}(X,Y)=\bigoplus_{r\in\mathbb{Z}}\operatorname{Hom}_{\mathcal{A}}(X,Y)^{r}\) of
\(k\)-modules, such that the composition
\(\operatorname{Hom}_{\mathcal{A}}(Y,Z)\otimes\operatorname{Hom}_{\mathcal{A}}(X,Y)\to\operatorname{Hom}_{\mathcal{A}}(X,Z)\) is a morphism of
complexes and the identities are closed of degree \(0\). Its homotopy category
\(\mathrm{H}^0(\mathcal{A})\) has the same objects and morphism spaces
\(\mathrm{H}^0(\operatorname{Hom}_{\mathcal{A}}(X,Y))\). A \emph{dg functor}
\(G\colon\mathcal{A}\to\mathcal{C}\) is a functor such that the maps
\(\operatorname{Hom}_{\mathcal{A}}(X,Y)\to\operatorname{Hom}_{\mathcal{C}}(GX,GY)\) are morphisms of complexes; it
induces a functor \(\mathrm{H}^0(G)\colon\mathrm{H}^0(\mathcal{A})\to\mathrm{H}^0(\mathcal{C})\),
and \(G\) is a \emph{quasi-equivalence} if these maps are quasi-isomorphisms
and \(\mathrm{H}^0(G)\) is an equivalence.
Denote by \(\mathsf{D}(\mathcal{A})\) the derived category of dg
\(\mathcal{A}\)-modules~\cite{29}. A \emph{quasi-functor} \(\mathcal{A}\to\mathcal{C}\) is an object \(M\) of
\(\mathsf{D}(\mathcal{A}^{\mathrm{op}}\otimes\mathcal{C})\) such that each
\(M(-,X)\) is isomorphic in \(\mathsf{D}(\mathcal{C})\) to a representable
module; it induces a functor \(\mathrm{H}^0(\mathcal{A})\to\mathrm{H}^0(\mathcal{C})\).
For an additive category
\(\mathcal{A}\), \(\mathsf{C}_{\mathrm{dg}}(\mathcal{A})\) denotes the dg
category of complexes over \(\mathcal{A}\): its objects are complexes over
\(\mathcal{A}\), and
\(\operatorname{Hom}_{\mathsf{C}_{\mathrm{dg}}(\mathcal{A})}(X,Y)^{r}=\prod_{p\in\mathbb{Z}}\operatorname{Hom}_{\mathcal{A}}(X^{p},Y^{p+r})\)
with differential \(d(f)=d_Y\circ f-(-1)^{r}f\circ d_X\) for \(f\) of degree
\(r\). The shift acts on a morphism
\(f\) of degree \(r\) by \(([1]f)^{p}=(-1)^{r}f^{p+1}\), so that
\(\mathrm{H}^0(\mathsf{C}_{\mathrm{dg}}(\mathcal{A}))=\mathsf{K}(\mathcal{A})\).
For a dg category \(\mathcal{A}\) we denote by
\(\operatorname{pretr}(\mathcal{A})\) its pretriangulated hull, the dg category
of one-sided twisted complexes~\cite{9}. The category
\(\mathrm{H}^0(\operatorname{pretr}(\mathcal{A}))\), called the
\emph{triangulated hull} of \(\mathcal{A}\), is triangulated and generated by
the image of the fully faithful functor
\(\mathrm{H}^0(\mathcal{A})\to\mathrm{H}^0(\operatorname{pretr}(\mathcal{A}))\);
\(\mathcal{A}\) is \emph{pretriangulated} if this functor is an equivalence. A
\emph{dg enhancement} of a triangulated category \(\mathcal{T}\) is a
pretriangulated dg category \(\mathcal{A}\) with a triangulated equivalence
\(\mathrm{H}^0(\mathcal{A})\simeq\mathcal{T}\). For a full dg subcategory
\(\mathcal{B}\subseteq\mathcal{A}\) we denote by \(\mathcal{A}/\mathcal{B}\)
Drinfeld's dg quotient~\cite{15}; over a field,
\(\mathrm{H}^0(\operatorname{pretr}(\mathcal{A}/\mathcal{B}))\) is the Verdier
quotient of \(\mathrm{H}^0(\operatorname{pretr}(\mathcal{A}))\) by the thick
subcategory generated by \(\mathrm{H}^0(\mathcal{B})\)~\cite[\S3]{15}.
Moreover, if \(G\colon\mathcal{A}\to\mathcal{C}\) is a dg functor such that
\(G(U)\) is zero in \(\mathrm{H}^{0}(\mathcal{C})\) for all \(U\in\mathcal{B}\),
then, choosing contracting homotopies of the objects \(G(U)\), \(G\) extends to
a dg functor \(\mathcal{A}/\mathcal{B}\to\mathcal{C}\)~\cite[\S3]{15}.

\begin{definition}\label{def:dg-derived}
Let \(\mathcal{E}\) be an exact category as in Lemma~\ref{lem:2of3}.
\begin{enumerate}
\item The \emph{dg bounded derived category} is the Drinfeld quotient
\(\mathsf{D}_{\mathrm{dg}}^{b}(\mathcal{E}):=\mathsf{C}_{\mathrm{dg}}^{b}(\mathcal{E})/\mathcal{A}_{\mathrm{dg}}^{b}(\mathcal{E})\)
of the dg category of bounded complexes by the full dg subcategory of acyclic
complexes. Then
\(\mathrm{H}^{0}(\mathsf{D}_{\mathrm{dg}}^{b}(\mathcal{E}))\simeq\mathsf{D}^{b}(\mathcal{E})\).
\item If \(\mathcal{E}\) has enough projectives, the \emph{dg singularity
category} is
\[
\mathsf{D}_{\mathrm{sg}}^{\mathrm{dg}}(\mathcal{E}):=\mathsf{C}_{\mathrm{dg}}^{-,b}(\operatorname{Proj}\mathcal{E})/\mathsf{C}_{\mathrm{dg}}^{b}(\operatorname{Proj}\mathcal{E}),
\]
where \(\mathsf{C}_{\mathrm{dg}}^{-,b}\) denotes bounded above complexes with
bounded cohomology. Then
\(\mathrm{H}^{0}(\mathsf{D}_{\mathrm{sg}}^{\mathrm{dg}}(\mathcal{E}))\simeq\mathsf{D}^{b}(\mathcal{E})/\mathsf{K}^{b}(\operatorname{Proj}\mathcal{E})\).
\end{enumerate}
\end{definition}

\begin{definition}\label{def:galois}
Let \(F\) be a triangulated autoequivalence of a triangulated category
\(\mathcal{T}\), and let \(G=\langle F\rangle\) be the infinite cyclic group
generated by \(F\). A triangulated functor
\(\ell\colon\mathcal{T}\to\mathcal{T}'\), together with an isomorphism
\(\varphi\colon\ell\circ F\xrightarrow{\sim}\ell\), is a
\emph{\(G\)-Galois precovering} if, for all \(X,Y\in\mathcal{T}\), the map
\[
\bigoplus_{n\in\mathbb{Z}}\operatorname{Hom}_{\mathcal{T}}(X,F^{n}Y)
\longrightarrow\operatorname{Hom}_{\mathcal{T}'}(\ell X,\ell Y)
\]
induced by \(\ell\) and \(\varphi\) is bijective. It is a
\emph{\(G\)-Galois covering} if, moreover, \(\ell\) is dense. For instance,
\(\langle\langle 1\rangle\rangle\) denotes the group generated by the grading
shift \(\langle 1\rangle\).
\end{definition}

\begin{definition}\label{def:orbit}
Let \(F\) be an autoequivalence of a triangulated category \(\mathcal{T}\).
The \emph{orbit category} \(\mathcal{T}/F\) has the same objects as
\(\mathcal{T}\) and morphism spaces
\(\operatorname{Hom}_{\mathcal{T}/F}(X,Y):=\bigoplus_{n\in\mathbb{Z}}\operatorname{Hom}_{\mathcal{T}}(X,F^{n}Y)\),
with composition \((g\circ f)_{s}=\sum_{m+n=s} F^{n}(g_{m})\circ f_{n}\).
We denote by \(\pi\colon \mathcal{T}\to\mathcal{T}/F\) the canonical functor.
\end{definition}

The functor \(\pi\), together with the canonical isomorphism
\(\pi\circ F\cong\pi\), is dense and satisfies the conditions of
Definition~\ref{def:galois}, except that \(\mathcal{T}/F\) is in general only an
additive category.

\begin{proposition}\label{prop:orbit-univ}
Let \(\ell\colon\mathcal{T}\to\mathcal{T}'\) be an additive functor together
with an isomorphism \(\varphi\colon \ell\circ F\xrightarrow{\sim}\ell\). Then
there is a unique additive functor
\(\overline{\ell}\colon\mathcal{T}/F\to\mathcal{T}'\) with
\(\overline{\ell}\circ\pi=\ell\) which is compatible with \(\varphi\). It is
fully faithful if and only if, for all \(X,Y\in\mathcal{T}\), the map
\(\bigoplus_{n\in\mathbb{Z}}\operatorname{Hom}_{\mathcal{T}}(X,F^nY)\to\operatorname{Hom}_{\mathcal{T}'}(\ell X,\ell Y)\)
induced by \(\ell\) and \(\varphi\) is bijective, and it is an equivalence if,
moreover, \(\ell\) is dense.
\end{proposition}

In general, \(\mathcal{T}/F\) does not carry a triangulated structure
compatible with that of \(\mathcal{T}\). Keller's dg orbit categories provide
a canonical triangulated hull~\cite{28}.

\begin{definition}\label{def:dg-orbit}
Let \(\mathcal{A}\) be a dg category and \(F\colon \mathcal{A}\to\mathcal{A}\)
a dg functor inducing an autoequivalence of \(\mathrm{H}^0(\mathcal{A})\). The
\emph{dg orbit category} \(\mathcal{A}/F\) has the same objects as
\(\mathcal{A}\) and morphism complexes
\[
\operatorname{Hom}_{\mathcal{A}/F}(X,Y)
:= \operatorname*{colim}_{p\ge0}\;
\bigoplus_{n\ge 0}\operatorname{Hom}_{\mathcal{A}}\bigl(F^{n}X,\,F^{p}Y\bigr),
\]
with the transition maps and composition of~\cite{28}. Then
\(\mathrm{H}^0(\mathcal{A}/F)\simeq\mathrm{H}^0(\mathcal{A})/\mathrm{H}^0(F)\),
and \(\mathrm{H}^0(\operatorname{pretr}(\mathcal{A}/F))\) is called the
\emph{triangulated hull} of the orbit category
\(\mathrm{H}^0(\mathcal{A})/\mathrm{H}^0(F)\).
\end{definition}

We use the following universal property of the dg orbit category, due to
Keller~\cite{28}. Let \(\mathcal{B}\) be a pretriangulated dg category and
\(\ell\colon \mathcal{A}\to \mathcal{B}\) a dg functor together with an
isomorphism \(\ell\circ F \cong \ell\) in
\(\mathsf{D}(\mathcal{A}^{\mathrm{op}}\otimes\mathcal{B})\). Then there is a
quasi-functor \(\overline{\ell}\colon \operatorname{pretr}(\mathcal{A}/F)\to
\mathcal{B}\) whose composition with the canonical projection
\(\mathcal{A}\to\mathcal{A}/F\) is isomorphic to \(\ell\) in
\(\mathsf{D}(\mathcal{A}^{\mathrm{op}}\otimes\mathcal{B})\). The following
result is essentially due to Keller~\cite{28}; compare
also~\cite[Proposition~7.13]{1}. We include the short proof
for completeness.

\begin{theorem}\label{thm:orbit-hull}
Let \(\mathcal{T}\) and \(\mathcal{T}'\) be triangulated categories, let
\(F\colon \mathcal{T}\to \mathcal{T}\) be a triangulated autoequivalence, and
let \(\ell\colon\mathcal{T}\to\mathcal{T}'\) be a
\(\langle F\rangle\)-Galois precovering. Assume that \(\mathcal{T}\) and
\(\mathcal{T}'\) admit dg enhancements \(\mathcal{T}_{\mathrm{dg}}\) and
\(\mathcal{T}'_{\mathrm{dg}}\),
that \(\ell\) and \(F\) lift to dg functors
\(\ell_{\mathrm{dg}} \colon \mathcal{T}_{\mathrm{dg}} \to \mathcal{T}'_{\mathrm{dg}}\)
and \(F_{\mathrm{dg}} \colon \mathcal{T}_{\mathrm{dg}} \to \mathcal{T}_{\mathrm{dg}}\),
and that \(\ell_{\mathrm{dg}} \circ F_{\mathrm{dg}} \cong \ell_{\mathrm{dg}}\) in
\(\mathsf{D}(\mathcal{T}_{\mathrm{dg}}^{\mathrm{op}}\otimes\mathcal{T}'_{\mathrm{dg}})\),
inducing the given isomorphism \(\ell\circ F\cong\ell\). Then
\(\ell_{\mathrm{dg}}\) induces a fully faithful triangulated functor
\[
\mathrm{H}^0\bigl(\operatorname{pretr}(\mathcal{T}_{\mathrm{dg}} / F_{\mathrm{dg}})\bigr)
\longrightarrow
\mathcal{T}'.
\]
If its essential image generates \(\mathcal{T}'\) as a triangulated category,
it is an equivalence, and
\(\operatorname{pretr}(\mathcal{T}_{\mathrm{dg}} / F_{\mathrm{dg}})\) is a dg
enhancement of \(\mathcal{T}'\).
\end{theorem}

\begin{proof}
Let \(\pi\colon \mathcal{T}_{\mathrm{dg}}\to\mathcal{T}_{\mathrm{dg}}/F_{\mathrm{dg}}\)
be the canonical projection. By the universal property recalled above, there
is a quasi-functor
\(\overline{\ell}_{\mathrm{dg}}\colon\operatorname{pretr}(\mathcal{T}_{\mathrm{dg}}/F_{\mathrm{dg}})\to\mathcal{T}'_{\mathrm{dg}}\)
whose composition with \(\pi\) is isomorphic to \(\ell_{\mathrm{dg}}\) in
\(\mathsf{D}(\mathcal{T}_{\mathrm{dg}}^{\mathrm{op}}\otimes\mathcal{T}'_{\mathrm{dg}})\).
Passing to \(\mathrm{H}^0\), we obtain a triangulated functor
\(\Phi\colon\mathrm{H}^{0}(\operatorname{pretr}(\mathcal{T}_{\mathrm{dg}}/F_{\mathrm{dg}}))\to\mathrm{H}^{0}(\mathcal{T}'_{\mathrm{dg}})\simeq\mathcal{T}'\).
Its restriction to the full subcategory
\(\mathrm{H}^{0}(\mathcal{T}_{\mathrm{dg}}/F_{\mathrm{dg}})\simeq\mathcal{T}/F\)
is the functor \(\overline{\ell}\) of Proposition~\ref{prop:orbit-univ}, which
is fully faithful since \(\ell\) is an \(\langle F\rangle\)-Galois
precovering. This subcategory is closed under shifts, since \(\mathcal{T}_{\mathrm{dg}}\) is
pretriangulated, and it generates
\(\mathrm{H}^{0}(\operatorname{pretr}(\mathcal{T}_{\mathrm{dg}}/F_{\mathrm{dg}}))\)
as a triangulated category; hence \(\Phi\) is fully faithful by a standard
d\'evissage argument. Finally, the essential image of \(\Phi\) is a
triangulated subcategory of \(\mathcal{T}'\); if it generates
\(\mathcal{T}'\), then \(\Phi\) is dense, hence an equivalence, and
\(\operatorname{pretr}(\mathcal{T}_{\mathrm{dg}}/F_{\mathrm{dg}})\) is a dg
enhancement of \(\mathcal{T}'\).
\end{proof}

\begin{remark}\label{rem:models}
Let \(\Psi\colon\mathcal{A}\to\mathcal{A}'\) be a quasi-equivalence of dg
categories with \(\Psi\circ F=F'\circ\Psi\), where \(F\) and \(F'\) are as in
Definition~\ref{def:dg-orbit}. Then \(\Psi\) induces a quasi-equivalence
\(\mathcal{A}/F\to\mathcal{A}'/F'\), since morphism complexes in a dg orbit
category are filtered colimits of direct sums of morphism complexes, and hence
a triangulated equivalence between the triangulated hulls. All dg models of
bounded derived and singularity categories used below are related by zigzags
of such quasi-equivalences, induced by additive functors on complexes, and
these commute strictly with the grading shift and the suspension. In
particular, the triangulated hulls in Section~\ref{sec:nongraded} do not depend
on the choice of models, and in Theorem~\ref{thm:orbit-hull} we may compute
with any model on which the relevant functors are dg functors.
\end{remark}
\section{Graded Koszul dualities}\label{sec:graded}

Unless stated otherwise, \(\Lambda\) denotes a finite-dimensional
\emph{quadratic} algebra. Assuming that \(\Lambda\) is Koszul, we establish a
triangulated equivalence between the bounded derived category of
finite-dimensional graded \(\Lambda\)-modules and the bounded derived category
of locally finite-dimensional graded coperfect \(\Lambda^{!}\)-modules, as well
as with the bounded derived category of graded perfect \(\Lambda^{!}\)-modules. This removes the finiteness assumptions on
\(\Lambda^{!}\) that appear in~\cite[Theorem~2.12.6]{6} (Noetherian) and
in~\cite{33} (coherent). We further determine when this graded Koszul
duality induces both a graded derived equivalence and a derived equivalence
between \(\Lambda\) and its Koszul dual \(\Lambda^{!}\).

As applications, we obtain a graded singularity duality for arbitrary
finite-dimensional Koszul algebras and, for Iwanaga--Gorenstein Koszul
algebras, a general graded version of the BGG correspondence. Moreover, we
establish two dg Koszul dualities: one for the dg bounded derived
categories and another for dg singularity categories, thereby extending
the BGG correspondence to the dg setting.

\subsection{The derived quadratic functor}\label{sec:quadratic-functor}
In this subsection, we construct explicitly the \emph{derived quadratic functor}
from the bounded derived category of right bounded, locally
finite\mbox{-}dimensional graded \(\Lambda^{!}\)\nobreakdash-modules with
bounded cohomology, in the sense of Definition~\ref{def:bounded-coh}, to the
bounded derived category of finite\mbox{-}dimensional graded
\(\Lambda\)\nobreakdash-modules.

At the abelian level, this construction goes back to Martínez-Villa and Saorín
\cite[Thm.~2.4]{33}, who established an equivalence between the category of finitely
generated graded \(\Lambda^{!}\)\nobreakdash-modules and the category of linear complexes of
projective \(\Lambda\)\nobreakdash-modules by assigning to a graded
\(\Lambda^{!}\)\nobreakdash-module a complex of projective
\(\Lambda\)\nobreakdash-modules. See also Mazorchuk--Ovsienko--Stroppel in the
context of graded categories; cf.~\cite[Thm.~12]{34}.

Bautista and Liu~\cite[Thm.~3.9]{3} extended this construction to bounded
derived categories and obtained a triangulated equivalence for
radical\mbox{-}square\mbox{-}zero algebras given by possibly infinite
gradable quivers. More recently, a general Koszul duality framework between
\(\mathsf{D}^{\downarrow}(\mathrm{Mod}\,\Lambda^{!})\) and
\(\mathsf{D}^{\uparrow}(\mathrm{Mod}\,\Lambda)\) was developed in~\cite{11} under the
assumption that \(\Lambda\) is Koszul and given by a gradable quiver.

Our approach follows the same general strategy as in~\cite{3,11}. The
construction is carried out for arbitrary finite\mbox{-}dimensional quadratic
algebras, and the resulting equivalences hold for all finite\mbox{-}dimensional
Koszul algebras; for finite quivers, the radical\mbox{-}square\mbox{-}zero case
of~\cite{3} is recovered as a special case. Whenever the arguments coincide
with those in~\cite{3,11}, we omit the proofs.
Throughout, we write \(\otimes=\otimes_k\) for the tensor product over the base
field \(k\), and all functors are assumed to be \(k\)\nobreakdash-linear.
\medskip

We begin with the following elementary yet useful observation; see \cite[Lemma~3.3]{11}.

\begin{proposition}\label{prop:hom-proj}
Let \(\Lambda=kQ/I\) be a finite\mbox{-}dimensional graded algebra with \(I\)
homogeneous, and let \(P_{x}\langle n\rangle, P_{y}\langle m\rangle\) be the
indecomposable graded projectives in \(\Lambda\textup{-}\mathrm{gmod}\)
corresponding to vertices \(x,y\in Q_{0}\) and shifts \(n,m\in\mathbb{Z}\).
Fix a basis \(\mathcal{P}_{y,x}\) of \(e_{x}\Lambda_{m-n}e_{y}\) consisting of
classes of paths from \(y\) to \(x\) of length \(m-n\). Let \(U,V\) be
finite\mbox{-}dimensional \(k\)\nobreakdash-vector spaces. Then every graded
\(\Lambda\)\nobreakdash-module homomorphism
\(f\colon P_{x}\langle n\rangle\otimes U\to P_{y}\langle m\rangle\otimes V\)
(of degree \(0\)) admits a unique expression
\(f=\sum_{l\in\mathcal{P}_{y,x}} P_{l}\otimes\psi_{l}\) with
\(\psi_{l}\in\operatorname{Hom}_{k}(U,V)\), where \(P_{l}\) denotes right
multiplication by \(l\): \(P_{l}(q)=q\,l\) for every path \(q\) with
\(s(q)=x\).
\end{proposition}
\medskip 

We now begin the construction of the derived quadratic functor. Let
\(\Lambda^{!}\textup{-}\mathrm{GMod}\) denote the abelian category of locally
finite-dimensional \(\mathbb{Z}\)-graded left \(\Lambda^{!}\)-modules (as fixed
in Section~\ref{sec:prelim}), and let
\(\mathsf{C}\bigl(\Lambda\textup{-}\mathrm{proj}^{\mathbb{Z}}\bigr)\) be the
category of cochain complexes of graded projective \(\Lambda\)-modules whose
terms are finite direct sums of indecomposable graded projectives
\(P_{x}\langle n\rangle\), with \(x\in Q_{0}\) and \(n\in\mathbb{Z}\).

Following~\cite[Thm.~12]{34} (see also~\cite[Thm.~2.4]{33},
\cite[Lem.~3.1]{3}, and \cite[Prop.~5.1]{11}), we define a
\(k\)\nobreakdash-linear functor
\[
K\colon \Lambda^{!}\textup{-}\mathrm{GMod}\longrightarrow\mathsf{C}\bigl(\Lambda\textup{-}\mathrm{proj}^{\mathbb{Z}}\bigr)
\]
as follows. Let \(M\in \Lambda^{!}\textup{-}\mathrm{GMod}\). For \(x\in Q_0\) and
\(n\in\mathbb{Z}\), we write \(M_x^n\) for the homogeneous component of degree
\(n\) at the vertex \(x\), denoted \(M_n(x)\) in Section~\ref{sec:prelim}. We
set
\[
K(M)^n:=\bigoplus_{x\in Q_0} P_x\langle n\rangle \otimes_k M_x^n,\qquad n\in\mathbb{Z},
\]
and define the differential
\(d^n\colon K(M)^n\to K(M)^{n+1}\) by its \((y,x)\)\nobreakdash-component
\[
d^n_{yx}
=
\sum_{\alpha\colon y\to x}
\bigl(P_{\alpha}\colon P_x\langle n\rangle \to P_y\langle n+1\rangle\bigr)
\otimes
\bigl(M(\alpha^{\mathrm{op}})\colon M_x^n \to M_y^{n+1}\bigr),
\]
where the sum runs over all arrows \(\alpha\) in \(Q\) with \(s(\alpha)=y\)
and \(t(\alpha)=x\), and \(\alpha^{\mathrm{op}}\) denotes the corresponding
arrow of \(Q^{\mathrm{op}}\). Here \(P_{\alpha}\) is right multiplication by
\(\alpha\), that is, \(P_{\alpha}(q)=q\,\alpha\in P_y\) for a path \(q\in P_x\),
and \(M(\alpha^{\mathrm{op}})\) is given by the \(\Lambda^{!}\)-module
structure of \(M\). Then \(K(M)=(K(M)^n,d^n)\) is a cochain complex of graded
projective \(\Lambda\)-modules.

The assignment \(M\mapsto K(M)\) defines an exact and fully faithful functor
\(K\colon \Lambda^{!}\textup{-}\mathrm{GMod}\to\mathsf{C}\bigl(\Lambda\textup{-}\mathrm{proj}^{\mathbb{Z}}\bigr)\),
whose essential image is the full subcategory
\(\mathcal{LC}\bigl(\Lambda\textup{-}\mathrm{proj}^{\mathbb{Z}}\bigr)\) of
\emph{linear complexes}, that is, complexes \(\mathbf{P}=(P^n,d^n)\) in which
each \(P^n\) is a finite direct sum of modules \(P_x\langle n\rangle\),
\(x\in Q_0\), and is therefore generated in degree \(-n\).

\begin{definition}\label{def:bounded-coh}
A module \(M \in \Lambda^{!}\textup{-}\mathrm{GMod}^{-}\) is said to have
\emph{bounded cohomology} if the complex \(K(M)\) has bounded cohomology. We
denote by \(\Lambda^{!}\textup{-}\mathrm{GMod}^{-,b}\) the full subcategory of
\(\Lambda^{!}\textup{-}\mathrm{GMod}^{-}\) consisting of such modules, and by
\(\mathcal{LC}^{-,b}\bigl(\Lambda\textup{-}\mathrm{proj}^{\mathbb{Z}}\bigr)\)
the full subcategory of
\(\mathcal{LC}\bigl(\Lambda\textup{-}\mathrm{proj}^{\mathbb{Z}}\bigr)\)
consisting of right bounded linear complexes with bounded cohomology. Thus
\(K\) restricts to a functor
\[
K \colon \Lambda^{!}\textup{-}\mathrm{GMod}^{-,b} \longrightarrow \mathcal{LC}^{-,b}\bigl(\Lambda\textup{-}\mathrm{proj}^{\mathbb{Z}}\bigr).
\]
\end{definition}

Since \(K\) is exact, the long exact cohomology sequence shows that
\(\Lambda^{!}\textup{-}\mathrm{GMod}^{-,b}\) satisfies the two-out-of-three
property in the abelian category \(\Lambda^{!}\textup{-}\mathrm{GMod}^{-}\).
Hence it is an exact category as in Section~\ref{sec:prelim}, closed under
extensions, and weakly idempotent complete by Lemma~\ref{lem:2of3}. The
following result is an immediate consequence of~\cite[Theorem~12]{34}.

\begin{proposition}\label{prop:K-equiv}
The \(k\)-linear functor
\[
K \colon \Lambda^{!}\textup{-}\mathrm{GMod}^{-,b}\longrightarrow
\mathcal{LC}^{-,b}\bigl(\Lambda\textup{-}\mathrm{proj}^{\mathbb{Z}}\bigr)
\]
is an equivalence of exact categories.
\end{proposition}

We now extend \(K\) to bounded cochain complexes in
\(\Lambda^{!}\textup{-}\mathrm{GMod}^{-,b}\). In this way we obtain a
\(k\)-linear functor
\[
\mathfrak{K} \colon \mathsf{C}^{b}\bigl(\Lambda^{!}\textup{-}\mathrm{GMod}^{-,b}\bigr) \longrightarrow \mathsf{DC}^{b,-}\bigl(\Lambda\textup{-}\mathrm{proj}^{\mathbb{Z}}\bigr),
\]
where \(\mathsf{DC}^{b,-}\bigl(\Lambda\textup{-}\mathrm{proj}^{\mathbb{Z}}\bigr)\)
denotes the category of double cochain complexes with terms in
\(\Lambda\textup{-}\mathrm{proj}^{\mathbb{Z}}\), bounded in the horizontal
direction and right bounded in the vertical direction. More precisely, let
\(M^\bullet=(M^{p},d^{p})\) be a bounded cochain complex with each
\(M^{p}\in \Lambda^{!}\textup{-}\mathrm{GMod}^{-,b}\), and write
\(M^{p}_{x,q}\) for the homogeneous component of \(M^{p}\) of internal degree
\(q\) at the vertex \(x\). Applying \(K\) degreewise yields a double complex
\(\mathfrak{K}(M^\bullet)\) with
\(\mathfrak{K}(M^\bullet)^{p,q}:=\bigoplus_{x\in Q_{0}} P_{x}\langle q\rangle \otimes_k M^{p}_{x,q}\),
a portion of which is displayed in the diagram
\[
\xymatrix@C=3.4em@R=3.0em{
\cdots & \cdots & \cdots & \cdots \\
\cdots \ar[r] &
\displaystyle\bigoplus_{y} P_{y}\langle q+1\rangle \otimes M^{p}_{y,\,q+1}
\ar[r]^{\ d^{p,\,q+1}_{1}\ } \ar[u] &
\displaystyle\bigoplus_{y} P_{y}\langle q+1\rangle \otimes M^{p+1}_{y,\,q+1}
\ar[r] \ar[u] &
\cdots \\
\cdots \ar[r] &
\displaystyle\bigoplus_{x} P_{x}\langle q\rangle \otimes M^{p}_{x,\,q}
\ar[r]^{\ d^{p,\,q}_{1}\ } \ar[u]^{\ d^{p,\,q}_{2}\ } &
\displaystyle\bigoplus_{x} P_{x}\langle q\rangle \otimes M^{p+1}_{x,\,q}
\ar[r] \ar[u]_{\,d^{p+1,\,q}_{2}} &
\cdots \\
\cdots & \cdots \ar[u] & \cdots \ar[u] & \cdots
}
\]
The horizontal differential is
\(d^{p,q}_{1}=\sum_{x\in Q_{0}} \mathrm{id}_{P_{x}\langle q\rangle}\otimes d^{p}_{x,q}\),
where \(d^{p}_{x,q}\colon M^{p}_{x,q}\to M^{p+1}_{x,q}\) is the
\((x,q)\)-component of \(d^{p}\), and the vertical differential
\(d^{p,q}_{2}\) is the differential of \(K(M^{p})\), that is,
\[
\bigl(d^{p,q}_{2}\bigr)_{yx}
=
\sum_{\alpha\colon y\to x}
\bigl(P_{\alpha}\colon P_{x}\langle q\rangle \to P_{y}\langle q+1\rangle\bigr)
\otimes
\bigl(M^{p}(\alpha^{\mathrm{op}})\colon M^{p}_{x,q} \to M^{p}_{y,q+1}\bigr),
\]
with \(P_\alpha\) and \(M^{p}(\alpha^{\mathrm{op}})\) as above. Since each
\(d^{p}\) is \(\Lambda^{!}\)-linear, one has \(d_1^2=0\), \(d_2^2=0\) and
\(d_1d_2=d_2d_1\), so \(\mathfrak{K}(M^\bullet)\) is a double cochain complex.

The total complex of a double complex \(A\in\mathsf{DC}^{b,-}\bigl(\Lambda\textup{-}\mathrm{proj}^{\mathbb{Z}}\bigr)\)
is defined as usual by \(\operatorname{Tot}(A)^{n}=\bigoplus_{p+q=n}A^{p,q}\),
a finite sum, with differential \(d_1^{p,q} + (-1)^{p} d_2^{p,q}\) on
\(A^{p,q}\). This defines a \(k\)-linear functor
\(\operatorname{Tot} \colon \mathsf{DC}^{b,-}\bigl(\Lambda\textup{-}\mathrm{proj}^{\mathbb{Z}}\bigr) \to \mathsf{C}^{-}\bigl(\Lambda\textup{-}\mathrm{proj}^{\mathbb{Z}}\bigr)\).
Composing \(\mathfrak{K}\) with \(\operatorname{Tot}\), we obtain a
\(k\)-linear functor
\(\mathcal{F} \colon \mathsf{C}^{b}\bigl(\Lambda^{!}\textup{-}\mathrm{GMod}^{-,b}\bigr) \to \mathsf{C}^{-,b}\bigl(\Lambda\textup{-}\mathrm{proj}^{\mathbb{Z}}\bigr)\),
which we call the \emph{complex quadratic functor}; its values have bounded
cohomology, since \(\operatorname{Tot}(\mathfrak{K}(M^\bullet))\) is an iterated
extension of the finitely many complexes \(K(M^{p})[-p]\). By construction,
\(\mathcal{F}\) extends \(K\colon \Lambda^{!}\textup{-}\mathrm{GMod}^{-,b} \to \mathcal{LC}^{-,b}\bigl(\Lambda\textup{-}\mathrm{proj}^{\mathbb{Z}}\bigr)\).

By~\cite[Lemmas~3.6 and~3.7]{3}, the functor \(\mathcal{F}\) commutes with
the shift functor \([1]\), preserves mapping cones, and sends null-homotopic
morphisms to null-homotopic morphisms. Hence it induces a functor on homotopy
categories. Moreover, \(\mathcal{F}\) preserves acyclic complexes: if
\(M^\bullet\) is acyclic, the rows of \(\mathfrak{K}(M^\bullet)\) are exact, and
since there are only finitely many columns,
\(\operatorname{Tot}(\mathfrak{K}(M^\bullet))\) is exact. Hence \(\mathcal{F}\)
descends to the derived categories, and we obtain a commutative diagram
\[
\xymatrix@C=5.5em@R=4em{
\mathsf{C}^{b}\bigl(\Lambda^{!}\textup{-}\mathrm{GMod}^{-,b}\bigr)
\ar[r]^-{\mathcal{F}} \ar[d]_{\scriptstyle \pi} &
\mathsf{C}^{-,b}\bigl(\Lambda\textup{-}\mathrm{proj}^{\mathbb{Z}}\bigr)
\ar[d]^{\scriptstyle \pi'} \\
\mathsf{K}^{b}\bigl(\Lambda^{!}\textup{-}\mathrm{GMod}^{-,b}\bigr)
\ar[r]^-{\mathscr{F}} \ar[d]_{\scriptstyle \epsilon} &
\mathsf{K}^{-,b}\bigl(\Lambda\textup{-}\mathrm{proj}^{\mathbb{Z}}\bigr)
\ar[d]^{\scriptstyle \mathfrak{T}} \\
\mathsf{D}^{b}\bigl(\Lambda^{!}\textup{-}\mathrm{GMod}^{-,b}\bigr)
\ar[r]^-{\mathfrak{F}} &
\mathsf{D}^{b}\bigl(\Lambda\textup{-}\mathrm{gmod}\bigr)
}
\]
where \(\pi,\pi'\) are the canonical projections, \(\epsilon\) is the
localization functor, \(\mathscr{F}\) and \(\mathfrak{F}\) are the functors
induced by \(\mathcal{F}\), and \(\mathfrak{T}\) is the standard equivalence
\(\mathsf{K}^{-,b}\bigl(\Lambda\textup{-}\mathrm{proj}^{\mathbb{Z}}\bigr) \xrightarrow{\sim} \mathsf{D}^{b}\bigl(\Lambda\textup{-}\mathrm{gmod}\bigr)\)
of Section~\ref{sec:prelim}.

The triangulated functor
\(\mathfrak{F} \colon \mathsf{D}^{b}\bigl(\Lambda^{!}\textup{-}\mathrm{GMod}^{-,b}\bigr) \to \mathsf{D}^{b}\bigl(\Lambda\textup{-}\mathrm{gmod}\bigr)\)
will be referred to as the \emph{derived quadratic functor}. Explicitly,
\(\mathfrak{F}(M^\bullet) = \mathfrak{T}\bigl(\operatorname{Tot}(\mathfrak{K}(M^\bullet))\bigr)\)
for a bounded cochain complex \(M^\bullet\). In the next subsection we
introduce the category \(\Lambda^{!}\textup{-}\mathrm{Cop}^{\mathbb{Z}}\) of
coperfect graded \(\Lambda^{!}\)-modules and prove that, if \(\Lambda\) is
Koszul, then
\(\Lambda^{!}\textup{-}\mathrm{GMod}^{-,b}=\Lambda^{!}\textup{-}\mathrm{Cop}^{\mathbb{Z}}\)
and \(\mathfrak{F}\) is a triangulated equivalence. Conversely, Koszulity is
characterized both by \(\mathfrak{F}\) and by the analogous functor in the
opposite direction.

\subsection{Graded derived Koszul duality}

Before stating the main results, we fix some terminology.

We apply the notation of Section~\ref{sec:prelim} to \(\Lambda^{!}\). In particular, for \(x\in Q_0\) and
\(n\in\mathbb{Z}\) we write \(P^{!}_x\langle n\rangle\),
\(I^{!}_x\langle n\rangle\) and \(S^{!}_x\langle n\rangle\) for the
indecomposable projective, injective and simple graded \(\Lambda^{!}\)-modules,
and \(\Lambda^{!}\textup{-}\mathrm{Proj}^{\mathbb{Z}}\) and
\(\Lambda^{!}\textup{-}\mathrm{Inj}^{\mathbb{Z}}\) for the categories of finite
direct sums of them.

Let \(M\in \Lambda^{!}\textup{-}\mathrm{GMod}^{-}\). The \emph{socle}
\(\operatorname{soc}(M)\) of \(M\) is its largest semisimple graded submodule;
concretely, for \(x\in Q_{0}\) and \(n\in\mathbb{Z}\),
\[
\operatorname{soc}(M)_{n}(x)
=
\bigcap_{\alpha\colon y\to x}
\ker\bigl(M(\alpha^{\mathrm{op}})\colon M_{n}(x)\longrightarrow M_{n+1}(y)\bigr),
\]
where the intersection is taken over all arrows \(\alpha\colon y\to x\) in
\(Q\). The module \(M\) is \emph{finitely cogenerated} if
\(\operatorname{soc}(M)\) is finite-dimensional and essential in \(M\);
equivalently, there exists a monomorphism \(M \to I^{!}\) with
\(I^{!}\in\Lambda^{!}\textup{-}\mathrm{Inj}^{\mathbb{Z}}\). In particular, each
\(I^{!}_{x}\langle n\rangle\) is finitely cogenerated, with
\(\operatorname{soc}(I^{!}_{x}\langle n\rangle) \cong S^{!}_{x}\langle n\rangle\).
The module \(M\) is \emph{finitely copresented} if it admits an injective
copresentation \(0 \to M \to I^{!0} \to I^{!1}\) with
\(I^{!0},I^{!1}\in\Lambda^{!}\textup{-}\mathrm{Inj}^{\mathbb{Z}}\); we denote by
\(\Lambda^{!}\textup{-}\mathrm{Fcp}^{\mathbb{Z}}\) the full subcategory of
\(\Lambda^{!}\textup{-}\mathrm{GMod}^{-}\) consisting of such modules. Finally,
\(M\) is \emph{coperfect} if it has finite injective dimension and admits a
minimal graded injective coresolution
\(0 \to M \to I^{!0} \to I^{!1} \to \cdots \to I^{!d} \to 0\) with each
\(I^{!i}\) finitely cogenerated, that is,
\(I^{!i}\in\Lambda^{!}\textup{-}\mathrm{Inj}^{\mathbb{Z}}\). We denote by
\(\Lambda^{!}\textup{-}\mathrm{Cop}^{\mathbb{Z}}\) the full subcategory of
\(\Lambda^{!}\textup{-}\mathrm{GMod}^{-}\) consisting of coperfect graded
modules.

Dually, a graded module \(M\in \Lambda^{!}\textup{-}\mathrm{GMod}^{+}\) is
\emph{finitely generated} if there exists an epimorphism \(P^{!}\to M\) with
\(P^{!}\in\Lambda^{!}\textup{-}\mathrm{Proj}^{\mathbb{Z}}\), and \emph{finitely
presented} if it admits a projective presentation
\(P^{!-1}\to P^{!0}\to M\to 0\) with
\(P^{!-1},P^{!0}\in\Lambda^{!}\textup{-}\mathrm{Proj}^{\mathbb{Z}}\). We denote
the corresponding full subcategories of \(\Lambda^{!}\textup{-}\mathrm{GMod}\)
by \(\Lambda^{!}\textup{-}\mathrm{Fg}^{\mathbb{Z}}\) and
\(\Lambda^{!}\textup{-}\mathrm{Fp}^{\mathbb{Z}}\). A graded module \(M\) is
\emph{perfect} if it admits a finite graded projective resolution
\(0 \to P^{!-d} \to \cdots \to P^{!-1} \to P^{!0} \to M \to 0\) with each
\(P^{!i}\in\Lambda^{!}\textup{-}\mathrm{Proj}^{\mathbb{Z}}\); we denote by
\(\Lambda^{!}\textup{-}\mathrm{Pe}^{\mathbb{Z}}\) the full subcategory of
perfect graded \(\Lambda^{!}\)-modules. We say that \(\Lambda^{!}\) is
\emph{graded coherent} (resp.\ \emph{graded cocoherent}) if
\(\Lambda^{!}\textup{-}\mathrm{Fp}^{\mathbb{Z}}\) (resp.\
\(\Lambda^{!}\textup{-}\mathrm{Fcp}^{\mathbb{Z}}\)) is an abelian subcategory of
\(\Lambda^{!}\textup{-}\mathrm{GMod}\).

Recall that \(I^{!}_{x}=\mathbb{D}P^{!\,\mathrm{op}}_{x}\), where
\(P^{!\,\mathrm{op}}_{x}=\Lambda^{!\,\mathrm{op}}(x,-)\). If
\(\alpha\colon y\to z\) is an arrow in \(Q\), then \(\alpha^{\mathrm{op}}\colon
z\to y\) induces the map
\(\Lambda^{!\,\mathrm{op}}(x,y)\to \Lambda^{!\,\mathrm{op}}(x,z)\),
\(f\mapsto f\circ \alpha^{\mathrm{op}}\), and dualizing gives the structure map
\(I^{!}_{x}(\alpha^{\mathrm{op}})\colon I^{!}_{x}(z)^{m}\to I^{!}_{x}(y)^{m+1}\).
Hence, for \(x\in Q_{0}\) and \(n\in\mathbb{Z}\),
\[
K\bigl(I^{!}_{x}\langle n\rangle\bigr)^{m}
=
\bigoplus_{z\in Q_{0}} P_{z}\langle m\rangle \otimes_{k} I^{!}_{x}(z)^{m+n},
\qquad m\in\mathbb{Z},
\]
and the differential of \(K(I^{!}_{x}\langle n\rangle)\) is given by the
formula for \(d^{n}_{yx}\) in Section~\ref{sec:quadratic-functor}, with
\(M(\alpha^{\mathrm{op}})\) replaced by \(I^{!}_{x}(\alpha^{\mathrm{op}})\).

The following lemma extends~\cite[Lemma~3.4]{3}, originally proved for radical
square zero algebras, to the class of arbitrary finite-dimensional Koszul
algebras. The argument of~\cite[Lemma~5.4]{11} carries over verbatim to this
setting; see also~\cite[Theorem~30]{34} and~\cite[Remark~(1), p.~477]{6}.

\begin{lemma}\label{lem:koszul-res}
Let \(\Lambda\) be a finite-dimensional Koszul algebra with Koszul dual
\(\Lambda^!\). Then, for each \(x\in Q_0\) and \(n\in\mathbb{Z}\), the complex
\(K\bigl(I^{!}_{x}\langle n\rangle\bigr)\) is a minimal \emph{graded} projective
resolution of \(S_{x}\langle -n\rangle[n]\), that is, of the simple graded
\(\Lambda\)-module \(S_x\langle -n\rangle\) placed in cohomological degree
\(-n\).
\end{lemma}
Note that \(I^{!}_x\) is concentrated in degrees \(\le 0\), with socle
\(S^{!}_x\) in degree \(0\); hence \(K(I^{!}_x\langle n\rangle)^m\neq0\) only
for \(m\le -n\), and its term in degree \(-n\) is \(P_x\langle -n\rangle\).

The next proposition shows that, for a finite\mbox{-}dimensional Koszul algebra
\(\Lambda\), the right bounded graded \(\Lambda^{!}\)\nobreakdash-modules with
bounded cohomology coincide with the coperfect
\(\Lambda^{!}\)\nobreakdash-modules.

\begin{proposition}\label{prop:cop}
Let \(\Lambda\) be a finite\mbox{-}dimensional Koszul algebra. Then
\(\Lambda^{!}\textup{-}\mathrm{Cop}^{\mathbb{Z}}=\Lambda^{!}\textup{-}\mathrm{GMod}^{-,b}\).
In particular, \(\Lambda^{!}\textup{-}\mathrm{Cop}^{\mathbb{Z}}\) is an exact,
weakly idempotent complete subcategory of
\(\Lambda^{!}\textup{-}\mathrm{GMod}^{-}\) and has enough strongly injective
modules.
\end{proposition}

\begin{proof}
Suppose first that \(M \in \Lambda^{!}\textup{-}\mathrm{Cop}^{\mathbb{Z}}\). By
definition, there exists a finite graded injective coresolution
\(0 \to M \to I^{!0} \to I^{!1} \to\cdots\to I^{!k} \to0\) in which every
\(I^{!i}\) is finitely cogenerated. By Lemma~\ref{lem:koszul-res}, each complex
\(K(I^{!i})\) has bounded cohomology. Since
\(K \colon \Lambda^{!}\textup{-}\mathrm{GMod} \to \mathcal{LC}\bigl(\Lambda\textup{-}\mathrm{proj}^{\mathbb{Z}}\bigr)\)
is exact, applying \(K\) to this coresolution and splitting it into short
exact sequences shows that \(K(M)\) has bounded cohomology. Hence
\(M\in \Lambda^{!}\textup{-}\mathrm{GMod}^{-,b}\).

Conversely, let \(M\in \Lambda^{!}\textup{-}\mathrm{GMod}^{-,b}\). Because
\(M\) is right bounded, its graded socle \(\operatorname{soc}(M)\) is
essential in \(M\). Moreover, there exists \(m\in\mathbb{Z}\) with
\(H^{k}(K(M))=0\) for all \(k<m\). Assume, for a contradiction, that
\(\operatorname{soc}(M)\) is infinite\mbox{-}dimensional. Since \(M\) is right
bounded and locally finite-dimensional, and \(Q_0\) is finite, there are then a
vertex \(x\in Q_{0}\) and a degree \(n<m\) with
\(\operatorname{soc}(M)_{n}(x)\neq 0\); choose
\(0\neq u\in\operatorname{soc}(M)_{n}(x)\). The summand of \(K(M)^n\) at \(x\)
is \(P_{x}\langle n\rangle\otimes_k M_{n}(x)\), and by the description of the
socle,
\(d^{n}(e_{x}\otimes u)=\sum_{\alpha\colon y\to x}P_{\alpha}(e_{x})\otimes M(\alpha^{\mathrm{op}})(u)=0\).
As \(H^{n}(K(M))=0\), the element \(e_{x}\otimes u\) lies in
\(\operatorname{im} d^{n-1}\), and therefore, since \(K(M)\) is linear, in the
radical of \(K(M)^{n}\). This is impossible, since \(u\neq0\) and
\(e_{x}\otimes u\) lies in the top of \(P_{x}\langle n\rangle\otimes_k M_{n}(x)\).
Thus \(\operatorname{soc}(M)\) is finite\mbox{-}dimensional.

Consequently, \(M\) admits an injective hull \(I^{!0}\) that is finitely
cogenerated. Since \(I^{!0}\) lies in \(\Lambda^{!}\textup{-}\mathrm{GMod}^{-,b}\)
by the first part, so does the cokernel of \(M\to I^{!0}\), by the long exact
cohomology sequence. Applying the same argument to this cokernel and iterating,
we obtain a minimal graded injective coresolution of \(M\) whose terms are
finitely cogenerated. Since \(\Lambda^{!}\) has finite global dimension \(d\)
by Theorem~\ref{thm:fd-gldim}, the \(d\)-th cokernel is injective, so this
coresolution is finite. Hence \(M\in \Lambda^{!}\textup{-}\mathrm{Cop}^{\mathbb{Z}}\).
\end{proof}

\begin{remark}\label{rem:cocoherent}
The category \(\Lambda^{!}\textup{-}\mathrm{Cop}^{\mathbb{Z}}\) (resp.\
\(\Lambda^{!}\textup{-}\mathrm{Pe}^{\mathbb{Z}}\)) is abelian if and only if
the algebra \(\Lambda^{!}\) is graded cocoh\-erent (resp.\ graded coherent).
This need not be the case. For instance, let
\(A=k\langle x,y,z\rangle/(yz-zy,\ zx)\), which is not graded coherent by
Polishchuk~\cite[Example~2]{37}. With respect to the degree-lexicographic order
with \(z>y>x\), the relations form a quadratic Gröbner basis with leading words
\(zy\) and \(zx\), which have no overlaps; hence \(A\) is Koszul, with Hilbert
series \(1/(1-3t+2t^2)\). Consequently \(\Lambda:=A^{!}\) is a
finite-dimensional Koszul algebra, with Hilbert series \(1+3t+2t^2\), and
\(\Lambda^{!}\cong A\). By graded duality, \(\Lambda^{!}\) is graded left
cocoherent if and only if it is graded right coherent; replacing \(A\) by
\(A^{\mathrm{op}}\) if necessary, we obtain a finite-dimensional Koszul algebra
\(\Lambda\) whose Koszul dual is not graded cocoherent. For such algebras
\(\Lambda^{!}\textup{-}\mathrm{Cop}^{\mathbb{Z}}\) is not abelian, and the
results of this section apply nevertheless.
\end{remark}

The following lemma compares finite-dimensional and coperfect modules.

\begin{lemma}\label{lem:gmod-cop}
Let \(\Lambda\) be a finite-dimensional Koszul algebra. The inclusion
\(\Lambda^{!}\textup{-gmod}\subseteq\Lambda^{!}\textup{-}\mathrm{Cop}^{\mathbb{Z}}\)
induces a fully faithful triangulated functor
\(\mathsf{D}^b(\Lambda^{!}\textup{-gmod})\to\mathsf{D}^b(\Lambda^{!}\textup{-}\mathrm{Cop}^{\mathbb{Z}})\),
whose essential image is the thick subcategory
\(\mathsf{D}^b_{\Lambda^{!}\textup{-gmod}}(\Lambda^{!}\textup{-}\mathrm{Cop}^{\mathbb{Z}})\)
of complexes with cohomology in \(\Lambda^{!}\textup{-gmod}\).
\end{lemma}

\begin{proof}
First, \(\Lambda^{!}\textup{-gmod}\subseteq\Lambda^{!}\textup{-}\mathrm{GMod}^{-,b}=\Lambda^{!}\textup{-}\mathrm{Cop}^{\mathbb{Z}}\), since \(K(M)\) is bounded for finite-dimensional \(M\). By Propositions~\ref{prop:strong-inj-ff} and~\ref{prop:cop}, the functor \(\mathsf{D}^b(\Lambda^{!}\textup{-}\mathrm{Cop}^{\mathbb{Z}})\to\mathsf{D}^b(\Lambda^{!}\textup{-}\mathrm{GMod}^{-})\) is fully faithful. We check the criterion of~\cite[Theorem~12.1]{27} for \(\Lambda^{!}\textup{-gmod}\subseteq\Lambda^{!}\textup{-}\mathrm{GMod}^{-}\): given an epimorphism \(g\colon M\to N\) in \(\Lambda^{!}\textup{-}\mathrm{GMod}^{-}\) with \(N\) finite-dimensional, we need a finite-dimensional \(M'\) and a morphism \(M'\to M\) whose composite with \(g\) is an epimorphism. Choose a finite-dimensional graded subspace \(V\subseteq M\) with \(g(V)=N\), let \(a\) be the smallest degree occurring in \(V\), and put \(M'=\Lambda^{!}V\). Since \(\Lambda^{!}\) is positively graded, \(M'\subseteq M_{\ge a}\), and \(M_{\ge a}\) is finite-dimensional since \(M\) is right bounded and locally finite-dimensional and \(Q_0\) is finite. Hence \(\mathsf{D}^b(\Lambda^{!}\textup{-gmod})\to\mathsf{D}^b(\Lambda^{!}\textup{-}\mathrm{GMod}^{-})\) is fully faithful, and so is \(\mathsf{D}^b(\Lambda^{!}\textup{-gmod})\to\mathsf{D}^b(\Lambda^{!}\textup{-}\mathrm{Cop}^{\mathbb{Z}})\).

If \(X\in\mathsf{D}^b(\Lambda^{!}\textup{-}\mathrm{Cop}^{\mathbb{Z}})\) has all its cohomology objects in \(\Lambda^{!}\textup{-gmod}\), then in \(\mathsf{D}^b(\Lambda^{!}\textup{-}\mathrm{GMod}^{-})\) it lies in the triangulated subcategory generated by these objects (use truncations). This subcategory is contained in the essential image of \(\mathsf{D}^b(\Lambda^{!}\textup{-gmod})\), which is closed under cones since the functor is fully faithful. As \(\mathsf{D}^b(\Lambda^{!}\textup{-}\mathrm{Cop}^{\mathbb{Z}})\to\mathsf{D}^b(\Lambda^{!}\textup{-}\mathrm{GMod}^{-})\) is fully faithful, \(X\) lies in the essential image of \(\mathsf{D}^b(\Lambda^{!}\textup{-gmod})\).
\end{proof}

We shall need two further preparatory results. We begin with an elementary
observation.

\begin{lemma}\label{lem:hom-shift}
Let \(\Lambda\) be a finite-dimensional graded algebra. Then
\(\operatorname{Hom}_{\Lambda\textup{-}\mathrm{gmod}}\bigl(P_{x}\langle i+1\rangle,P_{y}\langle i\rangle\bigr)=0\)
for all \(x,y\in Q_0\) and \(i\in\mathbb{Z}\).
\end{lemma}

\begin{proof}
By Proposition~\ref{prop:hom-proj}, we have
\(\operatorname{Hom}_{\Lambda\textup{-}\mathrm{gmod}}\bigl(P_x\langle i+1\rangle,P_y\langle i\rangle\bigr)\cong e_x\Lambda_{-1}e_y\),
which is zero since \(\Lambda\) is positively graded.
\end{proof}

\begin{lemma}\label{lem:no-homotopy}
Let \(\Lambda\) be a finite-dimensional quadratic algebra, and let
\(X^\bullet, Y^\bullet \in \mathcal{LC}\bigl(\Lambda\textup{-proj}^{\mathbb{Z}}\bigr)\).
Then every homotopy between two morphisms of complexes
\(X^\bullet \to Y^\bullet\) is zero. Consequently, there are no nonzero
null-homotopic morphisms between linear complexes, and
\(\operatorname{Hom}_{\mathsf{K}(\Lambda\textup{-proj}^{\mathbb{Z}})}(X^\bullet,Y^\bullet)
=\operatorname{Hom}_{\mathsf{C}(\Lambda\textup{-proj}^{\mathbb{Z}})}(X^\bullet,Y^\bullet)\).
\end{lemma}

\begin{proof}
A homotopy consists of morphisms \(h^n\colon X^n\to Y^{n-1}\) of graded
\(\Lambda\)-modules, \(n\in\mathbb{Z}\). Since \(X^\bullet\) and \(Y^\bullet\)
are linear, \(X^n\) is a finite direct sum of modules
\(P_x\langle n\rangle\) and \(Y^{n-1}\) is a finite direct sum of modules
\(P_y\langle n-1\rangle\), with \(x,y\in Q_0\). Hence \(h^n=0\) by
Lemma~\ref{lem:hom-shift}.
\end{proof}

We are now in a position to state the main result of this section. We shall
refer to the equivalence \(\mathfrak{F}\) below as the \emph{graded derived
Koszul duality}. It may be viewed as the bounded analogue
of~\cite[Theorem~30]{34}. Here
\(\mathfrak{F}'\colon\mathsf{D}^b(\Lambda\textup{-gmod})\to\mathsf{D}^b(\Lambda^{!}\textup{-}\mathrm{Pe}^{\mathbb{Z}})\)
denotes the derived quadratic functor in the opposite direction, obtained by
the same construction with the roles of \(\Lambda\) and \(\Lambda^{!}\)
exchanged: it sends a finite-dimensional graded \(\Lambda\)-module to a bounded
linear complex of finitely generated graded projective \(\Lambda^{!}\)-modules,
so it is defined for every quadratic algebra \(\Lambda\).

\begin{theorem}\label{thm:graded-kd}
Let \(\Lambda\) be a finite-dimensional quadratic algebra with quadratic dual
\(\Lambda^{!}\). The following conditions are equivalent:
\begin{enumerate}
\item \(\Lambda\) is Koszul.
\item \(\Lambda^{!}\textup{-}\mathrm{Cop}^{\mathbb{Z}}=\Lambda^{!}\textup{-}\mathrm{GMod}^{-,b}\),
and the derived quadratic functor
\[
\mathfrak{F}\colon
\mathsf{D}^b\bigl(\Lambda^{!}\textup{-}\mathrm{Cop}^{\mathbb{Z}}\bigr)
\;\longrightarrow\;
\mathsf{D}^b\bigl(\Lambda\textup{-gmod}\bigr)
\]
is an equivalence of triangulated categories.
\item The derived quadratic functor
\[
\mathfrak{F}'\colon
\mathsf{D}^b\bigl(\Lambda\textup{-gmod}\bigr)
\;\longrightarrow\;
\mathsf{D}^b\bigl(\Lambda^{!}\textup{-}\mathrm{Pe}^{\mathbb{Z}}\bigr)
\]
is an equivalence of triangulated categories.
\end{enumerate}
\end{theorem}

\begin{proof}
(1)\(\Rightarrow\)(2). Assume that \(\Lambda\) is Koszul. Then
\(\Lambda^{!}\textup{-}\mathrm{Cop}^{\mathbb{Z}}=\Lambda^{!}\textup{-}\mathrm{GMod}^{-,b}\)
by Proposition~\ref{prop:cop}. By definition, every coperfect module has a
finite coresolution by objects of
\(\Lambda^{!}\textup{-}\mathrm{Inj}^{\mathbb{Z}}\); hence the triangulated
category \(\mathsf{D}^b(\Lambda^{!}\textup{-}\mathrm{Cop}^{\mathbb{Z}})\) is
generated by the graded injectives \(I_x^{!}\langle n\rangle\),
\(x\in Q_0\), \(n\in\mathbb{Z}\). By Lemmas~\ref{lem:koszul-res}
and~\ref{lem:no-homotopy}, together with Proposition~\ref{prop:K-equiv}, the
functor \(\mathfrak{F}\) induces a chain of canonical isomorphisms
\[
\begin{aligned}
\operatorname{Hom}_{\mathsf{D}^b(\Lambda^{!}\textup{-}\mathrm{Cop}^{\mathbb{Z}})}
\bigl(I_x^{!}\langle n\rangle,\,I_y^{!}\langle m\rangle\bigr)
&\;\cong\;
\operatorname{Hom}_{\Lambda^{!}\textup{-}\mathrm{Cop}^{\mathbb{Z}}}
\bigl(I_x^{!}\langle n\rangle,\,I_y^{!}\langle m\rangle\bigr)
\\
&\;\cong\;
\operatorname{Hom}_{\mathcal{LC}^{-,b}(\Lambda\textup{-proj}^{\mathbb{Z}})}
\bigl(K(I_x^{!}\langle n\rangle),\,K(I_y^{!}\langle m\rangle)\bigr)
\\
&\;\cong\;
\operatorname{Hom}_{\mathsf{K}^{-,b}(\Lambda\textup{-proj}^{\mathbb{Z}})}
\bigl(K(I_x^{!}\langle n\rangle),\,K(I_y^{!}\langle m\rangle)\bigr)
\\
&\;\cong\;
\operatorname{Hom}_{\mathsf{D}^b(\Lambda\textup{-gmod})}
\bigl(S_x\langle -n\rangle[n],\,S_y\langle -m\rangle[m]\bigr).
\end{aligned}
\]
Moreover, for \(i\neq0\) we have
\(\operatorname{Hom}_{\mathsf{D}^b(\Lambda^{!}\textup{-}\mathrm{Cop}^{\mathbb{Z}})}(I^{!}_x\langle n\rangle,I^{!}_y\langle m\rangle[i])=0\),
since \(I^{!}_y\langle m\rangle\) is strongly injective
(Propositions~\ref{prop:strong-inj-ff} and~\ref{prop:cop}), while
\[
\operatorname{Hom}_{\mathsf{D}^b(\Lambda\textup{-gmod})}\bigl(S_x\langle -n\rangle[n],S_y\langle -m\rangle[m+i]\bigr)
\cong\operatorname{Ext}^{m-n+i}_{\Lambda\textup{-gmod}}\bigl(S_x,S_y\langle -(m-n)\rangle\bigr)=0
\]
by Proposition~\ref{prop:koszul-ext}, because \(m-n+i\neq m-n\). Hence
\(\mathfrak{F}\) is fully faithful on the set
\(\{I^{!}_x\langle n\rangle[i]\}\), which generates
\(\mathsf{D}^b(\Lambda^{!}\textup{-}\mathrm{Cop}^{\mathbb{Z}})\), and therefore
\(\mathfrak{F}\) is fully faithful. Furthermore, the essential image of
\(\mathfrak{F}\) contains all simples \(S_z\langle \ell\rangle\), which generate
\(\mathsf{D}^b(\Lambda\textup{-gmod})\) as a triangulated category. Thus
\(\mathfrak{F}\) is essentially surjective, hence an equivalence.

(2)\(\Rightarrow\)(1). Assume (2). For \(x\in Q_0\) and \(n\in\mathbb{Z}\),
the complex \(K(S^{!}_x\langle n\rangle)\) is the stalk complex
\(P_x\langle -n\rangle\) concentrated in degree \(-n\). Hence
\(S^{!}_x\langle n\rangle\in\Lambda^{!}\textup{-}\mathrm{GMod}^{-,b}=\Lambda^{!}\textup{-}\mathrm{Cop}^{\mathbb{Z}}\)
and \(\mathfrak{F}(S^{!}_x\langle n\rangle)=P_x\langle -n\rangle[n]\). Since
\(\mathfrak{F}\) is fully faithful, we obtain, for all \(x,y\in Q_0\) and
\(m,n,i\in\mathbb{Z}\),
\[
\begin{aligned}
\operatorname{Hom}_{\mathsf{D}^b(\Lambda^{!}\textup{-}\mathrm{Cop}^{\mathbb{Z}})}
\bigl(S^{!}_x\langle n\rangle,\,S^{!}_y\langle m\rangle[i]\bigr)
&\;\cong\;
\operatorname{Hom}_{\mathsf{D}^b(\Lambda\textup{-gmod})}
\bigl(P_x\langle -n\rangle[n],\,P_y\langle -m\rangle[m+i]\bigr)
\\
&\;\cong\;
\operatorname{Hom}_{\mathsf{K}^b(\Lambda\textup{-proj}^{\mathbb{Z}})}
\bigl(P_x\langle -n\rangle,\,P_y\langle -m\rangle[m-n+i]\bigr),
\end{aligned}
\]
which vanishes for \(i\neq n-m\). By definition,
\(S^{!}_y\langle m\rangle\) has a finite coresolution in
\(\Lambda^{!}\textup{-}\mathrm{Cop}^{\mathbb{Z}}\) by finitely cogenerated
graded injectives, which are injective in \(\Lambda^{!}\textup{-}\mathrm{GMod}\);
hence the left-hand side is
\(\operatorname{Ext}^{i}_{\Lambda^{!}\textup{-}\mathrm{GMod}}\bigl(S^{!}_x\langle n\rangle,S^{!}_y\langle m\rangle\bigr)\).
Putting \(j=n-m\), we get
\(\operatorname{Ext}^{i}_{\Lambda^{!}\textup{-}\mathrm{GMod}}\bigl(S^{!}_x,S^{!}_y\langle -j\rangle\bigr)=0\)
for \(i\neq j\), so \(\Lambda^{!}\) is Koszul by
Proposition~\ref{prop:koszul-ext} applied to \(\Lambda^{!}\). Since
\(\Lambda^{!}\) is Koszul, so is its quadratic dual
\((\Lambda^{!})^{!}\)~\cite[Proposition~2.9.1]{6}; as
\((\Lambda^{!})^{!}\cong\Lambda\), the algebra \(\Lambda\) is Koszul. 

(1)\(\Rightarrow\)(3). Assume that \(\Lambda\) is Koszul. The algebra
isomorphism
\(\Lambda^{!}\cong\bigoplus_{n\ge0}\operatorname{Ext}^{n}_{\Lambda}(\Lambda_0,\Lambda_0)\)
of~\((\ast)\)~\cite[Theorem~2.10.1]{6}, together with
Proposition~\ref{prop:hom-proj}, gives, for all \(x,y\in Q_0\) and
\(m,n\in\mathbb{Z}\),
\[
\operatorname{Ext}^{n-m}_{\Lambda\textup{-gmod}}
\bigl(S_y\langle n\rangle,S_x\langle m\rangle\bigr)
\cong
e_y\,\Lambda^{!}_{n-m}\,e_x
\cong
\operatorname{Hom}_{\Lambda^{!}\textup{-}\mathrm{Pe}^{\mathbb{Z}}}
\bigl(P_y^{!}\langle -n\rangle,P_x^{!}\langle -m\rangle\bigr),
\]
and all other extension groups between simples vanish by
Proposition~\ref{prop:koszul-ext}. The functor \(\mathfrak{F}'\) sends
\(S_x\langle m\rangle\) to \(P^{!}_x\langle -m\rangle[m]\), and the map it
induces between these spaces is surjective, hence bijective: in degree~\(1\)
it sends the extension given by the action of an arrow \(\alpha\colon y\to x\)
to \(P_{\alpha^{\mathrm{op}}}\), and \(\Lambda^{!}\) is generated in
degree~\(1\). Hence \(\mathfrak{F}'\) is fully faithful on the simples, and
therefore fully faithful; its essential image contains the generators
\(P^{!}_x\langle n\rangle\), so it is an equivalence.

(3)\(\Rightarrow\)(1). Assume that \(\mathfrak{F}'\) is an equivalence. By
construction, \(\mathfrak{F}'\) sends \(S_y\langle -i\rangle\) to the stalk
complex \(P^{!}_y\langle i\rangle[-i]\). Since
\(\mathsf{D}^b(\Lambda^{!}\textup{-}\mathrm{Pe}^{\mathbb{Z}})\simeq\mathsf{K}^b(\Lambda^{!}\textup{-}\mathrm{Proj}^{\mathbb{Z}})\),
full faithfulness of \(\mathfrak{F}'\) gives
\[
\operatorname{Ext}^j_{\Lambda\textup{-gmod}}\bigl(S_x,S_y\langle -i\rangle\bigr)\cong\operatorname{Hom}_{\mathsf{K}^b(\Lambda^{!}\textup{-}\mathrm{Proj}^{\mathbb{Z}})}\bigl(P^{!}_x,P^{!}_y\langle i\rangle[j-i]\bigr)=0
\qquad\text{for } j\neq i,
\]
and \(\Lambda\) is Koszul by Proposition~\ref{prop:koszul-ext}.
\end{proof}
\begin{remark}\label{rem:bgs}
Up to our conventions, the equivalence \(\mathfrak{F}\) of
Theorem~\ref{thm:graded-kd} is the restriction of the Koszul duality
\(\mathsf{D}^{\downarrow}(\Lambda^{!})\simeq\mathsf{D}^{\uparrow}(\Lambda)\)
of~\cite[Theorem~2.12.1]{6} and~\cite[Theorem~5.7]{11}. In~\cite{6}, the
subcategory corresponding to \(\mathsf{D}^b(\Lambda\textup{-gmod})\) is
described as a category of complexes with finitely many finitely generated
cohomology groups, and it is identified with a bounded derived category of
graded \(\Lambda^{!}\)-modules only when \(\Lambda^{!}\) is left
Noetherian~\cite[Lemma~2.12.8 and Theorem~2.12.6]{6}.
Theorem~\ref{thm:graded-kd} identifies it with
\(\mathsf{D}^b(\Lambda^{!}\textup{-}\mathrm{Cop}^{\mathbb{Z}})\) in general.
\end{remark}

If \(\Lambda\) is a finite-dimensional Koszul algebra and \(\Lambda^{!}\) is
left graded coherent, then, since \(\Lambda^{!}\) has finite global dimension
(Theorem~\ref{thm:fd-gldim}), every finitely presented graded
\(\Lambda^{!}\)-module is perfect, that is,
\(\Lambda^{!}\textup{-}\mathrm{Pe}^{\mathbb{Z}}=\Lambda^{!}\textup{-}\mathrm{Fp}^{\mathbb{Z}}\).
In this situation, Theorem~\ref{thm:graded-kd} recovers the result of
Martínez-Villa and Saorín~\cite[Proposition~4.1]{33}.

\begin{corollary}\label{cor:coherent}
Let \(\Lambda\) be a finite-dimensional Koszul algebra, and suppose that its
Koszul dual \(\Lambda^{!}\) is left graded coherent. Then the functor
\(\mathfrak{F}'\) of Theorem~\ref{thm:graded-kd} induces a triangulated
equivalence
\(\mathsf{D}^b\bigl(\Lambda\textup{-gmod}\bigr)\simeq\mathsf{D}^b\bigl(\Lambda^{!}\textup{-}\mathrm{Fp}^{\mathbb{Z}}\bigr)\).
\end{corollary}

If \(\Lambda^{!}\) is left Noetherian, then
\(\Lambda^{!}\textup{-}\mathrm{Fg}^{\mathbb{Z}}=\Lambda^{!}\textup{-}\mathrm{Fp}^{\mathbb{Z}}\),
and we recover the Koszul duality of Beilinson, Ginzburg and
Soergel~\cite[Theorem~2.12.6]{6}.

\begin{corollary}\label{cor:noetherian}
Let \(\Lambda\) be a finite-dimensional Koszul algebra, and suppose that its
Koszul dual \(\Lambda^{!}\) is left Noetherian. Then there is a triangulated
equivalence
\(\mathsf{D}^b\bigl(\Lambda\textup{-gmod}\bigr)\simeq\mathsf{D}^b\bigl(\Lambda^{!}\textup{-}\mathrm{Fg}^{\mathbb{Z}}\bigr)\).
\end{corollary}

\subsection{Special cases of derived Koszul duality}

We record two further consequences of the graded derived Koszul duality. In
particular, we determine when there exists a triangulated equivalence
\(\mathsf{D}^b(\Lambda^{!}\textup{-gmod})\simeq\mathsf{D}^b(\Lambda\textup{-gmod})\).

\begin{corollary}\label{cor:special}
Let \(\Lambda\) be a finite-dimensional Koszul algebra.
\begin{enumerate}
\item The functors \(\mathfrak{F}\) and \(\mathfrak{F}'\) of
Theorem~\ref{thm:graded-kd} restrict to triangulated equivalences
\[
\mathfrak{F}\colon
\mathsf{D}^{b}\bigl(\Lambda^{!}\textup{-gmod}\bigr)
\xrightarrow{\;\sim\;}
\mathsf{D}^{b}\bigl(\Lambda\textup{-}\mathrm{pe}^{\mathbb{Z}}\bigr),
\qquad
\mathfrak{F}'\colon
\mathsf{D}^{b}\bigl(\Lambda\textup{-}\mathrm{cop}^{\mathbb{Z}}\bigr)
\xrightarrow{\;\sim\;}
\mathsf{D}^{b}\bigl(\Lambda^{!}\textup{-gmod}\bigr),
\]
where \(\Lambda\textup{-}\mathrm{pe}^{\mathbb{Z}}\) (resp.\
\(\Lambda\textup{-}\mathrm{cop}^{\mathbb{Z}}\)) denotes the full subcategory of
finite-dimensional graded \(\Lambda\)-modules of finite projective dimension
(resp.\ finite injective dimension).
\item The functor \(\mathfrak{F}\) restricts to a triangulated equivalence
\(\mathsf{D}^{b}(\Lambda^{!}\textup{-gmod})\simeq\mathsf{D}^{b}(\Lambda\textup{-gmod})\)
if and only if \(\Lambda\) has finite global dimension.
\end{enumerate}
\end{corollary}

\begin{proof}
\emph{(1)} By Lemma~\ref{lem:gmod-cop},
\(\mathsf{D}^{b}(\Lambda^{!}\textup{-gmod})\) is a full triangulated
subcategory of \(\mathsf{D}^{b}(\Lambda^{!}\textup{-}\mathrm{Cop}^{\mathbb{Z}})\),
and it is generated by the simple modules \(S_x^{!}\langle n\rangle\),
\(x\in Q_0\), \(n\in\mathbb{Z}\). Since \(S_x^{!}\langle n\rangle\) is
concentrated in degree \(-n\), we have
\(\mathfrak{F}(S_x^{!}\langle n\rangle)=P_x\langle -n\rangle[n]\). Hence
\(\mathfrak{F}\) restricts to an equivalence between
\(\mathsf{D}^{b}(\Lambda^{!}\textup{-gmod})\) and the triangulated subcategory
of \(\mathsf{D}^{b}(\Lambda\textup{-gmod})\) generated by the graded
projectives \(P_y\langle m\rangle\), which is
\(\mathsf{K}^{b}(\Lambda\textup{-proj}^{\mathbb{Z}})\simeq\mathsf{D}^{b}(\Lambda\textup{-}\mathrm{pe}^{\mathbb{Z}})\).

Similarly,
\(\mathsf{D}^{b}(\Lambda\textup{-}\mathrm{cop}^{\mathbb{Z}})\simeq\mathsf{K}^{b}(\Lambda\textup{-inj}^{\mathbb{Z}})\)
is the full triangulated subcategory of \(\mathsf{D}^{b}(\Lambda\textup{-gmod})\)
generated by the graded injectives \(I_x\langle n\rangle\). By
Lemma~\ref{lem:koszul-res}, applied with the roles of \(\Lambda\) and
\(\Lambda^{!}\) exchanged (recall that \(\Lambda^{!}\) is Koszul and
\((\Lambda^{!})^{!}\cong\Lambda\); see also~\cite[Theorem~2.6.1]{6}),
\(\mathfrak{F}'(I_x\langle n\rangle)\) is a projective resolution of
\(S^{!}_x\langle -n\rangle[n]\). In particular, every finite-dimensional graded
\(\Lambda^{!}\)-module, being an iterated extension of simple modules, is
perfect, and the dual of the argument in Lemma~\ref{lem:gmod-cop}, using the
finite-dimensional quotients \(M/M_{>a}\), shows that
\(\mathsf{D}^{b}(\Lambda^{!}\textup{-gmod})\to\mathsf{D}^{b}(\Lambda^{!}\textup{-}\mathrm{Pe}^{\mathbb{Z}})\)
is fully faithful. Since \(\mathfrak{F}'\) is an equivalence, it restricts to an
equivalence between \(\mathsf{D}^{b}(\Lambda\textup{-}\mathrm{cop}^{\mathbb{Z}})\)
and the triangulated subcategory generated by the simple modules
\(S^{!}_x\langle n\rangle\), that is, \(\mathsf{D}^{b}(\Lambda^{!}\textup{-gmod})\).

\emph{(2)} Suppose first that \(\mathfrak{F}\) induces a triangulated
equivalence
\(\mathsf{D}^{b}(\Lambda^{!}\textup{-gmod})\xrightarrow{\sim}\mathsf{D}^{b}(\Lambda\textup{-gmod})\).
By~\emph{(1)}, the essential image of \(\mathsf{D}^{b}(\Lambda^{!}\textup{-gmod})\)
under \(\mathfrak{F}\) is \(\mathsf{D}^{b}(\Lambda\textup{-}\mathrm{pe}^{\mathbb{Z}})\);
hence
\(\mathsf{D}^{b}(\Lambda\textup{-gmod})=\mathsf{D}^{b}(\Lambda\textup{-}\mathrm{pe}^{\mathbb{Z}})\),
that is, every finite-dimensional graded \(\Lambda\)-module has finite
projective dimension, and thus \(\Lambda\) has finite global dimension.

Conversely, assume that \(\Lambda\) has finite global dimension. Applying
Theorem~\ref{thm:fd-gldim} to the Koszul algebra \(\Lambda^{!}\), whose Koszul
dual is \(\Lambda\), we see that \(\Lambda^{!}\) is finite-dimensional; it has
finite global dimension by Theorem~\ref{thm:fd-gldim} applied to \(\Lambda\).
Consequently, every graded injective \(\Lambda^{!}\)-module
\(I_x^{!}\langle n\rangle\) is finite-dimensional, so
\(\Lambda^{!}\textup{-}\mathrm{Cop}^{\mathbb{Z}}=\Lambda^{!}\textup{-gmod}\).
Therefore, in this case the graded derived Koszul duality identifies
\(\mathsf{D}^{b}(\Lambda^{!}\textup{-gmod})\) with
\(\mathsf{D}^{b}(\Lambda\textup{-gmod})\), as required.
\end{proof}

Let \(\Lambda\) be a finite-dimensional algebra whose quiver \(Q\) is well
directed (see Section~\ref{sec:prelim}), with levels
\(Q_0=Q_0^{(1)}\sqcup\cdots\sqcup Q_0^{(r)}\). Since all paths between two
vertices have the same length, every \(\Lambda\)-module \(M\) has a unique
grading with \(M(x)\) in degree \(i\) for \(x\in Q_0^{(i)}\), and every
\(\Lambda^{!}\)-module \(M\) has a unique grading with \(M(x)\) in degree
\(-i\) for \(x\in Q_0^{(i)}\). Hence, by Proposition~\ref{prop:koszul-ext},
\(\Lambda\) is Koszul if and only if, for all \(x\in Q_0^{(i)}\),
\(y\in Q_0^{(j)}\) and \(n\ge0\),
\(\operatorname{Ext}^{n}_{\Lambda\textup{-mod}}(S_x,S_y)=0\) whenever
\(n\neq j-i\).

A second natural question is whether the graded derived Koszul duality can be
upgraded to an equivalence between the ungraded bounded derived categories
\(\mathsf{D}^b(\Lambda^{!}\textup{-mod})\) and
\(\mathsf{D}^b(\Lambda\textup{-mod})\). In general such an equivalence does not
exist; we return to this issue in Section~\ref{sec:nongraded}. Nevertheless, in
the well-directed case one obtains the following consequence.

\begin{corollary}\label{cor:well-directed}
Let \(\Lambda\) be a finite-dimensional Koszul algebra whose quiver is well
directed. Then the derived quadratic functor induces a triangulated
equivalence
\[
\mathfrak{F}\colon
\mathsf{D}^b\bigl(\Lambda^{!}\textup{-mod}\bigr)
\;\xrightarrow{\;\sim\;}
\mathsf{D}^b\bigl(\Lambda\textup{-mod}\bigr).
\]
\end{corollary}

\begin{proof}
Since \(Q\) has no oriented cycles, \(\Lambda\) has finite global dimension,
so, by Theorem~\ref{thm:fd-gldim}, \(\Lambda^{!}\) is finite-dimensional and
has finite global dimension. By the gradings above, the quadratic functor is
defined on ungraded modules, with
\(K(M)^n=\bigoplus_{x\in Q_0^{(-n)}}P_x\otimes_k M(x)\), and
Proposition~\ref{prop:K-equiv} holds for ungraded modules. Moreover,
\(\operatorname{Hom}_{\Lambda}(P_x,P_y)=0\) for \(x\in Q_0^{(i)}\) and
\(y\in Q_0^{(j)}\) with \(i<j\), since this space is spanned by the paths from
\(y\) to \(x\); this is the ungraded analogue of Lemma~\ref{lem:hom-shift}, and
hence the ungraded version of Lemma~\ref{lem:no-homotopy} holds. Likewise,
Lemma~\ref{lem:koszul-res} extends: for \(x\in Q_0^{(i)}\), \(K(I^{!}_x)\) is a
minimal projective resolution of the stalk complex \(S_x[i]\), that is, of
\(S_x\) placed in cohomological degree \(-i\). Together with the Ext criterion
above, the remainder of the argument is identical to that of the implication
(1)\(\Rightarrow\)(2) in Theorem~\ref{thm:graded-kd}, with
\(\Lambda^{!}\textup{-}\mathrm{Cop}^{\mathbb{Z}}\) replaced by
\(\Lambda^{!}\textup{-mod}\), which is generated by the finite-dimensional
injectives \(I^{!}_x\).
\end{proof}

\begin{remark}\label{rem:beilinson}
For example, the \(d\)-th Beilinson algebra \(B_d\) is Koszul, and its quiver,
with vertices \(0,\dots,d\) and \(d+1\) arrows from \(s\) to \(s+1\) for
\(0\le s<d\), is well directed. Corollary~\ref{cor:well-directed} therefore
gives
\(\mathsf{D}^b(B_d^{!}\textup{-mod})\simeq\mathsf{D}^b(B_d\textup{-mod})\simeq\mathsf{D}^b(\mathrm{coh}\,\mathbb{P}^d)\),
the last equivalence being Beilinson's theorem~\cite{5}.
\end{remark}

\subsection{Graded singular Koszul duality and the graded BGG correspondence}\label{sec:graded-sing}

In~\cite{8}, Bernstein, Gel'fand and Gel'fand constructed the first example of
Koszul duality and deduced the triangulated equivalence now known as the
\emph{BGG correspondence}: there is a triangulated equivalence between the
bounded derived category of coherent sheaves on projective space
\(\mathbb{P}^n\) and the stable category of finite-dimensional graded modules
over the exterior algebra \(E\)~\cite[Theorem~2]{8},
\(\mathsf{D}^b(\mathrm{coh}\,\mathbb{P}^n)\xrightarrow{\sim}E\textup{-}\underline{\mathrm{gmod}}\).
They also formulated a noncommutative analogue~\cite[Theorem~4]{8},
identifying \(\mathsf{D}^b(\mathrm{coh}\,\mathbb{P}^n)\) with the Verdier
quotient
\(\mathsf{D}^b(A\textup{-}\mathrm{Fg}^{\mathbb{Z}})/\mathsf{D}^b(A\textup{-gmod})\),
where \(A\) is the symmetric algebra and
\(\mathsf{D}^b(A\textup{-gmod})\) is identified with the thick subcategory of
complexes with finite-dimensional total cohomology. In particular, one obtains
an equivalence
\(\mathsf{D}^b(A\textup{-}\mathrm{Fg}^{\mathbb{Z}})/\mathsf{D}^b(A\textup{-gmod})\xrightarrow{\sim}E\textup{-}\underline{\mathrm{gmod}}\),
which follows from the Koszul derived equivalence
\(\mathsf{D}^b(A\textup{-}\mathrm{Fg}^{\mathbb{Z}})\xrightarrow{\sim}\mathsf{D}^b(E\textup{-gmod})\)
of~\cite[Theorem~3]{8}. In this example both algebras are Koszul: the exterior
algebra \(E\) is finite-dimensional, self-injective and Koszul, whereas its
Koszul dual \(A\) is a polynomial algebra in finitely many variables and has
finite global dimension.

Later, Mart\'{\i}nez-Villa and Saor\'{\i}n generalized the BGG correspondence
to finite-dimensional Koszul algebras whose Koszul dual is coherent. Their
results are not phrased in terms of singularity categories. Instead, they
proved that if \(\Lambda\) is a finite-dimensional Koszul algebra and
\(\Lambda^{!}\) is left coherent, then there is a triangulated equivalence
\(\mathsf{D}^b\bigl(\Lambda^{!}\textup{-}\mathrm{Fp}^{\mathbb{Z}}/\Lambda^{!}\textup{-gmod}\bigr)\xrightarrow{\sim}\mathsf{S}\bigl(\overline{\Lambda\textup{-gmod}}\bigr)\),
where \(\mathsf{S}(\overline{\Lambda\textup{-gmod}})\) denotes the
stabilization of the stable category of finite-dimensional graded
\(\Lambda\)-modules modulo injectives. The latter is a triangulated category,
canonically equivalent to the Verdier quotient
\(\mathsf{D}^b(\Lambda\textup{-gmod})/\mathsf{K}^b(\Lambda\textup{-inj}^{\mathbb{Z}})\);
see~\cite[Section~4]{33}.

\begin{remark}\label{rem:tails}
The quotient category
\(\Lambda^{!}\textup{-}\mathrm{Fp}^{\mathbb{Z}}/\Lambda^{!}\textup{-gmod}\) is
commonly referred to as the \emph{tails category} of \(\Lambda^{!}\). It is
abelian whenever \(\Lambda^{!}\) is coherent, and can be viewed as a
noncommutative analogue of \(\mathrm{coh}\,\mathbb{P}^n\).
\end{remark}

By Lemma~\ref{lem:gmod-cop} and Proposition~\ref{prop:cop},
\(\bigl(\Lambda^{!}\textup{-gmod},\ \Lambda^{!}\textup{-}\mathrm{Cop}^{\mathbb{Z}},\ \Lambda^{!}\textup{-}\mathrm{GMod}^{-}\bigr)\)
is a triple in the sense of Section~\ref{sec:prelim}, so the quotient
\(\Lambda^{!}\textup{-}\mathrm{Cop}^{\mathbb{Z}}/\Lambda^{!}\textup{-gmod}\) is
an exact, weakly idempotent complete category
(Proposition~\ref{prop:quot-wic}). Moreover,
\(\mathsf{D}^b(\Lambda^{!}\textup{-gmod})\simeq\mathsf{D}^b_{\Lambda^{!}\textup{-gmod}}(\Lambda^{!}\textup{-}\mathrm{Cop}^{\mathbb{Z}})\)
by Lemma~\ref{lem:gmod-cop}. It then follows from
Theorem~\ref{thm:quot-derived} that there is a triangulated equivalence
\[
\mathsf{D}^b\bigl(\Lambda^{!}\textup{-}\mathrm{Cop}^{\mathbb{Z}}\bigr)
\big/
\mathsf{D}^b_{\Lambda^{!}\textup{-gmod}}\bigl(\Lambda^{!}\textup{-}\mathrm{Cop}^{\mathbb{Z}}\bigr)
\;\xrightarrow{\ \sim\ }\;
\mathsf{D}^b\bigl(\Lambda^{!}\textup{-}\mathrm{Cop}^{\mathbb{Z}}/\Lambda^{!}\textup{-gmod}\bigr).
\]

\begin{remark}
If \(\Lambda^{!}\) is graded cocoherent, then the quotient category
\(\Lambda^{!}\textup{-}\mathrm{Cop}^{\mathbb{Z}}/\Lambda^{!}\textup{-gmod}\) is
abelian and, via the graded duality \(\mathbb{D}\), equivalent to the opposite
of the tails category of \((\Lambda^{!})^{\mathrm{op}}\) (see
Remark~\ref{rem:tails}).
\end{remark}

We are now ready to prove the second main result of this section, which we
call the \emph{graded singular Koszul duality}.

\begin{theorem}\label{thm:graded-sing}
Let \(\Lambda\) be a finite-dimensional Koszul algebra, and let
\(\Lambda^{!}\) be its Koszul dual. Then the graded derived Koszul duality
induces a triangulated equivalence
\[
\mathsf{D}^{b}\bigl(\Lambda^{!}\textup{-}\mathrm{Cop}^{\mathbb{Z}}/\Lambda^{!}\textup{-gmod}\bigr)
\;\xrightarrow{\ \sim\ }\;
\mathsf{D}_{\mathrm{sg}}\bigl(\Lambda\textup{-gmod}\bigr).
\]
\end{theorem}

\begin{proof}
By Lemma~\ref{lem:gmod-cop}, \(\mathsf{D}^b(\Lambda^{!}\textup{-gmod})\)
identifies with the thick subcategory
\(\mathsf{D}^b_{\Lambda^{!}\textup{-gmod}}(\Lambda^{!}\textup{-}\mathrm{Cop}^{\mathbb{Z}})\),
and by Corollary~\ref{cor:special}(1) the equivalence \(\mathfrak{F}\) of
Theorem~\ref{thm:graded-kd} maps it onto
\(\mathsf{K}^b(\Lambda\textup{-proj}^{\mathbb{Z}})\). Consider the commutative
diagram of triangulated categories
\[
\begin{tikzcd}
\mathsf{D}^{b}\bigl(\Lambda^{!}\textup{-}\mathrm{Cop}^{\mathbb{Z}}\bigr)
   \ar[d] \ar[r,"\mathfrak{F}","\sim"'] &
\mathsf{D}^{b}\bigl(\Lambda\textup{-gmod}\bigr) \ar[d] \\
\mathsf{D}^{b}\bigl(\Lambda^{!}\textup{-}\mathrm{Cop}^{\mathbb{Z}}\bigr)
   \big/ \mathsf{D}^b_{\Lambda^{!}\textup{-gmod}}\bigl(\Lambda^{!}\textup{-}\mathrm{Cop}^{\mathbb{Z}}\bigr)
   \ar[r,dashed] &
\mathsf{D}_{\mathrm{sg}}\bigl(\Lambda\textup{-gmod}\bigr),
\end{tikzcd}
\]
where the vertical arrows are the canonical Verdier quotient functors. Since
\(\mathfrak{F}\) is an equivalence mapping
\(\mathsf{D}^b_{\Lambda^{!}\textup{-gmod}}(\Lambda^{!}\textup{-}\mathrm{Cop}^{\mathbb{Z}})\)
onto \(\mathsf{K}^b(\Lambda\textup{-proj}^{\mathbb{Z}})\), it induces a
triangulated equivalence between the Verdier quotients, the dashed arrow.
Composing it with the inverse of the equivalence preceding the theorem, we
obtain the required triangulated equivalence, which we denote again by
\(\mathfrak{F}\).
\end{proof}

We now establish the following generalization of the BGG correspondence,
valid for all finite-dimensional Iwanaga--Gorenstein Koszul algebras.

\begin{theorem}\label{thm:graded-bgg}
Let \(\Lambda\) be a finite-dimensional Iwanaga--Gorenstein Koszul algebra,
and let \(\Lambda^{!}\) be its Koszul dual. Then the graded singular Koszul
duality induces a triangulated equivalence
\[
\mathsf{D}^b\bigl(\Lambda^{!}\textup{-}\mathrm{Cop}^{\mathbb{Z}}/\Lambda^{!}\textup{-gmod}\bigr)
\;\xrightarrow{\ \sim\ }\;
\Lambda\textup{-}\underline{\mathrm{Gproj}}^{\mathbb{Z}}.
\]
\end{theorem}

\begin{proof}
This follows immediately from Theorem~\ref{thm:graded-sing} together with
Buchweitz's equivalence~\cite[Theorem~4.4.1]{12} (see also
Orlov~\cite[Lemma~1.10, Proposition~1.11]{36}).
\end{proof}

We conclude this subsection with a characterization of finite-dimensional
Koszul Iwanaga--Gorenstein algebras.

\begin{theorem}\label{thm:ig-char}
Let \(\Lambda\) be a finite-dimensional Koszul algebra. The following
statements are equivalent:
\begin{enumerate}
\item \(\Lambda\) is Iwanaga--Gorenstein.
\item The functor
\(\mathfrak{F}'\colon\mathsf{D}^{b}(\Lambda\textup{-gmod})\to\mathsf{D}^{b}(\Lambda^{!}\textup{-}\mathrm{Pe}^{\mathbb{Z}})\)
of Theorem~\ref{thm:graded-kd} restricts to a triangulated equivalence
\(\mathsf{D}^{b}(\Lambda\textup{-}\mathrm{pe}^{\mathbb{Z}})\xrightarrow{\sim}\mathsf{D}^{b}(\Lambda^{!}\textup{-gmod})\),
and the graded derived Koszul duality
\(\mathfrak{F}\colon\mathsf{D}^{b}(\Lambda^{!}\textup{-}\mathrm{Cop}^{\mathbb{Z}})\to\mathsf{D}^{b}(\Lambda\textup{-gmod})\)
restricts to a triangulated equivalence
\(\mathsf{D}^{b}(\Lambda^{!}\textup{-gmod})\xrightarrow{\sim}\mathsf{D}^{b}(\Lambda\textup{-}\mathrm{cop}^{\mathbb{Z}})\).
\end{enumerate}
\end{theorem}

\begin{proof}
Assume first that \(\Lambda\) is Iwanaga--Gorenstein. Then a
finite-dimensional module has finite projective dimension if and only if it
has finite injective dimension, so
\(\mathsf{D}^{b}(\Lambda\textup{-}\mathrm{pe}^{\mathbb{Z}})=\mathsf{D}^{b}(\Lambda\textup{-}\mathrm{cop}^{\mathbb{Z}})\),
and Corollary~\ref{cor:special}(1) implies that the two functors restrict to
triangulated equivalences
\[
\mathsf{D}^{b}\bigl(\Lambda\textup{-}\mathrm{pe}^{\mathbb{Z}}\bigr)
=\mathsf{D}^{b}\bigl(\Lambda\textup{-}\mathrm{cop}^{\mathbb{Z}}\bigr)
\simeq
\mathsf{D}^{b}\bigl(\Lambda^{!}\textup{-gmod}\bigr)
\quad\text{and}\quad
\mathsf{D}^{b}\bigl(\Lambda^{!}\textup{-gmod}\bigr)
\simeq
\mathsf{D}^{b}\bigl(\Lambda\textup{-}\mathrm{pe}^{\mathbb{Z}}\bigr)=\mathsf{D}^{b}\bigl(\Lambda\textup{-}\mathrm{cop}^{\mathbb{Z}}\bigr),
\]
which yields~\emph{(2)}.

Conversely, assume~\emph{(2)}. Let \(P_x\) be an indecomposable graded
projective \(\Lambda\)-module generated in degree \(0\). By the first
equivalence in~\emph{(2)}, \(\mathfrak{F}'(P_x)\) is a bounded complex of
finitely generated graded projective \(\Lambda^{!}\)-modules whose cohomology
groups are finite-dimensional.
Since \(\mathfrak{F}'(S_y\langle -i\rangle)=P^{!}_y\langle i\rangle[-i]\), we
have, for each \(n\in\mathbb{Z}\), natural isomorphisms
\begin{align*}
\mathrm{H}^{n}\bigl(\mathfrak{F}'(P_x)\bigr)
&\cong
\bigoplus_{\substack{i\in\mathbb{Z}\\ y\in Q_{0}}}
\operatorname{Hom}_{\mathsf{D}^{b}(\Lambda^{!}\textup{-}\mathrm{Pe}^{\mathbb{Z}})}
\bigl(P^{!}_y\langle i\rangle[-n],\mathfrak{F}'(P_x)\bigr)
\\
&\cong
\bigoplus_{\substack{i\in\mathbb{Z}\\ y\in Q_{0}}}
\operatorname{Hom}_{\mathsf{D}^{b}(\Lambda\textup{-gmod})}
\bigl(S_y\langle -i\rangle[i-n],P_x\bigr)
\cong
\bigoplus_{\substack{i\in\mathbb{Z}\\ y\in Q_{0}}}
\operatorname{Ext}^{\,n-i}_{\Lambda\textup{-gmod}}
\bigl(S_y\langle -i\rangle,P_x\bigr).
\end{align*}
Hence only finitely many of the groups
\(\operatorname{Ext}^{j}_{\Lambda\textup{-gmod}}(S_y\langle k\rangle,P_x)\) are
nonzero, so \(P_x\) has finite injective dimension, that is,
\(P_x\in\Lambda\textup{-}\mathrm{cop}^{\mathbb{Z}}\). Consequently,
\(\operatorname{id}{}_{\Lambda}\Lambda<\infty\). Using the second equivalence
in~\emph{(2)} and applying the dual argument to the indecomposable graded
injective \(\Lambda\)-modules, one obtains that every indecomposable injective
\(\Lambda\)-module has finite projective dimension, that is,
\(\operatorname{id}\Lambda_{\Lambda}<\infty\). Hence \(\Lambda\) is
Iwanaga--Gorenstein.
\end{proof}
\subsection{Dg Koszul duality}
The purpose of this subsection is to construct a differential graded lift
\(\mathfrak{F}_{\mathrm{dg}}\) of the graded derived Koszul duality
\(\mathfrak{F}\), which will play a key role in the proofs of the main results
of Section~\ref{sec:nongraded} (Theorems~\ref{thm:nongraded-kd}
and~\ref{thm:nongraded-sing}).
We work in the homotopy category \(\operatorname{Ho}(\mathrm{dgcat})\) of small
dg categories, the localization of the category of small dg categories at the
quasi-equivalences~\cite{29}; two dg categories are called
\emph{quasi-equivalent} if they are isomorphic in
\(\operatorname{Ho}(\mathrm{dgcat})\). Our construction
shows that
\(\mathsf{D}^{b}_{\mathrm{dg}}(\Lambda^{!}\textup{-}\mathrm{Cop}^{\mathbb{Z}})\)
and \(\mathsf{D}^{b}_{\mathrm{dg}}(\Lambda\textup{-gmod})\) are quasi-equivalent,
and so are
\(\mathsf{D}^{b}_{\mathrm{dg}}(\Lambda^{!}\textup{-}\mathrm{Cop}^{\mathbb{Z}}/\Lambda^{!}\textup{-gmod})\)
and \(\mathsf{D}^{\mathrm{dg}}_{\mathrm{sg}}(\Lambda\textup{-gmod})\). Applying
\(\mathrm{H}^{0}\) recovers, respectively, the graded derived Koszul duality
and the graded singular Koszul duality.

We begin by constructing a dg enhancement of the derived Koszul duality.
\begin{theorem}\label{thm:dg-kd}
Let \(\Lambda\) be a finite-dimensional Koszul algebra. There is an isomorphism
\[
\mathfrak{F}_{\mathrm{dg}}\colon
\mathsf{D}^b_{\mathrm{dg}}\!\bigl(\Lambda^!\textup{-Cop}^{\mathbb{Z}}\bigr)
\longrightarrow
\mathsf{D}^b_{\mathrm{dg}}\!\bigl(\Lambda\textup{-gmod}\bigr)
\]
in \(\operatorname{Ho}(\mathrm{dgcat})\), induced by a quasi-equivalence between
dg models constructed in the proof, and its induced functor on homotopy categories yields a
triangulated equivalence
\[
H^{0}\!\bigl(\mathfrak{F}_{\mathrm{dg}}\bigr)\colon
\mathsf{D}^{b}\!\bigl(\Lambda^{!}\textup{-Cop}^{\mathbb{Z}}\bigr)
\;\xrightarrow{\ \sim\ }\;
\mathsf{D}^{b}\!\bigl(\Lambda\textup{-gmod}\bigr).
\]
\end{theorem}

\begin{proof}
By Definition~\ref{def:dg-derived}, and since
\(\mathsf{K}^{b}(\Lambda^{!}\textup{-}\mathrm{Inj}^{\mathbb{Z}})\simeq\mathsf{D}^{b}(\Lambda^{!}\textup{-}\mathrm{Cop}^{\mathbb{Z}})\)
and
\(\mathsf{K}^{-,b}(\Lambda\textup{-proj}^{\mathbb{Z}})\simeq\mathsf{D}^{b}(\Lambda\textup{-gmod})\),
the dg categories
\(\mathsf{D}^{b}_{\mathrm{dg}}(\Lambda^{!}\textup{-}\mathrm{Cop}^{\mathbb{Z}})\) and
\(\mathsf{D}^{b}_{\mathrm{dg}}(\Lambda\textup{-gmod})\) are isomorphic in
\(\operatorname{Ho}(\mathrm{dgcat})\) to
\(\mathsf{C}^{b}_{\mathrm{dg}}(\Lambda^{!}\textup{-}\mathrm{Inj}^{\mathbb{Z}})\) and
\(\mathsf{C}^{-,b}_{\mathrm{dg}}(\Lambda\textup{-proj}^{\mathbb{Z}})\), respectively;
we work with these models.

On objects we set
\(\mathfrak{F}_{\mathrm{dg}}(M^{\bullet}):=\mathcal{F}(M^{\bullet})=\operatorname{Tot}(\mathfrak{K}(M^{\bullet}))\),
with \(\mathfrak{K}\) and \(\operatorname{Tot}\) as in
Section~\ref{sec:quadratic-functor}. Let \(f=(f^{p})\colon M^{\bullet}\to N^{\bullet}\)
be homogeneous of degree \(r\), so that \(f^{p}\colon M^{p}\to N^{p+r}\). Applying \(K\) degreewise gives a homogeneous map
\(\mathfrak{K}(f)\) of bigraded objects of bidegree \((r,0)\),
with \(\mathfrak{K}(f)^{p,q}=K(f^{p})^{q}\), which commutes
with the vertical differentials. We set
\(\mathfrak{F}_{\mathrm{dg}}(f):=
\operatorname{Tot}(\mathfrak{K}(f))\), which is homogeneous
of degree \(r\). This assignment is clearly compatible with
composition and identities. Since \(K\) is a functor and each \(K(f^{p})\) is a
morphism of complexes, we have
\(d_1\mathfrak{K}(f)-(-1)^{r}\mathfrak{K}(f)d_1=\mathfrak{K}(d(f))\) and
\(d_2\mathfrak{K}(f)=\mathfrak{K}(f)d_2\). As the total differential is
\(d_1+(-1)^{p}d_2\) and \(\mathfrak{K}(f)\) raises \(p\) by \(r\), the terms
involving \(d_2\) in
\(d_{\operatorname{Tot}}\operatorname{Tot}(\mathfrak{K}(f))-(-1)^{r}\operatorname{Tot}(\mathfrak{K}(f))d_{\operatorname{Tot}}\)
are \((-1)^{p+r}\bigl(d_2\mathfrak{K}(f)-\mathfrak{K}(f)d_2\bigr)=0\), and hence
\(d(\mathfrak{F}_{\mathrm{dg}}(f))=\mathfrak{F}_{\mathrm{dg}}(d(f))\). Thus \(\mathfrak{F}_{\mathrm{dg}}\) is a dg functor, and
\(\mathrm{H}^{0}(\mathfrak{F}_{\mathrm{dg}})\colon\mathsf{K}^{b}(\Lambda^{!}\textup{-}\mathrm{Inj}^{\mathbb{Z}})\to\mathsf{K}^{-,b}(\Lambda\textup{-proj}^{\mathbb{Z}})\)
is the functor induced by \(\mathcal{F}\), which is an equivalence by
Theorem~\ref{thm:graded-kd}.

For \(X,Y\in\mathsf{C}^{b}_{\mathrm{dg}}(\Lambda^{!}\textup{-}\mathrm{Inj}^{\mathbb{Z}})\)
and \(i\in\mathbb{Z}\), the map induced by \(\mathfrak{F}_{\mathrm{dg}}\) on the
\(i\)-th cohomology of the morphism complexes identifies with
\(\operatorname{Hom}_{\mathsf{D}^{b}(\Lambda^{!}\textup{-}\mathrm{Cop}^{\mathbb{Z}})}(X,Y[i])\to\operatorname{Hom}_{\mathsf{D}^{b}(\Lambda\textup{-gmod})}(\mathfrak{F}X,\mathfrak{F}Y[i])\),
which is bijective by Theorem~\ref{thm:graded-kd}. Hence
\(\mathfrak{F}_{\mathrm{dg}}\) induces quasi-isomorphisms of morphism complexes, and it is essentially
surjective on \(\mathrm{H}^{0}\) since \(\mathfrak{F}\) is. Therefore
\(\mathfrak{F}_{\mathrm{dg}}\) is a quasi-equivalence.
\end{proof}
The dg derived Koszul duality descends to a dg singular Koszul duality.

\begin{theorem}\label{thm:dg-sing}
Let \(\Lambda\) be a finite-dimensional Koszul algebra. There is an isomorphism
\[
\mathfrak{F}^{\mathrm{sg}}_{\mathrm{dg}}\colon
\mathsf{D}^{b}_{\mathrm{dg}}\!\Bigl(\Lambda^{!}\textup{-Cop}^{\mathbb{Z}}
        \big/ \Lambda^{!}\textup{-gmod}\Bigr)
\longrightarrow
\mathsf{D}^{\mathrm{dg}}_{\mathrm{sg}}\!\bigl(\Lambda\textup{-gmod}\bigr)
\]
in \(\operatorname{Ho}(\mathrm{dgcat})\), induced by a quasi-equivalence between
dg models constructed in the proof, and its induced functor on homotopy categories yields a
triangulated equivalence
\[
H^{0}\!\bigl(\mathfrak{F}^{\mathrm{sg}}_{\mathrm{dg}}\bigr)\colon
\mathsf{D}^{b}\!\Bigl(\Lambda^{!}\textup{-Cop}^{\mathbb{Z}}
        \big/ \Lambda^{!}\textup{-gmod}\Bigr)
\;\xrightarrow{\ \sim\ }\;
\mathsf{D}_{\mathrm{sg}}\!\bigl(\Lambda\textup{-gmod}\bigr).
\]
\end{theorem}

\begin{proof}
By Definition~\ref{def:dg-derived}, the dg singularity category is
\(
\mathsf{D}^{\mathrm{dg}}_{\mathrm{sg}}(\Lambda\textup{-gmod})=\mathsf{C}^{-,b}_{\mathrm{dg}}(\Lambda\textup{-proj}^{\mathbb{Z}})/\mathsf{C}^{b}_{\mathrm{dg}}(\Lambda\textup{-proj}^{\mathbb{Z}}).
\)
On the \(\Lambda^{!}\)-side, we regard
\(\mathsf{D}^{b}_{\mathrm{dg}}(\Lambda^{!}\textup{-gmod})\) as the full dg
subcategory of \(\mathsf{D}^{b}_{\mathrm{dg}}(\Lambda^{!}\textup{-}\mathrm{Cop}^{\mathbb{Z}})\)
consisting of the complexes with cohomology in \(\Lambda^{!}\textup{-gmod}\),
which is justified by Lemma~\ref{lem:gmod-cop}. The exact quotient functor
\(\Lambda^{!}\textup{-}\mathrm{Cop}^{\mathbb{Z}}\to\Lambda^{!}\textup{-}\mathrm{Cop}^{\mathbb{Z}}/\Lambda^{!}\textup{-gmod}\)
induces a dg functor
\[
\mathsf{D}^{b}_{\mathrm{dg}}\bigl(\Lambda^{!}\textup{-}\mathrm{Cop}^{\mathbb{Z}}\bigr)
\big/\mathsf{D}^{b}_{\mathrm{dg}}\bigl(\Lambda^{!}\textup{-gmod}\bigr)
\longrightarrow
\mathsf{D}^{b}_{\mathrm{dg}}\bigl(\Lambda^{!}\textup{-}\mathrm{Cop}^{\mathbb{Z}}/\Lambda^{!}\textup{-gmod}\bigr)
\]
between pretriangulated dg categories, which is an equivalence on
\(\mathrm{H}^{0}\) by Theorem~\ref{thm:quot-derived}, hence a
quasi-equivalence. We may therefore work with the dg quotient on the left.

By Corollary~\ref{cor:special}(1), the dg functor \(\mathfrak{F}_{\mathrm{dg}}\)
of Theorem~\ref{thm:dg-kd} maps the full dg subcategory of
\(\mathsf{C}^{b}_{\mathrm{dg}}(\Lambda^{!}\textup{-}\mathrm{Inj}^{\mathbb{Z}})\)
of complexes with cohomology in \(\Lambda^{!}\textup{-gmod}\), which
corresponds to \(\mathsf{D}^{b}_{\mathrm{dg}}(\Lambda^{!}\textup{-gmod})\), onto
\(\mathsf{C}^{b}_{\mathrm{dg}}(\Lambda\textup{-proj}^{\mathbb{Z}})\), up to
homotopy equivalence. Hence the composition of \(\mathfrak{F}_{\mathrm{dg}}\)
with the quotient functor to
\(\mathsf{D}^{\mathrm{dg}}_{\mathrm{sg}}(\Lambda\textup{-gmod})\) sends the
objects of \(\mathsf{D}^{b}_{\mathrm{dg}}(\Lambda^{!}\textup{-gmod})\) to objects
that are zero in \(\mathrm{H}^{0}\), and, choosing contracting homotopies for
them, it descends to a dg functor
\[
\mathfrak{F}^{\mathrm{sg}}_{\mathrm{dg}}\colon
\mathsf{D}^{b}_{\mathrm{dg}}\bigl(\Lambda^{!}\textup{-}\mathrm{Cop}^{\mathbb{Z}}\bigr)
\big/\mathsf{D}^{b}_{\mathrm{dg}}\bigl(\Lambda^{!}\textup{-gmod}\bigr)
\longrightarrow
\mathsf{D}^{\mathrm{dg}}_{\mathrm{sg}}\bigl(\Lambda\textup{-gmod}\bigr).
\]
By construction, \(\mathrm{H}^{0}(\mathfrak{F}^{\mathrm{sg}}_{\mathrm{dg}})\) is
the graded singular Koszul duality, which is an equivalence by
Theorem~\ref{thm:graded-sing}. Since both dg categories are pretriangulated,
the maps induced by \(\mathfrak{F}^{\mathrm{sg}}_{\mathrm{dg}}\) on
\(\mathrm{H}^{i}\) of morphism complexes are
\(\operatorname{Hom}(X,Y[i])\to\operatorname{Hom}(\mathfrak{F}X,\mathfrak{F}Y[i])\),
and hence bijective. Thus \(\mathfrak{F}^{\mathrm{sg}}_{\mathrm{dg}}\) is a
quasi-equivalence.
\end{proof}
\section{Non-graded Koszul dualities}\label{sec:nongraded}
Many categories of interest admit realizations as categories of
finite-dimensional modules over finite-dimensional Koszul algebras. Prominent
examples arise in the representation theory of semisimple Lie algebras,
notably the blocks of the Bernstein--Gel'fand--Gel'fand category
\(\mathcal{O}\), as well as in geometric settings such as categories of
perverse sheaves on varieties with a stratification by affine spaces, whose
mixed versions correspond to graded modules (see, for example,~\cite{6}). It is
therefore natural to study bounded derived categories of finite-dimensional
modules over finite-dimensional Koszul algebras without imposing a grading,
that is, to seek a formulation of Koszul duality in the ungraded setting.

One might expect that a non-graded analogue of Koszul duality would yield an
equivalence
\(\mathsf{D}^{b}(\Lambda^{!}\textup{-mod})\simeq\mathsf{D}^{b}(\Lambda\textup{-mod})\).
However, such an equivalence does not hold in general. For instance, if
\(\Lambda=kQ/J^{2}\) is a radical-square-zero algebra with \(Q\) acyclic, where
\(J\) denotes the ideal of \(kQ\) generated by the arrows, then
\(\Lambda\) is Koszul and \(\Lambda^{!}\cong kQ^{\mathrm{op}}\) is hereditary,
so such an equivalence would force \(\Lambda\) to be piecewise hereditary, that is, derived
equivalent to a hereditary abelian category;
but the radical-square-zero algebra of~\cite[p.~157]{24}, whose quiver is
acyclic but not well directed (Section~\ref{sec:prelim}), is not piecewise
hereditary. On the other hand, Corollary~\ref{cor:well-directed} shows that the
derived quadratic functor yields such an equivalence when the quiver of
\(\Lambda\) is well directed; see also Corollary~\ref{cor:rsz-hered}.

To overcome this difficulty in general, we employ dg techniques, notably dg
enhancements and the dg orbit categories recalled in
Section~\ref{sec:prelim}, to construct a non-graded version of Koszul duality.
Throughout this section, unless stated otherwise, \(\Lambda\) denotes a
finite-dimensional graded algebra.

\subsection{Non-graded derived Koszul duality}
In this subsection, we establish a non-graded Koszul duality for the bounded
derived category of finite-dimensional modules over an arbitrary
finite-dimensional Koszul algebra, as well as for its singularity category.
Our strategy is to compare the graded categories
\(\mathsf{D}^b(\Lambda\textup{-gmod})\) and
\(\mathsf{D}_{\mathrm{sg}}(\Lambda\textup{-gmod})\) with their ungraded
counterparts \(\mathsf{D}^b(\Lambda\textup{-mod})\) and
\(\mathsf{D}_{\mathrm{sg}}(\Lambda\textup{-mod})\), and then to apply the
graded derived and singular Koszul dualities to realize
\(\mathsf{D}^b(\Lambda\textup{-mod})\) and
\(\mathsf{D}_{\mathrm{sg}}(\Lambda\textup{-mod})\) as triangulated hulls of dg
orbit categories of
\(\mathsf{D}^b_{\mathrm{dg}}(\Lambda^{!}\textup{-}\mathrm{Cop}^{\mathbb{Z}})\)
and
\(\mathsf{D}^b_{\mathrm{dg}}(\Lambda^{!}\textup{-}\mathrm{Cop}^{\mathbb{Z}}/\Lambda^{!}\textup{-gmod})\),
respectively. As a consequence, we obtain a non-graded form of the
Bernstein--Gel'fand--Gel'fand correspondence for the stable category
\(\Lambda\textup{-}\underline{\mathrm{Gproj}}\) of finite-dimensional
Gorenstein-projective modules over any finite-dimensional
Iwanaga--Gorenstein Koszul algebra.

The forgetful functor
\(F\colon\Lambda\textup{-gmod}\to\Lambda\textup{-mod}\) is exact and is a
\(\langle\langle1\rangle\rangle\)-Galois precovering (Proposition~\ref{prop:forget}). Since it
sends graded projective modules to projective modules, it induces triangulated
functors
\(F\colon\mathsf{D}^b(\Lambda\textup{-gmod})\to\mathsf{D}^b(\Lambda\textup{-mod})\)
and
\(F\colon\mathsf{D}_{\mathrm{sg}}(\Lambda\textup{-gmod})\to\mathsf{D}_{\mathrm{sg}}(\Lambda\textup{-mod})\),
which we denote again by \(F\). The following result describes the
relationship between the graded and ungraded categories.

\begin{theorem}\label{thm:forget-hull}
Let \(\Lambda\) be a finite-dimensional graded algebra. Then the following hold:
\begin{enumerate}
    \item The forgetful functor \(F\) induces an equivalence of triangulated categories
    \[
    \mathrm{H}^0\Bigl(\operatorname{pretr}\bigl(\mathsf{D}^b_{\mathrm{dg}}\bigl(\Lambda\textup{-gmod}\bigr)\big/\langle 1 \rangle\bigr)\Bigr)
    \;\xrightarrow{\sim}\;
    \mathsf{D}^b\bigl(\Lambda\textup{-mod}\bigr).
    \]
    \item The forgetful functor \(F\) induces an equivalence of triangulated categories
    \[
    \mathrm{H}^0\Bigl(\operatorname{pretr}\bigl(\mathsf{D}_{\mathrm{sg}}^{\mathrm{dg}}\bigl(\Lambda\textup{-gmod}\bigr)\big/\langle 1 \rangle\bigr)\Bigr)
    \;\xrightarrow{\sim}\;
    \mathsf{D}_{\mathrm{sg}}\bigl(\Lambda\textup{-mod}\bigr).
    \]
\end{enumerate}
\end{theorem}

\begin{proof}
We prove only (1). Assertion (2) is due to Keller, Murfet and Van den Bergh;
see~\cite[Appendix~A]{30}. Their argument is given for commutative graded
rings, and it carries over verbatim to finite-dimensional graded algebras.
Forgetting the grading defines a dg functor
\(F_{\mathrm{dg}}\colon\mathsf{D}^b_{\mathrm{dg}}(\Lambda\textup{-gmod})\to\mathsf{D}^b_{\mathrm{dg}}(\Lambda\textup{-mod})\)
lifting \(F\), the grading shift \(\langle1\rangle\) is a dg automorphism of
\(\mathsf{D}^b_{\mathrm{dg}}(\Lambda\textup{-gmod})\), and
\(F_{\mathrm{dg}}\circ\langle1\rangle=F_{\mathrm{dg}}\). Moreover, the essential
image of \(F\) contains all simple \(\Lambda\)-modules, which generate
\(\mathsf{D}^b(\Lambda\textup{-mod})\). Hence, by
Theorem~\ref{thm:orbit-hull}, it suffices to show that \(F\) is a
\(G\)-Galois precovering, where \(G\) is the group generated by
\(\langle1\rangle\), that is, that \(F\) induces isomorphisms
\[
\bigoplus_{r \in \mathbb{Z}}
\mathrm{Hom}_{\mathsf{D}^b(\Lambda\textup{-gmod})}\bigl(X^\bullet,Y^\bullet\langle r\rangle\bigr)
\xrightarrow{\ \sim\ }
\mathrm{Hom}_{\mathsf{D}^b(\Lambda\textup{-mod})}\bigl(X^\bullet,Y^\bullet\bigr)
\]
for all \(X^\bullet,Y^\bullet\in\mathsf{D}^b(\Lambda\textup{-gmod})\).

We may assume that \(Y^\bullet\) is a bounded complex and that \(X^\bullet\) is
a bounded above complex \(P^\bullet\) of graded projective modules with bounded
cohomology. Then both sides are computed in the homotopy category, as
\(\mathrm{H}^{0}\) of the morphism complexes
\(\operatorname{Hom}_{\mathsf{C}_{\mathrm{dg}}(\Lambda\textup{-gmod})}(P^\bullet,Y^\bullet\langle r\rangle)\)
and
\(\operatorname{Hom}_{\mathsf{C}_{\mathrm{dg}}(\Lambda\textup{-mod})}(P^\bullet,Y^\bullet)\),
respectively. In degree \(n\), by Proposition~\ref{prop:forget},
\[
\operatorname{Hom}^{n}_{\mathsf{C}_{\mathrm{dg}}(\Lambda\textup{-mod})}(P^\bullet,Y^\bullet)
=\prod_{i\in\mathbb{Z}}\operatorname{Hom}_{\Lambda}\bigl(P^i,Y^{i+n}\bigr)
\cong\bigoplus_{m\in\mathbb{Z}}\prod_{i\in\mathbb{Z}}
\operatorname{Hom}_{\Lambda\textup{-gmod}}\bigl(P^i,Y^{i+n}\langle m\rangle\bigr),
\]
where the product and the sum may be interchanged because \(Y^\bullet\) is
bounded, so that only finitely many \(i\) contribute. These isomorphisms are
compatible with the differentials, so
\(\operatorname{Hom}_{\mathsf{C}_{\mathrm{dg}}(\Lambda\textup{-mod})}(P^\bullet,Y^\bullet)\cong\bigoplus_{m}\operatorname{Hom}_{\mathsf{C}_{\mathrm{dg}}(\Lambda\textup{-gmod})}(P^\bullet,Y^\bullet\langle m\rangle)\)
as complexes. Taking \(\mathrm{H}^{0}\), which commutes with direct sums,
gives the required isomorphism.
\end{proof}
We next record a preparatory lemma that will be used throughout the sequel.
For closely related statements, compare~\cite[Theorem~22]{34}
and~\cite[Theorem~1.2.6]{6}; see also~\cite[Lemma~4.4(2)]{3}.

\begin{lemma}\label{lem:shift}
Let \(\Lambda\) be a finite-dimensional Koszul algebra with Koszul dual
\(\Lambda^{!}\), and let
\(\mathfrak{F}\colon \mathsf{D}^{b}(\Lambda^{!}\textup{-}\mathrm{Cop}^{\mathbb{Z}})\to\mathsf{D}^{b}(\Lambda\textup{-gmod})\)
be the graded derived Koszul duality. Then for every \(i\in\mathbb{Z}\) and
every \(X^{\bullet}\in \mathsf{D}^{b}(\Lambda^{!}\textup{-}\mathrm{Cop}^{\mathbb{Z}})\)
there is a natural isomorphism
\[
\mathfrak{F}\bigl(X^{\bullet}\langle i\rangle[-i]\bigr)
\cong
\mathfrak{F}(X^{\bullet})\langle -i\rangle.
\]
\end{lemma}

\begin{proof}
By definition of the derived quadratic functor,
\(\mathfrak{F}(X^\bullet)=\mathfrak{T}\bigl(\operatorname{Tot}(\mathfrak{K}(X^\bullet))\bigr)\),
where \(\mathfrak{K}\) is obtained by applying \(K\) degreewise and
\(\operatorname{Tot}\) denotes the diagonal totalization with differential
\(d_{\operatorname{Tot}}=d_1+(-1)^{p}d_2\).

Fix \(i\in\mathbb{Z}\), and let
\(X^\bullet\in\mathsf{C}^{b}(\Lambda^{!}\textup{-}\mathrm{Cop}^{\mathbb{Z}})\) be a
representative. Set \(A:=\mathfrak{K}(X^\bullet)\) and
\(A_i:=\mathfrak{K}(X^\bullet\langle i\rangle[-i])\), so that
\(A^{p,q}=K(X^{p})^{q}\) and \(A_i^{p,q}=K(X^{p-i}\langle i\rangle)^{q}\). We
write \(d_1,d_2\) for the horizontal and vertical differentials of \(A\), and
\(d_1',d_2'\) for those of \(A_i\).

For every \(M\in \Lambda^{!}\textup{-}\mathrm{Cop}^{\mathbb{Z}}\), the graded
modules \(K(M\langle i\rangle)^{q}\) and \(K(M)^{q+i}\langle -i\rangle\) both
equal \(\bigoplus_{x}P_x\langle q\rangle\otimes_k M^{q+i}_x\), with the same
differentials; hence the maps \(\theta^{q}_{M,i}=(-1)^{iq}\,\mathrm{id}\) define
a canonical isomorphism of complexes
\(\theta_{M,i}\colon K(M\langle i\rangle)\xrightarrow{\sim}K(M)[i]\langle -i\rangle\).
Since the differential on the shifted complex \(K(M)[i]\) is
\(d_{K(M)[i]}^{q}=(-1)^{i}d_{K(M)}^{q+i}\), these maps satisfy
\[
\theta_{M,i}^{q+1}\circ d_{K(M\langle i\rangle)}^{q}
=(-1)^{i}\,d_{K(M)}^{q+i}\circ\theta_{M,i}^{q}.
\]
Applying this with \(M=X^{p-i}\), we obtain canonical isomorphisms
\[
\theta^{p,q}\colon
A_i^{p,q}=K\bigl(X^{p-i}\langle i\rangle\bigr)^{q}
\xrightarrow{\ \sim\ }
K(X^{p-i})^{q+i}\langle -i\rangle=A^{p-i,q+i}\langle -i\rangle.
\]
Since \((p-i)+(q+i)=p+q\), the reindexing \((p,q)\mapsto(p-i,q+i)\) preserves
the diagonals. Hence, for each \(n\in\mathbb{Z}\), the maps \(\theta^{p,q}\)
assemble to an isomorphism of graded modules
\[
\Phi^{n}\colon
\operatorname{Tot}(A_i)^{n}=\bigoplus_{p+q=n}A_i^{p,q}
\longrightarrow
\bigoplus_{p+q=n}A^{p-i,q+i}\langle -i\rangle=\operatorname{Tot}(A)^{n}\langle -i\rangle,
\qquad
\Phi^{n}\big|_{A_i^{p,q}}=(-1)^{ip}\theta^{p,q}.
\]
We claim that \(\Phi=(\Phi^{n})_{n\in\mathbb{Z}}\) is an isomorphism of
complexes
\(\Phi\colon\operatorname{Tot}(A_i)\xrightarrow{\sim}\operatorname{Tot}(A)\langle -i\rangle\).
Let \(u\in A_i^{p,q}\) and \(n=p+q\).

Since \(A_i\) is obtained from the shifted complex \(X^\bullet[-i]\), the
horizontal differential acquires the factor \((-1)^i\), that is,
\(d_1'=(-1)^i d_1\) under the above reindexing. Therefore
\[
\Phi^{n+1}(d_1'u)
=(-1)^{i(p+1)}\theta^{p+1,q}(d_1'u)
=(-1)^{i(p+1)}(-1)^i\,d_1\theta^{p,q}(u)
=(-1)^{ip}\,d_1\theta^{p,q}(u)
=d_1\Phi^{n}(u).
\]
For the vertical differential, the identity satisfied by \(\theta^{p,q}\)
yields \(\theta^{p,q+1}(d_2'u)=(-1)^i d_2\theta^{p,q}(u)\), and hence
\[
\Phi^{n+1}\bigl((-1)^p d_2'u\bigr)
=(-1)^p(-1)^{ip}(-1)^i\,d_2\theta^{p,q}(u)
=(-1)^{p-i}\,d_2\Phi^{n}(u).
\]
Since \(\Phi^n(u)\) lies in the summand \(A^{p-i,q+i}\langle -i\rangle\), the
total differential on the target is
\(d_{\operatorname{Tot}}=d_1+(-1)^{p-i}d_2\) there. Combining the two
computations yields
\(\Phi^{n+1}(d_{\operatorname{Tot}}u)=d_{\operatorname{Tot}}(\Phi^n(u))\).

Thus \(\Phi\) is an isomorphism of complexes, and
\(\operatorname{Tot}\bigl(\mathfrak{K}(X^\bullet\langle i\rangle[-i])\bigr)\cong\operatorname{Tot}\bigl(\mathfrak{K}(X^\bullet)\bigr)\langle -i\rangle\)
functorially in \(X^\bullet\). Applying \(\mathfrak{T}\), we obtain the desired
natural isomorphism
\(\mathfrak{F}(X^\bullet\langle i\rangle[-i])\cong\mathfrak{F}(X^\bullet)\langle -i\rangle\).
\end{proof}

\begin{remark}\label{rem:shift-dg}
The isomorphism \(\Phi\) is defined on the level of complexes and is natural
with respect to all homogeneous morphisms of the dg categories involved, when
the dg shift acts on morphisms with the sign convention of
Section~\ref{sec:prelim}. Indeed, for a morphism \(f\) of degree \(r\), on the
\(p\)-th column one has
\(\Phi\circ\mathfrak{F}([-i]f)=(-1)^{i(p+r)}(-1)^{ir}\,\theta\,\mathfrak{K}(f)=(-1)^{ip}\,\theta\,\mathfrak{K}(f)=\mathfrak{F}(f)\langle -i\rangle\circ\Phi\).
Hence \(\Phi\) defines an isomorphism of dg functors
\(\mathfrak{F}_{\mathrm{dg}}\circ(\langle i\rangle[-i])\cong\langle -i\rangle\circ\mathfrak{F}_{\mathrm{dg}}\),
and in particular of the associated dg bimodules, as required in
Theorem~\ref{thm:orbit-hull}.
\end{remark}

We are now ready to prove the second main result of this paper, which we call
the \emph{non-graded derived Koszul duality}.

\begin{theorem}\label{thm:nongraded-kd}
Let \(\Lambda\) be a finite-dimensional Koszul algebra with Koszul dual
\(\Lambda^!\). Then there is a triangulated equivalence
\[
\mathrm{H}^0\Bigl(\operatorname{pretr}\bigl(\mathsf{D}^b_{\mathrm{dg}}\bigl(\Lambda^!\textup{-}\mathrm{Cop}^{\mathbb{Z}}\bigr)\big/\langle 1 \rangle[-1]\bigr)\Bigr)
\;\xrightarrow{\sim}\;
\mathsf{D}^b\bigl(\Lambda\textup{-mod}\bigr).
\]
\end{theorem}

\begin{proof}
We apply Theorem~\ref{thm:orbit-hull} to the functor
\(\ell=F\circ\mathfrak{F}\colon\mathsf{D}^b(\Lambda^!\textup{-}\mathrm{Cop}^{\mathbb{Z}})\to\mathsf{D}^b(\Lambda\textup{-mod})\)
and the autoequivalence \(\langle1\rangle[-1]\) of
\(\mathsf{D}^b(\Lambda^!\textup{-}\mathrm{Cop}^{\mathbb{Z}})\), where
\(\mathfrak{F}\) is the graded derived Koszul duality and \(F\) is the
forgetful functor.

The functor \(\ell\) lifts to the dg functor
\(\ell_{\mathrm{dg}}=F_{\mathrm{dg}}\circ\mathfrak{F}_{\mathrm{dg}}\) (computed on
the model \(\mathsf{C}^b_{\mathrm{dg}}(\Lambda^!\textup{-}\mathrm{Inj}^{\mathbb{Z}})\)
of Theorem~\ref{thm:dg-kd}; see Remark~\ref{rem:models}), where
\(\mathfrak{F}_{\mathrm{dg}}\) is the dg lift of Theorem~\ref{thm:dg-kd} and
\(F_{\mathrm{dg}}\) forgets the grading, and \(\langle1\rangle[-1]\) is a dg
automorphism of this model.
We have isomorphisms of dg functors
\[
\ell_{\mathrm{dg}}\circ\langle1\rangle[-1]
=F_{\mathrm{dg}}\circ\mathfrak{F}_{\mathrm{dg}}\circ\langle1\rangle[-1]
\cong F_{\mathrm{dg}}\circ\langle-1\rangle\circ\mathfrak{F}_{\mathrm{dg}}
=F_{\mathrm{dg}}\circ\mathfrak{F}_{\mathrm{dg}}
=\ell_{\mathrm{dg}},
\]
where the isomorphism is given by Remark~\ref{rem:shift-dg} with \(i=1\), and
the second equality holds since \(F_{\mathrm{dg}}\) forgets the grading. This
induces the isomorphism \(\ell\circ\langle1\rangle[-1]\cong\ell\) of
Lemma~\ref{lem:shift}.

Next, \(\ell\) is a \(\langle\langle1\rangle[-1]\rangle\)-Galois precovering: for
all \(X^\bullet,Y^\bullet\in\mathsf{D}^b(\Lambda^!\textup{-}\mathrm{Cop}^{\mathbb{Z}})\),
\[
\begin{aligned}
\bigoplus_{m\in\mathbb{Z}}
\mathrm{Hom}_{\mathsf{D}^b(\Lambda^!\textup{-}\mathrm{Cop}^{\mathbb{Z}})}
\bigl(X^\bullet,Y^\bullet\langle m\rangle[-m]\bigr)
&\cong
\bigoplus_{m\in\mathbb{Z}}
\mathrm{Hom}_{\mathsf{D}^b(\Lambda\textup{-gmod})}
\bigl(\mathfrak{F}X^\bullet,\mathfrak{F}(Y^\bullet\langle m\rangle[-m])\bigr)
\\
&\cong
\bigoplus_{m\in\mathbb{Z}}
\mathrm{Hom}_{\mathsf{D}^b(\Lambda\textup{-gmod})}
\bigl(\mathfrak{F}X^\bullet,\mathfrak{F}Y^\bullet\langle -m\rangle\bigr)
\\
&\cong
\mathrm{Hom}_{\mathsf{D}^b(\Lambda\textup{-mod})}
\bigl(F\mathfrak{F}X^\bullet,F\mathfrak{F}Y^\bullet\bigr),
\end{aligned}
\]
where the first isomorphism holds since \(\mathfrak{F}\) is an equivalence
(Theorem~\ref{thm:graded-kd}), the second by Lemma~\ref{lem:shift}, and the
third since \(F\) is a \(\langle\langle1\rangle\rangle\)-Galois precovering on bounded
derived categories (see the proof of Theorem~\ref{thm:forget-hull}).

Finally, by Lemma~\ref{lem:koszul-res},
\(F\mathfrak{F}(I^{!}_x\langle n\rangle)\cong S_x[n]\) for all \(x\in Q_0\) and
\(n\in\mathbb{Z}\), so the essential image of \(F\mathfrak{F}\) contains the
simple \(\Lambda\)-modules, which generate \(\mathsf{D}^b(\Lambda\textup{-mod})\).
Theorem~\ref{thm:orbit-hull} now yields the claimed equivalence.
\end{proof}
\begin{remark}
The non-graded derived Koszul duality shows that
\(\mathsf{D}^b(\Lambda\textup{-mod})\) is the triangulated hull of the orbit
category
\(\mathsf{D}^b(\Lambda^!\textup{-}\mathrm{Cop}^{\mathbb{Z}})/\langle 1\rangle[-1]\).
\end{remark}

If \(\Lambda\) has finite global dimension, the non-graded derived Koszul
duality specializes to the following form.

\begin{corollary}\label{cor:fingldim}
Let \(\Lambda\) be a finite-dimensional Koszul algebra of finite global
dimension, and let \(\Lambda^!\) denote its Koszul dual. Then there is a
triangulated equivalence
\[
\mathrm{H}^0\Bigl(
\operatorname{pretr}\bigl(
\mathsf{D}^b_{\mathrm{dg}}(\Lambda^!\textup{-gmod})
\big/\langle 1\rangle[-1]
\bigr)
\Bigr)
\;\xrightarrow{\ \sim\ }\;
\mathsf{D}^b(\Lambda\textup{-mod}).
\]
\end{corollary}

\begin{proof}
By Corollary~\ref{cor:special}(2) and its proof, \(\Lambda^!\) is
finite-dimensional and
\(\Lambda^!\textup{-}\mathrm{Cop}^{\mathbb{Z}}=\Lambda^!\textup{-gmod}\). Hence
\(\mathsf{D}^b_{\mathrm{dg}}(\Lambda^!\textup{-}\mathrm{Cop}^{\mathbb{Z}})=\mathsf{D}^b_{\mathrm{dg}}(\Lambda^!\textup{-gmod})\),
and the assertion follows from Theorem~\ref{thm:nongraded-kd}.
\end{proof}
\subsection{Non-graded singular Koszul duality and the non-graded BGG correspondence}

We next establish a non-graded singular Koszul duality. As an application, we obtain a non-graded form of the
Bernstein--Gel'fand--Gel'fand correspondence.

\begin{theorem}\label{thm:nongraded-sing}
Let \(\Lambda\) be a finite-dimensional Koszul algebra and \(\Lambda^!\) its
Koszul dual. Then there is a triangulated equivalence
\[
\mathrm{H}^0\Bigl(\operatorname{pretr}\bigl(\mathsf{D}^b_{\mathrm{dg}}\bigl(\Lambda^!\textup{-}\mathrm{Cop}^{\mathbb{Z}}\big/\Lambda^!\textup{-gmod}\bigr)\big/\langle 1 \rangle[-1]\bigr)\Bigr)
\;\xrightarrow{\sim}\;
\mathsf{D}_{\mathrm{sg}}\bigl(\Lambda\textup{-mod}\bigr).
\]
\end{theorem}

\begin{proof}
The proof is the same as that of Theorem~\ref{thm:nongraded-kd}, applied to
\[
\ell=F\circ\mathfrak{F}\colon\mathsf{D}^b(\Lambda^!\textup{-}\mathrm{Cop}^{\mathbb{Z}}/\Lambda^!\textup{-gmod})\to\mathsf{D}_{\mathrm{sg}}(\Lambda\textup{-mod}),
\]
where \(\mathfrak{F}\) is the graded singular Koszul duality of
Theorem~\ref{thm:graded-sing}, with the following modifications. The dg lift
of \(\mathfrak{F}\) is given by Theorem~\ref{thm:dg-sing}. Since
\(\langle1\rangle[-1]\) preserves \(\mathsf{D}^b(\Lambda^!\textup{-gmod})\) and
\(\langle-1\rangle\) preserves \(\mathsf{K}^b(\Lambda\textup{-proj}^{\mathbb{Z}})\),
the isomorphism of Remark~\ref{rem:shift-dg} descends to the dg quotients,
once the contracting homotopies used in Theorem~\ref{thm:dg-sing} are chosen
compatibly with this isomorphism. The forgetful functor
\(F\colon\mathsf{D}_{\mathrm{sg}}(\Lambda\textup{-gmod})\to\mathsf{D}_{\mathrm{sg}}(\Lambda\textup{-mod})\)
is a \(\langle\langle1\rangle\rangle\)-Galois precovering by
Theorem~\ref{thm:forget-hull}(2); see~\cite[Lemma~A.7]{30}. Finally, the essential image of
\(\ell\) contains the images of the simple \(\Lambda\)-modules, which generate
\(\mathsf{D}_{\mathrm{sg}}(\Lambda\textup{-mod})\). Theorem~\ref{thm:orbit-hull}
now yields the claimed equivalence.
\end{proof}
\begin{remark}
The non-graded singular Koszul duality shows that
\(\mathsf{D}_{\mathrm{sg}}(\Lambda\textup{-mod})\) is the triangulated hull of
the orbit category
\(\mathsf{D}^b(\Lambda^!\textup{-}\mathrm{Cop}^{\mathbb{Z}}/\Lambda^!\textup{-gmod})/\langle 1\rangle[-1]\).
\end{remark}

We now present the non-graded version of the BGG correspondence.

\begin{theorem}\label{thm:nongraded-bgg}
Let \(\Lambda\) be a finite-dimensional Iwanaga--Gorenstein Koszul algebra with
Koszul dual \(\Lambda^!\). Then there is a triangulated equivalence
\[
\mathrm{H}^0\Bigl(\operatorname{pretr}\bigl(\mathsf{D}^b_{\mathrm{dg}}\bigl(\Lambda^!\textup{-}\mathrm{Cop}^{\mathbb{Z}}\big/\Lambda^!\textup{-gmod}\bigr)\big/\langle 1 \rangle[-1]\bigr)\Bigr)
\;\xrightarrow{\sim}\;
\Lambda\textup{-}\underline{\mathrm{Gproj}}.
\]
\end{theorem}

\begin{proof}
This follows from Theorem~\ref{thm:nongraded-sing} and Buchweitz's equivalence
\(\mathsf{D}_{\mathrm{sg}}(\Lambda\textup{-mod})\simeq\Lambda\textup{-}\underline{\mathrm{Gproj}}\)~\cite[Theorem~4.4.1]{12}.
\end{proof}

\section{Applications to category \texorpdfstring{$\mathcal{O}$}{O} and perverse sheaves}\label{sec:applications}

The non-graded dualities apply directly to the Koszul algebras of~\cite{6}.

\noindent\textbf{Category \(\mathcal{O}\).}
Let \(\mathfrak{g}\) be a finite-dimensional complex semisimple Lie algebra
with Borel subalgebra \(\mathfrak{b}\supseteq\mathfrak{h}\), let
\(\lambda\in\mathfrak{h}^*\) be integral and dominant, and let
\(\mathcal{O}_\lambda\) be the corresponding block of the
Bernstein--Gel'fand--Gel'fand category \(\mathcal{O}\). Let
\(\mathfrak{q}\supseteq\mathfrak{b}\) be a parabolic subalgebra,
\(\mathcal{O}^{\mathfrak{q}}\subseteq\mathcal{O}_0\) the parabolic category, and
assume that the corresponding Weyl group stabilizers satisfy
\(S_\lambda=S_{\mathfrak{q}}\). By~\cite[Theorem~1.1.3]{6} there are
finite-dimensional Koszul algebras \(A_Q\) and \(A^Q\), Koszul dual to each
other, with \(\mathcal{O}_\lambda\simeq A_Q\textup{-mod}\) and
\(\mathcal{O}^{\mathfrak{q}}\simeq A^Q\textup{-mod}\). The graded categories
\(\mathcal{O}_Q:=A_Q\textup{-gmod}\) and \(\mathcal{O}^Q:=A^Q\textup{-gmod}\) are
the graded versions considered in~\cite{6}, where Koszul duality yields
\(\mathsf{D}^b(\mathcal{O}_Q)\simeq\mathsf{D}^b(\mathcal{O}^Q)\)~\cite[Theorem~3.11.1]{6}.
Since the duality on \(\mathcal{O}\) fixes simple objects, \(A_Q\cong A_Q^{\mathrm{op}}\)
and \(A^Q\cong(A^Q)^{\mathrm{op}}\), so \(A^Q\) is also the Koszul dual of \(A_Q\)
in our sense. Both algebras have finite global dimension, so
Corollary~\ref{cor:fingldim} applies to \(A_Q\) and to \(A^Q\).

\begin{corollary}\label{cor:category-O}
Let \(\lambda\in\mathfrak{h}^*\) be integral and dominant, and let
\(\mathfrak{q}\) be a parabolic subalgebra with \(S_\lambda=S_{\mathfrak{q}}\).
Then there are triangulated equivalences
\begin{gather*}
\mathrm{H}^0\Bigl(\operatorname{pretr}\bigl(\mathsf{D}^b_{\mathrm{dg}}(\mathcal{O}^Q)/\langle 1\rangle[-1]\bigr)\Bigr)
\xrightarrow{\ \sim\ }\mathsf{D}^b(\mathcal{O}_\lambda),\\
\mathrm{H}^0\Bigl(\operatorname{pretr}\bigl(\mathsf{D}^b_{\mathrm{dg}}(\mathcal{O}_Q)/\langle 1\rangle[-1]\bigr)\Bigr)
\xrightarrow{\ \sim\ }\mathsf{D}^b(\mathcal{O}^{\mathfrak{q}}).
\end{gather*}
\end{corollary}

\noindent\textbf{Perverse sheaves.}
Let \(X=\bigsqcup_{w\in W}X_w\) be a complex algebraic variety with a Whitney
stratification by affine spaces, and let \(\mathcal{P}(X,W)\) be the category
of perverse sheaves constructible with respect to it. Assume that every
closure \(\overline{X}_w\) admits a resolution of singularities whose rational
homology is spanned by classes of algebraic cycles. Then
by~\cite[Theorem~1.4.2]{6} there is a finite-dimensional Koszul algebra \(A\)
of finite global dimension with \(\mathcal{P}(X,W)\simeq A\textup{-mod}\), whose
Koszul dual \(A^!\) is isomorphic to the Yoneda algebra
\(\operatorname{Ext}^\bullet_{\mathsf{D}^b(X)}(L,L)\), where \(L\) is the direct
sum of the simple perverse sheaves; this algebra is isomorphic to its opposite
by Verdier duality. Corollary~\ref{cor:fingldim} gives:

\begin{corollary}\label{cor:perverse}
There is a triangulated equivalence
\[
\mathrm{H}^0\Bigl(\operatorname{pretr}\bigl(\mathsf{D}^b_{\mathrm{dg}}(A^!\textup{-gmod})/\langle 1\rangle[-1]\bigr)\Bigr)
\xrightarrow{\ \sim\ }\mathsf{D}^b\bigl(\mathcal{P}(X,W)\bigr).
\]
\end{corollary}

\section{Classical examples}
The purpose of this section is to examine the behavior of the Koszul
dualities established in the previous sections for several important classes
of Koszul algebras. We consider, in particular, the hereditary and
radical\hyp square\hyp zero cases, as well as the exterior and symmetric
algebras, which are Koszul dual to one another. We show that our constructions
recover the Bernstein--Gel'fand--Gel'fand correspondence.

\subsection{Algebras with radical square zero and hereditary algebras}
In this subsection we apply our Koszul dualities to hereditary algebras and to
algebras with radical square zero. For radical square zero algebras given by
locally finite gradable quivers, a graded derived Koszul duality was
established by Bautista and Liu~\cite[Theorem~3.9]{3}, in terms of
representations of quivers rather than graded modules; they also constructed a
Galois covering of \(\mathsf{D}^b(\Lambda\textup{-mod})\)~\cite[Theorem~4.11]{3}.
Here these results are obtained, for finite quivers, as special cases of the
Koszul dualities of Sections~\ref{sec:graded}--\ref{sec:nongraded}, by a
different method. At the level of singularity categories, Smith~\cite{39}
obtained a related result using a different approach. Quadratic monomial
algebras are treated in~\cite{10}.

Recall that if \(\Lambda = kQ/J^2\) is an algebra with radical square zero,
then its Koszul dual \(\Lambda^!\) is the opposite path algebra,
\(\Lambda^! = kQ^{\mathrm{op}}\). In particular, \(\Lambda^!\) is hereditary.
We begin with the following result, which is a graded reformulation
of~\cite[Theorem~1.12]{4}.

\begin{lemma}\label{lem:rsz-fcp}
Let \(\Lambda = kQ/J^2\) be a finite-dimensional algebra with radical square
zero. Then any finitely copresented graded \(\Lambda^!\)-module \(M\) fits
into a short exact sequence \(0 \to F \to M \to I \to 0\), where \(F\) is a
finite-dimensional graded \(\Lambda^!\)-module and \(I\) is a finitely
cogenerated graded injective \(\Lambda^!\)-module.
\end{lemma}

When \(\Lambda^!\) is hereditary, the category of coperfect modules admits a
particularly explicit description, as does its associated tails category.

\begin{proposition}\label{prop:rsz}
Let \(\Lambda = kQ/J^2\) be an algebra with radical square zero. Then
\(\Lambda^!\textup{-}\mathrm{Cop}^\mathbb{Z} = \Lambda^!\textup{-}\mathrm{Fcp}^\mathbb{Z}\),
the algebra \(\Lambda^!\) is cocoherent, and the quotient category
\(\Lambda^!\textup{-}\mathrm{Fcp}^{\mathbb{Z}} / \Lambda^!\textup{-gmod}\) is
abelian and semisimple.
\end{proposition}

\begin{proof}
The first assertion follows immediately from the fact that
\(\Lambda^! = kQ^{\mathrm{op}}\) is hereditary.

We now prove that \(\Lambda^!\) is cocoherent. Let \(f \colon M \to N\) be a
morphism in \(\Lambda^!\textup{-}\mathrm{Cop}^\mathbb{Z}\), where both \(M\) and
\(N\) are graded finitely copresented \(\Lambda^!\)-modules. Since \(M\) is
finitely copresented and \(\operatorname{Im}(f)\) is finitely cogenerated, it
follows that \(\ker(f)\) is finitely copresented. Using the fact that
\(\Lambda^!\textup{-}\mathrm{Fcp}^{\mathbb{Z}}\) satisfies the two-out-of-three
property for short exact sequences, it follows that \(\operatorname{coker}(f)\)
is also finitely copresented. Hence,
\(\Lambda^!\textup{-}\mathrm{Fcp}^{\mathbb{Z}}\) is closed under kernels and
cokernels of morphisms, which proves that \(\Lambda^!\) is cocoherent.

We now show that the quotient category
\(\Lambda^!\textup{-}\mathrm{Fcp}^{\mathbb{Z}} / \Lambda^!\textup{-gmod}\) is
abelian and semisimple. Since \(\Lambda^!\textup{-gmod}\) is a Serre
subcategory of \(\Lambda^!\textup{-}\mathrm{Fcp}^{\mathbb{Z}}\), the quotient
category \(\Lambda^!\textup{-}\mathrm{Fcp}^{\mathbb{Z}} / \Lambda^!\textup{-gmod}\)
is abelian; see, for example,~\cite{20}.

It remains to show that the quotient is semisimple. By
Lemma~\ref{lem:rsz-fcp}, every object of the quotient is isomorphic to \(Q(I)\)
for a finitely cogenerated injective module \(I\). Let
\(0\to Q(I)\to E\to N\to0\) be a short exact sequence in the quotient. As
recalled in Section~\ref{sec:prelim}, it is isomorphic to the image under
\(Q\) of a short exact sequence \(0\to X\to E'\to N'\to0\) in
\(\Lambda^!\textup{-}\mathrm{GMod}^{-}\), and the isomorphism
\(Q(X)\cong Q(I)\) is represented by a roof
\(X\xrightarrow{f}Z\xleftarrow{s}I\) with \(s\in\Sigma\); since \(Q(f)\) is an
isomorphism, also \(f\in\Sigma\). Pushing out along \(f\), we may replace
\(X\) by \(Z\). Since \(\Lambda^!\) is hereditary, \(\operatorname{Im}s\) is
injective, so \(0\to\operatorname{Im}s\to Z\to\operatorname{coker}s\to0\)
splits, and the projection \(p\colon Z\to\operatorname{Im}s\) lies in
\(\Sigma\), because its kernel \(\operatorname{coker}s\) is finite-dimensional.
Pushing out along \(p\), we obtain a short exact sequence
\(0\to\operatorname{Im}s\to E''\to N'\to0\) in
\(\Lambda^!\textup{-}\mathrm{GMod}^{-}\) whose image under \(Q\) is isomorphic
to the given one. As \(\operatorname{Im}s\) is injective, this sequence splits,
and hence so does the given one. Thus every short exact sequence in the
quotient splits, i.e.\ the quotient is semisimple.
\end{proof}

We are now ready to state our graded derived and singular Koszul dualities
for finite-dimensional algebras with radical square zero.

\begin{theorem}\label{thm:rsz}
Let \(\Lambda = kQ/J^2\) be a finite-dimensional algebra with radical square
zero. Then the following hold:
\begin{enumerate}
    \item There are equivalences of triangulated categories
    \[
        \mathsf{D}^b\bigl(\Lambda^!\textup{-}\mathrm{Fcp}^{\mathbb{Z}}\bigr)
        \;\xrightarrow{\sim}\;
        \mathsf{D}^b\bigl(\Lambda\textup{-gmod}\bigr)
        \;\xrightarrow{\sim}\;
        \mathsf{D}^b\bigl(\Lambda^!\textup{-}\mathrm{Fp}^{\mathbb{Z}}\bigr).
    \]
    \item There is an equivalence of triangulated categories
    \[
        \mathsf{D}^b\bigl(\Lambda^!\textup{-}\mathrm{Fcp}^{\mathbb{Z}} / \Lambda^!\textup{-gmod}\bigr)
        \;\xrightarrow{\sim}\;
        \mathsf{D}_{\mathrm{sg}}\bigl(\Lambda\textup{-gmod}\bigr).
    \]
\end{enumerate}
In particular, \(\mathsf{D}_{\mathrm{sg}}(\Lambda\textup{-gmod})\) is a
semisimple abelian category.
\end{theorem}

\begin{proof}
By Proposition~\ref{prop:rsz},
\(\Lambda^!\textup{-}\mathrm{Cop}^{\mathbb{Z}}=\Lambda^!\textup{-}\mathrm{Fcp}^{\mathbb{Z}}\),
and \(\Lambda^!=kQ^{\mathrm{op}}\) is hereditary, hence coherent, so
\(\Lambda^!\textup{-}\mathrm{Pe}^{\mathbb{Z}}=\Lambda^!\textup{-}\mathrm{Fp}^{\mathbb{Z}}\).
The assertions follow from Theorems~\ref{thm:graded-kd}
and~\ref{thm:graded-sing}, and the last one from the semisimplicity of
\(\Lambda^!\textup{-}\mathrm{Fcp}^{\mathbb{Z}}/\Lambda^!\textup{-gmod}\).
\end{proof}

For algebras with radical square zero, Theorem~\ref{thm:forget-hull} can be
improved as follows.

\begin{theorem}\label{thm:rsz-orbit}
Let \(\Lambda\) be a finite-dimensional algebra with radical square zero. Then
the following hold:
\begin{enumerate}
    \item There is an equivalence of triangulated categories
    \[
    \mathsf{D}^b\bigl(\Lambda\textup{-gmod}\bigr)\big/\langle 1 \rangle
    \;\xrightarrow{\sim}\;
    \mathsf{D}^b\bigl(\Lambda\textup{-mod}\bigr),
    \]
    where the left-hand side denotes the triangulated orbit category with
    respect to the grading shift functor \(\langle 1 \rangle\).
    \item There is an equivalence of triangulated categories
    \[
    \mathsf{D}_{\mathrm{sg}}\bigl(\Lambda\textup{-gmod}\bigr)\big/\langle 1 \rangle
    \;\xrightarrow{\sim}\;
    \mathsf{D}_{\mathrm{sg}}\bigl(\Lambda\textup{-mod}\bigr).
    \]
\end{enumerate}
\end{theorem}

\begin{proof}
By Theorem~\ref{thm:forget-hull} and its proof, the forgetful functors on
\(\mathsf{D}^b\) and on \(\mathsf{D}_{\mathrm{sg}}\) are
\(\langle\langle1\rangle\rangle\)-Galois precoverings. For~(1), the forgetful functor
\(\mathsf{D}^b(\Lambda\textup{-gmod})\to\mathsf{D}^b(\Lambda\textup{-mod})\) is
dense by~\cite[Theorem~4.11]{3}. For~(2), every object of
\(\mathsf{D}_{\mathrm{sg}}(\Lambda\textup{-mod})\) is isomorphic to \(M[n]\)
for a finite-dimensional module \(M\) and \(n\in\mathbb{Z}\), and
\(M\cong(\Omega M)[1]\) there; since \(\operatorname{rad}^2\Lambda=0\), the
syzygy \(\Omega M\) is semisimple, hence gradable. In both cases the forgetful
functor is therefore a \(\langle\langle1\rangle\rangle\)-Galois covering, so the orbit
category is equivalent to the target, and we transport the triangulated
structure along this equivalence.
\end{proof}

The non-graded derived and singular Koszul dualities can also be refined as
follows.

\begin{theorem}\label{thm:rsz-nongraded}
Let \(\Lambda\) be a finite-dimensional algebra with radical square zero. Then
the following hold:
\begin{enumerate}
    \item There is an equivalence of triangulated categories
    \[
    \mathsf{D}^b\bigl(\Lambda^!\textup{-}\mathrm{Fcp}^\mathbb{Z}\bigr)\big/\langle 1 \rangle[-1]
    \;\xrightarrow{\sim}\;
    \mathsf{D}^b\bigl(\Lambda\textup{-mod}\bigr).
    \]
    \item There is an equivalence of triangulated categories
    \[
    \mathsf{D}^{b}\bigl(\Lambda^!\textup{-}\mathrm{Fcp}^\mathbb{Z} / \Lambda^!\textup{-gmod}\bigr)\big/\langle 1 \rangle[-1]
    \;\xrightarrow{\sim}\;
    \mathsf{D}_{\mathrm{sg}}\bigl(\Lambda\textup{-mod}\bigr).
    \]
\end{enumerate}
\end{theorem}

\begin{proof}
Combine Theorem~\ref{thm:rsz} with Theorem~\ref{thm:rsz-orbit} and
Lemma~\ref{lem:shift}.
\end{proof}

\begin{theorem}\label{thm:hereditary}
Let \(\Lambda\) be a finite-dimensional hereditary algebra. Then the
following hold:
\begin{enumerate}
    \item There is a triangulated equivalence
    \[
    \mathsf{D}^b\bigl(\Lambda^!\textup{-gmod}\bigr)
    \;\xrightarrow{\sim}\;
    \mathsf{D}^b\bigl(\Lambda\textup{-gmod}\bigr).
    \]
    \item There is a triangulated equivalence
    \[
    \mathrm{H}^0\Bigl(\operatorname{pretr}\bigl(\mathsf{D}_{\mathrm{dg}}^b(\Lambda^!\textup{-gmod}) \big/ \langle 1 \rangle[-1]\bigr)\Bigr)
    \;\xrightarrow{\sim}\;
    \mathsf{D}^b\bigl(\Lambda\textup{-mod}\bigr).
    \]
\end{enumerate}
\end{theorem}

\begin{proof}
Since \(\Lambda\) has finite global dimension, (1) follows from
Corollary~\ref{cor:special}(2), and (2) from Corollary~\ref{cor:fingldim}.
\end{proof}

When the quiver is well directed, radical square zero algebras and hereditary
algebras are related by a non-graded derived equivalence.

\begin{corollary}\label{cor:rsz-hered}
Let \(Q\) be a finite well directed quiver. Then the derived quadratic functor
induces triangulated equivalences
\[
\mathsf{D}^b\bigl(kQ^{\mathrm{op}}\textup{-mod}\bigr)
\;\xrightarrow{\sim}\;
\mathsf{D}^b\bigl(kQ/J^{2}\textup{-mod}\bigr),
\qquad
\mathsf{D}^b\bigl(kQ^{\mathrm{op}}/J^{2}\textup{-mod}\bigr)
\;\xrightarrow{\sim}\;
\mathsf{D}^b\bigl(kQ\textup{-mod}\bigr).
\]
In particular, \(kQ/J^{2}\) is piecewise hereditary.
\end{corollary}

\begin{remark}\label{rem:rsz-hered}
The assumption that \(Q\) is well directed cannot be dropped: for the acyclic
quiver \(Q\) of~\cite[p.~157]{24} (Section~\ref{sec:prelim}), the algebra
\(kQ/J^{2}\) is not piecewise hereditary, so
\(\mathsf{D}^b(kQ/J^{2}\textup{-mod})\not\simeq\mathsf{D}^b(kQ^{\mathrm{op}}\textup{-mod})\).
For every acyclic quiver \(Q\), however, \(kQ^{\mathrm{op}}\) is
finite-dimensional, so
\(kQ^{\mathrm{op}}\textup{-}\mathrm{Fcp}^{\mathbb{Z}}=kQ^{\mathrm{op}}\textup{-gmod}\),
and Theorem~\ref{thm:rsz-nongraded}(1) gives a triangulated equivalence
\(\mathsf{D}^b(kQ^{\mathrm{op}}\textup{-gmod})/\langle1\rangle[-1]\simeq\mathsf{D}^b(kQ/J^{2}\textup{-mod})\).
Thus the orbit category of the graded derived category of the Koszul dual
always recovers \(\mathsf{D}^b(kQ/J^{2}\textup{-mod})\), whereas the
ungraded derived category \(\mathsf{D}^b(kQ^{\mathrm{op}}\textup{-mod})\) does
so when \(Q\) is well directed.
\end{remark}

Recall that, by Gabriel's theorem~\cite{21}, a finite-dimensional hereditary
algebra \(\Lambda\) is of finite representation type if and only if the
underlying graph of each connected component of its quiver is a Dynkin diagram of type \(A_n\), \(D_n\),
\(E_6\), \(E_7\) or \(E_8\). It is well known that if a finite-dimensional
graded algebra is of finite representation type, then every module is
gradable; see~\cite[Theorem~4.3]{23}. In particular, the forgetful functor
\(F\colon \Lambda\textup{-gmod}\to \Lambda\textup{-mod}\) is a
\(\langle\langle1\rangle\rangle\)-Galois covering.

Before describing the bounded derived category of finite-dimensional
hereditary algebras of finite representation type, we recall the following
well-known characterization of objects in \(\mathsf{D}^b(\Lambda\textup{-mod})\);
see, for example,~\cite[Corollary~13.1.20, p.~324]{26}.

\begin{lemma}\label{lem:hered-split}
Let \(\Lambda\) be a finite-dimensional hereditary algebra and let
\(X^\bullet \in \mathsf{D}^b(\Lambda\textup{-mod})\). Then
\(X^\bullet \cong \bigoplus_{n\in\mathbb{Z}} H^n(X^\bullet)[-n]\) in
\(\mathsf{D}^b(\Lambda\textup{-mod})\).
\end{lemma}

We now give a new description of the bounded derived category of a
finite-dimensional hereditary algebra of finite representation type.

\begin{theorem}\label{thm:hereditary-finite}
Let \(\Lambda\) be a finite-dimensional hereditary algebra of finite
representation type and let \(\Lambda^{!}\) denote its Koszul dual algebra,
which is of radical square zero. Then the graded derived Koszul duality
together with the forgetful functor
\(F\colon\mathsf{D}^b(\Lambda\textup{-gmod})\to\mathsf{D}^b(\Lambda\textup{-mod})\)
induces a triangulated equivalence
\[
\mathsf{D}^b\bigl(\Lambda^{!}\textup{-gmod}\bigr)\big/\langle 1\rangle[-1]
\;\xrightarrow{\sim}\;
\mathsf{D}^b(\Lambda\textup{-mod}).
\]
\end{theorem}

\begin{proof}
By the proof of Theorem~\ref{thm:nongraded-kd}, the composition
\[
\begin{tikzcd}[column sep=large]
\mathsf{D}^b(\Lambda^{!}\textup{-gmod})
\arrow[r,"\mathfrak{F}"]
&
\mathsf{D}^b(\Lambda\textup{-gmod})
\arrow[r,"F"]
&
\mathsf{D}^b(\Lambda\textup{-mod})
\end{tikzcd}
\]
is a \(\langle\langle 1\rangle[-1]\rangle\)-Galois precovering. Since every
\(\Lambda\)-module is gradable, and since Lemma~\ref{lem:hered-split} shows
that every object of \(\mathsf{D}^b(\Lambda\textup{-mod})\) is isomorphic to a
finite direct sum of shifts of its cohomology modules, the induced functor is
dense. Hence it is a \(\langle\langle 1\rangle[-1]\rangle\)-Galois covering,
and the assertion follows.
\end{proof}

\begin{remark}
Since Dynkin quivers are well directed (Section~\ref{sec:prelim}),
Corollary~\ref{cor:well-directed} applies to \(\Lambda\) in
Theorem~\ref{thm:hereditary-finite} and gives
\(\mathsf{D}^b(\Lambda^{!}\textup{-mod})\simeq\mathsf{D}^b(\Lambda\textup{-mod})\).
Theorem~\ref{thm:hereditary-finite} complements this by realizing
\(\mathsf{D}^b(\Lambda\textup{-mod})\) as an orbit category of the graded
category \(\mathsf{D}^b(\Lambda^{!}\textup{-gmod})\).
\end{remark}

We now turn to the final example of this subsection and present both graded
and ungraded versions of the Bernstein--Gel'fand--Gel'fand correspondence for
algebras with radical square zero. Recall from~\cite{14} that, for a finite
connected quiver \(Q\), the algebra \(kQ/J^{2}\) is either self-injective or
Cohen--Macaulay free, that is, every Gorenstein-projective module is
projective. Invoking this result, we obtain the following triangulated
equivalences.

\begin{theorem}\label{thm:rsz-bgg}
Let \(\Lambda\) be a finite-dimensional self-injective algebra with radical
square zero. Then the following hold:
\begin{enumerate}
    \item There is a triangulated equivalence
    \[
        \mathsf{D}^b\bigl(\Lambda^!\textup{-}\mathrm{Fcp}^{\mathbb{Z}} \big/ \Lambda^!\textup{-gmod}\bigr)
        \;\xrightarrow{\sim}\;
        \Lambda\textup{-}\underline{\mathrm{gmod}}.
    \]
    \item There is a triangulated equivalence
    \[
        \mathsf{D}^{b}\bigl(\Lambda^!\textup{-}\mathrm{Fcp}^{\mathbb{Z}} \big/ \Lambda^!\textup{-gmod}\bigr)
        \big/ \langle 1 \rangle[-1]
        \;\xrightarrow{\sim}\;
        \Lambda\textup{-}\underline{\mathrm{mod}}.
    \]
\end{enumerate}
\end{theorem}

\begin{proof}
Since \(\Lambda\) is self-injective, every module is Gorenstein-projective. Hence
(1) follows from Theorem~\ref{thm:graded-bgg} and Proposition~\ref{prop:rsz},
and (2) follows from Theorem~\ref{thm:rsz-nongraded}(2) and Buchweitz's
equivalence \(\mathsf{D}_{\mathrm{sg}}(\Lambda\textup{-mod})\simeq\Lambda\textup{-}\underline{\mathrm{mod}}\)~\cite[Theorem~4.4.1]{12}.
\end{proof}

\subsection{Exterior algebras and symmetric algebras}
Let \(V\) be a vector space of dimension \(n\), let \(E=\bigwedge V\) be the
exterior algebra and \(S=S(V)\cong k[X_1,\dots,X_n]\) the symmetric algebra,
which is the Koszul dual of \(E\). The algebra \(E\) is finite-dimensional,
Koszul and self-injective, so its stable category of graded
Gorenstein-projective modules is \(E\textup{-}\underline{\mathrm{gmod}}\). The
algebra \(S\) is Noetherian of global dimension \(n\), so
\(S\textup{-}\mathrm{Pe}^{\mathbb{Z}}=S\textup{-}\mathrm{Fg}^{\mathbb{Z}}\).
Corollary~\ref{cor:noetherian} and Theorem~\ref{thm:graded-bgg} recover the
theorems of Bernstein, Gel'fand and Gel'fand~\cite[Theorems~2--4]{8}:
\[
\mathsf{D}^b(E\textup{-gmod})\simeq\mathsf{D}^b(S\textup{-}\mathrm{Fg}^{\mathbb{Z}}),
\qquad
\mathsf{D}^b\bigl(S\textup{-}\mathrm{Cop}^{\mathbb{Z}}/S\textup{-gmod}\bigr)\simeq E\textup{-}\underline{\mathrm{gmod}},
\]
where
\(S\textup{-}\mathrm{Cop}^{\mathbb{Z}}/S\textup{-gmod}\simeq(\mathrm{coh}\,\mathbb{P}^{n-1})^{\mathrm{op}}\)
via the graded duality \(\mathbb{D}\) and Serre's theorem. Theorems~\ref{thm:nongraded-kd}
and~\ref{thm:nongraded-bgg} give their non-graded versions:

\begin{corollary}\label{cor:exterior}
There are triangulated equivalences
\begin{gather*}
\mathrm{H}^0\Bigl(\operatorname{pretr}\bigl(\mathsf{D}^b_{\mathrm{dg}}(S\textup{-}\mathrm{Cop}^{\mathbb{Z}})/\langle1\rangle[-1]\bigr)\Bigr)\simeq\mathsf{D}^b(E\textup{-mod}),\\
\mathrm{H}^0\Bigl(\operatorname{pretr}\bigl(\mathsf{D}^b_{\mathrm{dg}}(S\textup{-}\mathrm{Cop}^{\mathbb{Z}}/S\textup{-gmod})/\langle1\rangle[-1]\bigr)\Bigr)\simeq E\textup{-}\underline{\mathrm{mod}}.
\end{gather*}
\end{corollary}

\section{Bounded derived categories of finite-dimensional graded algebras}
\subsection{Two conjectures}
We conclude by formulating two conjectures concerning the bounded derived
category and the singularity category of finite-dimensional graded algebras.
Motivated by the work of Beilinson--Ginzburg--Soergel on realizations of
categories of (mixed) perverse sheaves \(\mathsf{P}(X,W)\) (see
\cite[Proposition~1.4.1]{6}), and by the results established above, it is
natural to expect that analogous realization results, and hence more general
duality phenomena, should hold for arbitrary finite-dimensional graded
algebras. The following conjectures record directions that originally
inspired this work.

\begin{conjecture}\label{conj:derived-orbit}
Let \(\Lambda\) be a finite-dimensional graded algebra. Then there exist an
exact category \(\mathcal{A}\) of finite homological dimension and a group \(G\)
acting on \(\mathsf{D}^{b}_{\mathrm{dg}}(\mathcal{A})\) such that there is a
triangulated equivalence
\[
\mathrm{H}^{0}\bigl(\operatorname{pretr}(\mathsf{D}^{b}_{\mathrm{dg}}(\mathcal{A})/G)\bigr)
\;\simeq\;
\mathsf{D}^{b}\bigl(\Lambda\textup{-mod}\bigr).
\]
\end{conjecture}

\begin{conjecture}\label{conj:singularity-orbit}
Let \(\Lambda\) be a finite-dimensional graded algebra. Then there exist an
exact category \(\mathcal{B}\) of finite homological dimension and a group \(H\)
acting on \(\mathsf{D}^{b}_{\mathrm{dg}}(\mathcal{B})\) such that there is a
triangulated equivalence
\[
\mathrm{H}^{0}\bigl(\operatorname{pretr}(\mathsf{D}^{b}_{\mathrm{dg}}(\mathcal{B})/H)\bigr)
\;\simeq\;
\mathsf{D}_{\mathrm{sg}}\bigl(\Lambda\textup{-mod}\bigr).
\]
\end{conjecture}

For a finite-dimensional Koszul algebra \(\Lambda\),
Conjecture~\ref{conj:derived-orbit} holds with
\(\mathcal{A}=\Lambda^{!}\textup{-}\mathrm{Cop}^{\mathbb{Z}}\) and
\(G=\langle\langle1\rangle[-1]\rangle\) by Theorem~\ref{thm:nongraded-kd}.
Indeed, \(\Lambda^{!}\textup{-}\mathrm{Cop}^{\mathbb{Z}}\) has finite
homological dimension: its objects admit injective coresolutions by strongly
injective objects of length at most \(\operatorname{gldim}\Lambda^{!}<\infty\)
(Theorem~\ref{thm:fd-gldim} and Propositions~\ref{prop:strong-inj-ff}
and~\ref{prop:cop}). Theorem~\ref{thm:nongraded-sing} gives a description of
\(\mathsf{D}_{\mathrm{sg}}(\Lambda\textup{-mod})\) of the form predicted by
Conjecture~\ref{conj:singularity-orbit}, with
\(\mathcal{B}=\Lambda^{!}\textup{-}\mathrm{Cop}^{\mathbb{Z}}/\Lambda^{!}\textup{-gmod}\);
since \(\mathcal{B}\) has finite homological dimension by
Proposition~\ref{prop:quot-hd} below, Conjecture~\ref{conj:singularity-orbit}
holds for every finite-dimensional Koszul algebra. Moreover,
Conjecture~\ref{conj:singularity-orbit} is known to hold when \(\Lambda\) is
self-injective with Gorenstein parameter \(\ell\); see~\cite[Theorem~6.2]{40}.
These conjectures may also be viewed as being in the same general spirit as
Kontsevich's homological mirror symmetry conjecture~\cite{31}, where
triangulated hulls arise naturally in the setting of
\(A_{\infty}\)-categories.
\begin{proposition}\label{prop:quot-hd}
Let \(\Lambda\) be a finite-dimensional Koszul algebra. Then the exact category
\(\Lambda^{!}\textup{-}\mathrm{Cop}^{\mathbb{Z}}/\Lambda^{!}\textup{-gmod}\) has
homological dimension at most \(\operatorname{gldim}\Lambda^{!}\).
\end{proposition}
\begin{proof}
Let \(Q\colon\Lambda^{!}\textup{-}\mathrm{Cop}^{\mathbb{Z}}\to\Lambda^{!}\textup{-}\mathrm{Cop}^{\mathbb{Z}}/\Lambda^{!}\textup{-gmod}\)
be the quotient functor, and let \(L\) be the largest integer with
\(\Lambda_L\neq0\); by~\((\ast)\) and Theorem~\ref{thm:fd-gldim},
\(L=\operatorname{gldim}\Lambda^{!}\). Since every object of
\(\Lambda^{!}\textup{-}\mathrm{Cop}^{\mathbb{Z}}/\Lambda^{!}\textup{-gmod}\) is of
the form \(Q(M)\) with \(M\in\Lambda^{!}\textup{-}\mathrm{Cop}^{\mathbb{Z}}\), it
suffices to show that
\(\operatorname{Ext}^{i}_{\Lambda^{!}\textup{-}\mathrm{Cop}^{\mathbb{Z}}/\Lambda^{!}\textup{-gmod}}(Q(M),Q(N))=0\)
for all \(i>L\). By Theorem~\ref{thm:quot-derived}, Lemma~\ref{lem:gmod-cop} and
Theorem~\ref{thm:graded-sing},
\[
\operatorname{Ext}^{i}_{\Lambda^{!}\textup{-}\mathrm{Cop}^{\mathbb{Z}}/\Lambda^{!}\textup{-gmod}}(Q(M),Q(N))\cong
\operatorname{Hom}_{\mathsf{D}_{\mathrm{sg}}(\Lambda\textup{-gmod})}\bigl(K(M),K(N)[i]\bigr)
\]
for all \(i\ge0\), where
\(\mathsf{D}_{\mathrm{sg}}(\Lambda\textup{-gmod})\simeq\mathsf{K}^{-,b}(\Lambda\textup{-proj}^{\mathbb{Z}})/\mathsf{K}^{b}(\Lambda\textup{-proj}^{\mathbb{Z}})\).

Fix \(i>L\), and put \(X=K(M)\) and \(Y=K(N)[i]\). By the construction of \(K\),
each \(X^m\) is a finite direct sum of modules \(P_x\langle m\rangle\), and each
\(Y^m\) is a finite direct sum of modules \(P_y\langle m+i\rangle\). By
Proposition~\ref{prop:hom-proj},
\(\operatorname{Hom}_{\Lambda\textup{-gmod}}(P_x\langle m\rangle,P_y\langle m+i\rangle)\cong e_x\Lambda_i e_y=0\)
since \(i>L\). Consequently, every chain map from \(X\) to \(Y\), or to any
brutal truncation \(\sigma_{\le a}Y\), is zero.

We show that \(\operatorname{Hom}_{\mathsf{D}_{\mathrm{sg}}(\Lambda\textup{-gmod})}(X,Y)=0\).
By the calculus of fractions in Verdier quotients (see~\cite{Verdier}), every morphism \(X\to Y\) in
\(\mathsf{D}_{\mathrm{sg}}(\Lambda\textup{-gmod})\) is represented by a roof
\(X\xrightarrow{f}Y'\xleftarrow{s}Y\) in
\(\mathsf{K}^{-,b}(\Lambda\textup{-proj}^{\mathbb{Z}})\) with
\(C=\operatorname{Cone}(s)\in\mathsf{K}^{b}(\Lambda\textup{-proj}^{\mathbb{Z}})\);
we may choose an integer \(a\) such that \(C\) is represented by a complex
concentrated in degrees \(\ge a\). Consider the canonical projection
\(\pi\colon Y\to T:=\sigma_{\le a}Y\) onto the brutal truncation, the quotient
complex of \(Y\) with terms \(Y^m\) for \(m\le a\). Since \(Y\) is bounded above, the kernel
of \(\pi\) is a bounded complex of graded projectives, so
\(\operatorname{Cone}(\pi)\in\mathsf{K}^{b}(\Lambda\textup{-proj}^{\mathbb{Z}})\)
and \(\pi\) becomes invertible in \(\mathsf{D}_{\mathrm{sg}}(\Lambda\textup{-gmod})\).

Moreover, \(C[-1]\) is concentrated in degrees \(\ge a+1\), whereas \(T\) is
concentrated in degrees \(\le a\), so
\(\operatorname{Hom}_{\mathsf{K}(\Lambda\textup{-proj}^{\mathbb{Z}})}(C[-1],T)=0\).
Applying \(\operatorname{Hom}_{\mathsf{K}(\Lambda\textup{-proj}^{\mathbb{Z}})}(-,T)\)
to the triangle \(C[-1]\to Y\xrightarrow{s}Y'\to C\) shows that \(\pi=us\) for
some \(u\colon Y'\to T\). We obtain the commutative diagram
\[
\begin{tikzcd}[row sep=large,column sep=large]
X \arrow[r,"f"] \arrow[dr,"uf"']
  & Y' \arrow[d,"u"]
  & Y \arrow[l,"s"'] \arrow[dl,"\pi"] \\
  & T &
\end{tikzcd}
\]
in \(\mathsf{K}^{-,b}(\Lambda\textup{-proj}^{\mathbb{Z}})\). Hence the roof
\((f,s)\) is equivalent to \((uf,\pi)\). Since \(T=\sigma_{\le a}Y\), we have
\(uf=0\), and so \(s^{-1}f=\pi^{-1}(uf)=0\) in
\(\mathsf{D}_{\mathrm{sg}}(\Lambda\textup{-gmod})\). As the roof was arbitrary,
\(\operatorname{Hom}_{\mathsf{D}_{\mathrm{sg}}(\Lambda\textup{-gmod})}(X,Y)=0\),
and therefore
\(\operatorname{Ext}^{i}_{\Lambda^{!}\textup{-}\mathrm{Cop}^{\mathbb{Z}}/\Lambda^{!}\textup{-gmod}}(QM,QN)=0\)
for all \(i>L\).
\end{proof}

\subsection{Dg enhancement for the bounded derived category of a finite-dimensional graded algebra}

We next discuss Conjecture~\ref{conj:derived-orbit} for an arbitrary
finite-dimensional graded algebra, by exhibiting a natural candidate for the
exact category \(\mathcal{A}\) and for the acting group \(G\) that realize the
proposed orbit-category description. The construction is dictated by the
framework developed in the preceding sections. In our setting we take
\(\mathcal{A}=\mathcal{PLC}^{b}(\Lambda^{!})\), the abelian category of
perfectly linear bounded complexes defined below, and we specify the group
\(G\) explicitly.

We stress, however, that the finiteness of the homological dimension required
in Conjecture~\ref{conj:derived-orbit} is not addressed here: the category
\(\mathcal{PLC}^{b}(\Lambda^{!})\) is equivalent to \(\Lambda\textup{-gmod}\), so
it has finite homological dimension only when \(\Lambda\) does. In this sense,
Theorem~\ref{thm:plc} below is a reformulation of
Theorem~\ref{thm:forget-hull} in terms of the quadratic dual, which identifies
the candidate category \(\mathcal{A}\) and the group \(G\).

Let \(\Lambda=kQ/I\) be a finite-dimensional graded algebra. Denote by
\(\widetilde{\Lambda}=kQ/\widetilde{I}\) the \emph{quadratic algebra associated
to} \(\Lambda\), where \(\widetilde{I}\) is the ideal generated by the quadratic
relations in \(I\). Following Mart\'{\i}nez-Villa and Saor\'{\i}n~\cite{33}, the
\emph{quadratic dual} of \(\Lambda\) is defined to be the quadratic dual of
\(\widetilde{\Lambda}\) in the sense of Section~\ref{sec:prelim}, namely
\(\Lambda^{!}:=\widetilde{\Lambda}^{!}\).

By~\cite[Theorem~2.4]{33} (see also~\cite[Theorem~12]{34}), there is a fully
faithful functor
\(K\colon \Lambda\textup{-gmod}\to \mathcal{LC}^{b}(\Lambda^{!}\textup{-}\mathrm{Proj}^{\mathbb{Z}})\),
which sends a graded \(\Lambda\)-module \(M\) to the associated linear complex
of graded projective \(\Lambda^{!}\)-modules, as in
Section~\ref{sec:quadratic-functor}. More precisely, \(K\) is the composite of
the canonical embedding
\(\Lambda\textup{-gmod}\to \widetilde{\Lambda}\textup{-gmod}\) with the
equivalence
\(\widetilde{\Lambda}\textup{-gmod}\simeq\mathcal{LC}^{b}(\Lambda^{!}\textup{-}\mathrm{Proj}^{\mathbb{Z}})\).
In particular, if \(\Lambda\) is quadratic, then \(K\) is an equivalence.

We denote by \(\mathcal{PLC}^{b}(\Lambda^{!})\) the full subcategory of
\(\mathcal{LC}^{b}(\Lambda^{!}\textup{-}\mathrm{Proj}^{\mathbb{Z}})\) consisting of
\emph{perfectly linear} bounded complexes, that is, the essential image of
\(K\). Thus \(K\) restricts to an equivalence of abelian categories
\(\Lambda\textup{-gmod}\simeq\mathcal{PLC}^{b}(\Lambda^{!})\), and we fix a
quasi-inverse \(K'\colon \mathcal{PLC}^{b}(\Lambda^{!})\to \Lambda\textup{-gmod}\).
Consequently, \(K'\) induces a triangulated equivalence
\(\mathsf{D}^{b}(\mathcal{PLC}^{b}(\Lambda^{!}))\xrightarrow{\sim}\mathsf{D}^{b}(\Lambda\textup{-gmod})\),
which admits a dg enhancement; see, for instance,~\cite[\S\,9.8]{28}.

For a linear complex \(X\), we denote by \(\tau(X):=X[1]\langle-1\rangle\) the
shift of \(X\) \emph{as a linear complex}; it is again linear, and \(\tau\) is
an autoequivalence of the abelian category \(\mathcal{PLC}^{b}(\Lambda^{!})\). We
use the same letter for the induced autoequivalences of
\(\mathsf{D}^{b}(\mathcal{PLC}^{b}(\Lambda^{!}))\) and of its dg enhancement,
obtained by applying \(\tau\) termwise; this is not the suspension of
\(\mathsf{D}^{b}(\mathcal{PLC}^{b}(\Lambda^{!}))\). By construction of \(K\),
there are natural isomorphisms
\(K(M\langle i\rangle)\cong K(M)[i]\langle -i\rangle=\tau^{i}K(M)\) for all
\(i\in\mathbb{Z}\), compare the proof of Lemma~\ref{lem:shift}; equivalently,
\(K'(\tau^{i}X)\cong K'(X)\langle i\rangle\) for all
\(X\in\mathcal{PLC}^{b}(\Lambda^{!})\) and \(i\in\mathbb{Z}\). Combining
Theorem~\ref{thm:forget-hull} with the above triangulated equivalence and
invoking Theorem~\ref{thm:orbit-hull}, we obtain the following dg-enhanced
orbit descriptions.

\begin{theorem}\label{thm:plc}
Let \(\Lambda\) be a finite-dimensional graded algebra, and let \(\Lambda^{!}\)
denote its quadratic dual. Then the following hold:
\begin{enumerate}\itemsep=0.4em
\item There is a triangulated equivalence
\[
\mathrm{H}^{0}\Bigl(\operatorname{pretr}\bigl(
\mathsf{D}^{b}_{\mathrm{dg}}\bigl(\mathcal{PLC}^{b}(\Lambda^{!})\bigr)
\big/ \tau\bigr)\Bigr)
\;\xrightarrow{\ \sim\ }\;
\mathsf{D}^{b}\bigl(\Lambda\textup{-mod}\bigr).
\]
\item There is a triangulated equivalence
\[
\mathrm{H}^{0}\Bigl(\operatorname{pretr}\bigl(
\mathsf{D}^{\mathrm{dg}}_{\mathrm{sg}}\bigl(\mathcal{PLC}^{b}(\Lambda^{!})\bigr)
\big/ \tau\bigr)\Bigr)
\;\xrightarrow{\ \sim\ }\;
\mathsf{D}_{\mathrm{sg}}\bigl(\Lambda\textup{-mod}\bigr).
\]
\item If, in addition, \(\Lambda\) is monomial, then \(\Lambda^{!}\) is a
quadratic monomial algebra, the orbit category
\(\mathsf{D}_{\mathrm{sg}}(\mathcal{PLC}^{b}(\Lambda^{!}))/\tau\) carries a
triangulated structure, and there is a triangulated equivalence
\[
\mathsf{D}_{\mathrm{sg}}\bigl(\mathcal{PLC}^{b}(\Lambda^{!})\bigr)
\big/ \tau
\;\xrightarrow{\ \sim\ }\;
\mathsf{D}_{\mathrm{sg}}\bigl(\Lambda\textup{-mod}\bigr).
\]
\end{enumerate}
\end{theorem}

\begin{proof}
Under the equivalence \(K'\), the autoequivalence \(\tau\) of
\(\mathcal{PLC}^{b}(\Lambda^{!})\) corresponds to the grading shift
\(\langle1\rangle\) of \(\Lambda\textup{-gmod}\), and \(K'\) extends termwise to
dg enhancements compatibly with \(\tau\) and \(\langle1\rangle\), by the
argument of Remark~\ref{rem:shift-dg}. Hence (1) and (2) follow from
Theorem~\ref{thm:forget-hull}.

For (3), it suffices, by the same identification, to show that the forgetful
functor
\(F\colon\mathsf{D}_{\mathrm{sg}}(\Lambda\textup{-gmod})\to\mathsf{D}_{\mathrm{sg}}(\Lambda\textup{-mod})\)
is a \(\langle\langle1\rangle\rangle\)-Galois covering. By Theorem~\ref{thm:forget-hull}(2),
\(F\) is a \(\langle\langle1\rangle\rangle\)-Galois precovering, so by
Proposition~\ref{prop:orbit-univ} it induces a fully faithful functor
\(\mathsf{D}_{\mathrm{sg}}(\Lambda\textup{-gmod})/\langle1\rangle\to\mathsf{D}_{\mathrm{sg}}(\Lambda\textup{-mod})\),
and it remains to show that \(F\) is dense. Every object of
\(\mathsf{D}_{\mathrm{sg}}(\Lambda\textup{-mod})\) is isomorphic to \(M[n]\) for
some finite-dimensional module \(M\) and some \(n\in\mathbb{Z}\), and
\(M\cong\Omega^2M[2]\) in \(\mathsf{D}_{\mathrm{sg}}(\Lambda\textup{-mod})\). Since
\(\Lambda\) is monomial, the second syzygy \(\Omega^2M\) is a finite direct sum
of left ideals \(\Lambda p\), with \(p\) a path~\cite{41}. Each \(\Lambda p\) is
a graded submodule of \(\Lambda\), since \(p\) is homogeneous; hence \(M[n]\)
lies in the essential image of \(F\). Thus the orbit category is equivalent,
as an additive category, to \(\mathsf{D}_{\mathrm{sg}}(\Lambda\textup{-mod})\),
and we transport the triangulated structure along this equivalence.
\end{proof}
\section*{Acknowledgements}

The author thanks the anonymous referee for a careful reading of the
manuscript and for valuable comments and suggestions, which have significantly
improved the presentation of this paper.

This work was completed in 2021; its preparation for publication was delayed
due to personal circumstances and other commitments.

The author dedicates this work to the memory of I.\,M.~Gel'fand, whose
profound contributions have had a lasting influence on modern mathematics.
The author first encountered Gel'fand's work during graduate studies in
Morocco through functional analysis, in particular Banach algebras, and later
through homological algebra via the monograph \emph{Methods of Homological
Algebra} by S.\,I.~Gel'fand and Yu.\,I.~Manin.

\bigskip

\bigskip

\noindent
\textit{E-mail address:} \texttt{alesm.bouhada@gmail.com}


\begin{thebibliography}{99}

\bibitem{1}
C.~Amiot and S.~Oppermann,
\emph{Cluster equivalence and graded derived equivalence},
Doc. Math. \textbf{19} (2014), 1155--1206.

\bibitem{2}
R.~Bautista and S.~Liu,
\emph{Covering theory for linear categories with application to derived categories},
J. Algebra \textbf{406} (2014), 173--225.

\bibitem{3}
R.~Bautista and S.~Liu,
\emph{The bounded derived categories of an algebra with radical squared zero},
J. Algebra \textbf{482} (2017), 303--345.

\bibitem{4}
R.~Bautista, S.~Liu, and C.~Paquette,
\emph{Representation theory of strongly locally finite quivers},
Proc. London Math. Soc. \textbf{106} (2013), 97--162.

\bibitem{5}
A.~A.~Beilinson,
\emph{Coherent sheaves on \(\mathbb{P}^{n}\) and problems of linear algebra},
Funct. Anal. Appl. \textbf{12} (1978), 214--216.

\bibitem{6}
A.~Beilinson, V.~Ginzburg, and W.~Soergel,
\emph{Koszul duality patterns in representation theory},
J. Amer. Math. Soc. \textbf{9} (1996), 473--527.

\bibitem{7}
D.~J.~Benson,
\emph{Representations and Cohomology I},
Cambridge Stud. Adv. Math., vol.~30, Cambridge Univ. Press, Cambridge, 1991.

\bibitem{8}
I.~N.~Bernstein, I.~M.~Gel'fand, and S.~I.~Gel'fand,
\emph{Algebraic bundles over \(\mathbb{P}^n\) and problems of linear algebra},
Funct. Anal. Appl. \textbf{12} (1978), 212--214.

\bibitem{9}
A.~I.~Bondal and M.~M.~Kapranov,
\emph{Enhanced triangulated categories},
Math. USSR-Sb. \textbf{70} (1991), 93--107.

\bibitem{10}
A.~M.~Bouhada,
\emph{Koszul duality for quadratic monomial algebras},
preprint (2026), available at \url{https://arxiv.org/abs/2604.20177}.

\bibitem{11}
A.~M.~Bouhada, M.~Huang, and S.~Liu,
\emph{Koszul duality for non-graded derived categories},
preprint (2019), available at
\url{https://arxiv.org/abs/1908.06153}.

\bibitem{12}
R.-O.~Buchweitz,
\emph{Maximal Cohen--Macaulay modules and Tate cohomology},
Math. Surveys Monogr., vol.~262, Amer. Math. Soc., Providence, RI, 2021.

\bibitem{13}
T.~Bühler,
\emph{Exact categories},
Expo. Math. \textbf{28} (2010), 1--69.

\bibitem{14}
X.-W.~Chen,
\emph{Algebras with radical square zero are either self-injective or CM-free},
Proc. Amer. Math. Soc. \textbf{140} (2012), 93--98.

\bibitem{15}
V.~Drinfeld,
\emph{DG quotients of DG categories},
J. Algebra \textbf{272} (2004), 643--691.

\bibitem{16}
S.~Eilenberg,
\emph{Homological dimension and syzygies},
Ann. of Math. \textbf{64} (1956), 328--336.

\bibitem{17}
D.~Eisenbud, G.~Fløystad, and F.-O.~Schreyer,
\emph{Sheaf cohomology and free resolutions over exterior algebras},
Trans. Amer. Math. Soc. \textbf{355} (2003), 4397--4426.

\bibitem{18}
D.~Eisenbud and F.-O.~Schreyer,
\emph{Resultants and Chow forms via exterior syzygies},
J. Amer. Math. Soc. \textbf{16} (2003), 537--579.

\bibitem{19}
D.~Eisenbud and F.-O.~Schreyer,
\emph{Betti numbers of graded modules and cohomology of vector bundles},
J. Amer. Math. Soc. \textbf{22} (2009), 859--888.

\bibitem{20}
P.~Gabriel,
\emph{Des catégories abéliennes},
Bull. Soc. Math. France \textbf{90} (1962), 323--448.

\bibitem{21}
P.~Gabriel,
\emph{Unzerlegbare Darstellungen I},
Manuscripta Math. \textbf{6} (1972), 71--103.

\bibitem{22}
P.~Gabriel and M.~Zisman,
\emph{Calculus of Fractions and Homotopy Theory},
Ergeb. Math. Grenzgeb., vol.~35, Springer, Berlin, 1967.

\bibitem{23}
R.~Gordon and E.~L.~Green,
\emph{Representation theory of graded Artin algebras},
J. Algebra \textbf{76} (1982), 138--152.

\bibitem{24}
D.~Happel,
\emph{Triangulated Categories in the Representation Theory of Finite-Dimensional Algebras},
London Math. Soc. Lecture Note Ser., vol.~119, Cambridge Univ. Press, Cambridge, 1988.

\bibitem{25}
M.~Kapranov,
\emph{On the derived categories of coherent sheaves on some homogeneous spaces},
Invent. Math. \textbf{92} (1988), 479--508.

\bibitem{26}
M.~Kashiwara and P.~Schapira,
\emph{Categories and Sheaves},
Grundlehren der Mathematischen Wissenschaften, vol.~332,
Springer, Berlin--Heidelberg, 2006.

\bibitem{27}
B.~Keller,
\emph{Derived categories and their uses},
in: Handbook of Algebra, vol.~1, North-Holland, Amsterdam, 1996, 671--701.

\bibitem{28}
B.~Keller,
\emph{On triangulated orbit categories},
Doc. Math. \textbf{10} (2005), 551--581.

\bibitem{29}
B.~Keller,
\emph{On differential graded categories},
in: Proceedings of the International Congress of Mathematicians, vol.~II,
Eur. Math. Soc., Z\"urich, 2006, 151--190.

\bibitem{30}
B.~Keller, D.~Murfet, and M.~Van den Bergh,
\emph{On two examples by Iyama and Yoshino},
Compos. Math. \textbf{147} (2011), 591--612.

\bibitem{31}
M.~Kontsevich,
\emph{Homological algebra of mirror symmetry},
in: Proceedings of the International Congress of Mathematicians (Z\"urich, 1994),
vol.~1, Birkh\"auser, Basel, 1995, 120--139.

\bibitem{32}
J.-L.~Koszul,
\emph{Homologie et cohomologie des algèbres de Lie},
Bull. Soc. Math. France \textbf{78} (1950), 65--127.

\bibitem{33}
R.~Martínez-Villa and M.~Saorín,
\emph{Koszul equivalences and dualities},
Pacific J. Math. \textbf{214} (2004), 359--378.

\bibitem{34}
V.~Mazorchuk, S.~Ovsienko, and C.~Stroppel,
\emph{Quadratic duals, Koszul dual functors, and applications},
Trans. Amer. Math. Soc. \textbf{361} (2009), 1129--1172.

\bibitem{35}
J.~Miyachi,
\emph{Localization of triangulated categories and derived categories},
J. Algebra \textbf{141} (1991), 463--483.

\bibitem{36}
D.~Orlov,
\emph{Derived categories of coherent sheaves and triangulated categories of singularities},
in: Algebra, Arithmetic, and Geometry, vol.~II, Progr. Math., vol.~270,
Birkh\"auser, Boston, 2009, 503--531.

\bibitem{37}
A.~Polishchuk,
\emph{Noncommutative proj and coherent algebras},
Math. Res. Lett. \textbf{12} (2005), 63--74.

\bibitem{38}
S.~Priddy,
\emph{Koszul resolutions},
Trans. Amer. Math. Soc. \textbf{152} (1970), 39--60.

\bibitem{39}
S.~P.~Smith,
\emph{Category equivalences involving graded modules over path algebras of quivers},
Adv. Math. \textbf{230} (2012), 1780--1810.

\bibitem{Verdier}
J.-L.~Verdier,
\emph{Des cat\'egories d\'eriv\'ees des cat\'egories ab\'eliennes},
Ast\'erisque \textbf{239} (1996).

\bibitem{40}
K.~Yamaura,
\emph{Realizing stable categories as derived categories},
Adv. Math. \textbf{248} (2013), 784--819.

\bibitem{41}
B.~Zimmermann Huisgen,
\emph{Predicting syzygies over monomial relations algebras},
Manuscripta Math. \textbf{70} (1991), 157--182.

\end{thebibliography}
\end{document}